\documentclass[12pt]{amsart}
\usepackage{amsmath,amsfonts,amssymb,amscd,amsthm,amsbsy,epsf}
\newtheorem{thm}{Theorem}[section]

\newtheorem{lem}[thm]{Lemma}
\newtheorem{cor}[thm]{Corollary}
\newtheorem{prop}[thm]{Proposition}
\theoremstyle{definition} 		
\newtheorem{defn}[thm]{Definition}
\newtheorem{remark}[thm]{Remark}
\newtheorem*{note}{Notation}

\def\nat{{\mathbb N}}
\def\zat{{\mathbb Z}}
\def\A{{\mathcal A}}

\def\F{{\mathcal F}}
\def\G{{\mathcal G}}

\def\R{{\mathcal R}}
\def\U{{\mathcal U}}

\def\bs{{\mathbf s}}

\def\bu{{\mathbf u}}
\def\bw{{\mathbf w}}

\def\Tau{{\mathfrak T}}

\begin{document}
\title{Ramsey theory for words representing rationals}
\author{Vassiliki  Farmaki and Andreas Koutsogiannis }

\begin{abstract} Ramsey theory for words over a finite alphabet was unified in the work of Carlson and Furstenberg-Katznelson. Carlson, in the same work, outlined a method to extend the theory for words over an infinite alphabet, but subject to a fixed dominating principle, proving in particular an Ellentuck version, and a corresponding Ramsey theorem for $k=1.$ In the present work we develop in a systematic way a Ramsey theory for words (in fact for $\omega$-$\mathbb{Z}^\ast$-located words) over a doubly infinite alphabet extending Carlson's approach (to countable ordinals and Schreier-type families), and we apply this theory, exploiting the Budak-I\c{s}ik-Pym representation, to obtain a partition theory for the set of rational numbers. Furthermore, we show that the theory can be used to obtain partition theorems for arbitrary semigroups, stronger than known ones.
\end{abstract}

\keywords{Partition theorems, $\omega$-$\mathbb{Z}^\ast$-located words, rational numbers, Schreier families}
\subjclass{Primary 05C55; Secondary 05A18, 05A05}

\maketitle
\baselineskip=18pt
\pagestyle{plain}


\section*{Introduction}

The concept of a word over a finite alphabet was introduced in Ramsey theory by Hales-Jewett \cite{HaJ}, for the purpose of providing a purely combinatorial proof of van der Waerden's theorem \cite{vdW} on the existence of arbitrarily long arithmetic progressions in one of the cells of any partition of the positive integers. Subsequently, words over a finite alphabet, in the work of Carlson \cite{C} and Furstenberg-Katznelson \cite{FuK}, proved an essential tool for the unification of the two branches of Ramsey theory, the one involving Ramsey's classical theorem \cite{R}, and its extensions by Hindman \cite{H} and Milliken \cite{M}-Taylor \cite{T}, the other the van der Waerden and the Hales-Jewett theorems, just mentioned. These tools were extended and strengthened, with the systematic introduction of Schreier sets, in \cite{FN1}, \cite{FN2}.

The concept of a located word over a finite alphabet $\Sigma$ introduced formally by Bergelson-Blass-Hindman in \cite{BBH} as a function from a finite subset of the set $\nat$ of natural numbers, the support and location of the word, into the alphabet $\Sigma$. They established a partition theorem for located words over a finite alphabet and also a Ramsey type and a Nash-Williams type partition theorem \cite{NW} for located words over a finite alphabet.

In all these results, the alphabet $\Sigma$ under consideration is always assumed to be finite. It is clear that these combinatorial results do not generally hold if $\Sigma$ is assumed to be infinite, since e.g. it is generally impossible to find a monochromatic infinite arithmetic progression for every finite coloring of the set of natural numbers. In the work of Carlson ([C], Section 10) it was first realized that it is possible to relax the strict finiteness condition for the alphabet and to consider words ($\omega$-words) over an infinite alphabet $\Sigma=\{\alpha_1, \alpha_2, \ldots \}$ dominated by a sequence $\vec{k}=(k_n)_{n\in\nat}$ of positive integers. So, an $\omega$-word over $\Sigma$ dominated by $\vec{k}$ is a word $w=w_{1}\ldots w_{l}$ over $\Sigma$ such that in addition $ w_{i}\in \{\alpha_1, \ldots, \alpha_{k_{i}} \}$ for every $1\leq i \leq l$. Similarly a located word $w=w_{n_1}\ldots w_{n_l}$ over $\Sigma$ is an $\omega$-located word over $\Sigma$ dominated by $\vec{k},$ according to [F4], if $w_{n_i}\in \{\alpha_1, \ldots, \alpha_{k_{n_i}}\}$ for every $1\leq i \leq l$. Thus words and located words are $\omega$-words and $\omega$-located words, in case the dominating sequence $(k_n)_{n\in\nat}$ is a constant sequence.

In the present work we introduce the notion of $\omega$-$\zat^\ast$-located words over an alphabet $\Sigma=\{\alpha_n: \;n\in \zat^\ast\}$ ($\zat^\ast=\zat\setminus\{0\}$) dominated by a two-sided sequence $\vec{k}=(k_n)_{n\in \zat^\ast}$ of natural numbers to be a located word $w=w_{n_1}\ldots w_{n_l}$ over $\Sigma$ such that in addition $\{n_1<\ldots<n_l\}\subseteq \zat^\ast$ and $w_{n_i}\in \{\alpha_1, \ldots, \alpha_{k_{n_i}}\}$ if $n_i\in \nat,$ $w_{n_i}\in \{\alpha_{-k_{n_i}}, \ldots, \alpha_{-1}\}$ if $-n_i\in \nat.$ The inspiration for this notion came from the representation of rational numbers introduced by T. Budak, N. I\c{s}ik and J. Pym in [BIP], according to which a rational number $q$ has a unique representation as $$\sum^{\infty}_{s=1}q_{-s}\frac{(-1)^{s}}{(s+1)!}\;+\;\sum^{\infty}_{r=1}q_{r}(-1)^{r+1}r!$$ where $(q_n)_{n \in \mathbb{Z}^\ast}\subseteq \nat\cup\{0\}$ with $\;0\leq q_{-s}\leq s$ for every $s>0$, $ 0\leq q_r\leq r$ for every $r> 0$ and $q_{-s}=q_r=0$ for all but finite many $r,s$. So, the set of non-zero rational numbers can be identified with the set of all the $\omega$-$\zat^\ast$-located words over the alphabet $\Sigma=\{\alpha_n: \;n\in \zat^\ast\},$ where $\alpha_{-n}=\alpha_n=n$ for $n\in \nat$ dominated by the sequence $(k_n)_{n\in \zat^\ast}$ where $k_{-n}=k_n=n$ for $n\in\nat$.

It turns out that the whole of infinitary Ramsey theory can be obtained for $\omega$-$\zat^\ast$-located words. We thus obtain:

(a) Partition theorems for variable and constant $\omega$-$\zat^\ast$-located words over an alphabet $\Sigma=\{\alpha_n: \;n\in \zat^\ast\}$ dominated by a two-sided sequence $\vec{k}=(k_n)_{n\in \zat^\ast}$ of natural numbers, in Theorems~\ref{thm:block-Ramsey02} and ~\ref{thm:block-Ramsey2} (in stronger form), providing proper extensions of Carlson's partition theorem ([C], Theorem 15) for $\omega$-(located) words (see Corollary~\ref{thm:block-Ramsey1}) and consequently Bergelson-Blass-Hindman's partition theorem (Theorem 4.1 in \cite{BBH}) for located words over a finite alphabet and Carlson's partition theorem for words over a finite alphabet (Lemma 5.9 in \cite{C}).

(b) Extended Ramsey type partition theorems for each countable ordinal $\xi$ for variable and constant $\omega$-$\zat^\ast$-located words (in Theorem~\ref{thm:block-Ramsey}), involving the $\xi$-Schreier sequences of $\omega$-$\zat^\ast$-located words (Definition~\ref{recursivethinblock}), which imply Ramsey type partition theorems for each countable ordinal $\xi$ for variable and constant $\omega$-located words proved in [F4] (see Corollary~\ref{cor:nat}), extending Carlson's partition theorem corresponding to $\xi=1$ (Corollary~\ref{thm:block-Ramsey1}), and consequently an ordinal extension of Bergelson-Blass-Hindman's Ramsey type partition theorem for located words over a finite alphabet (Theorem 5.1 in  \cite{BBH}) corresponding to the case of finite ordinals and the block-Ramsey partition theorem for every countable ordinal, proved in \cite{FN1}.

(c) A strengthening of the extended Ramsey type partition theorem (Theorem~\ref{thm:block-Ramsey}), in case the partition family $\F$ of the finite orderly sequences of variable $\omega$-$\zat^\ast$-located words is a tree, providing a criterion, in terms of a Cantor-Bendixson type index of $\F$ to decide whether the $\xi$-homogeneous family falls in $\F$ or in its complement in Theorem~\ref{block-NashWilliams2}. Theorem~\ref{block-NashWilliams2} can be considered as a strengthened Nash-Williams type partition theorem for variable $\omega$-$\mathbb{Z}^\ast$-located words (see Theorem~\ref{cor:blockNW}), strengthening and extending the Nash-Williams type partition theorem for infinite sequences of variable $\omega$-(located) words proved by Carlson (that follows from Theorem 15 in [C]) and also of variable located words over a finite alphabet proved in  \cite{BBH} (Theorem 6.1).

As consequences of this Ramsey Theory, via the representation of rational numbers given by Budak-I\c{s}ik-Pym in [BIP] (Theorem 4.2), we can obtain analogous partition theorems for the rational numbers. So, we present, in Theorem~\ref{thm:block-Ramsey05}, a partition theorem for the set $\mathbb{Q}$ of rational numbers, which can be considered as a strengthened van der Waerden theorem for $\mathbb{Q},$ Ramsey type partition theorems for the rational numbers for each countable order $\xi,$ in Theorem~\ref{thm:block-Ramsey005}, defining the $\xi$-Schreier sequences of rational numbers for every countable ordinal $\xi,$ and finally a Nash-Williams type partition theorem for the infinite sequences of rational numbers in Theorem~\ref{thm:block-Ramsey0005}. Of course, analogous partition theorems can be obtained for semigroups which can be represented in the same way as $\omega$-$\mathbb{Z}^\ast$-located words.

Also, fixing sequences in a semigroup we can get partition theorems for commutative or non-commutative semigroups (in Theorems~\ref{thm: 011} and ~\ref{thm:block-Ramsey06} respectively), extending the analogous results of Hindman and Strauss in \cite{HS} (Theorems 14.12, 14.15), as well as multidimentional partition theorems corresponding to each countable order (Theorem~\ref{thm: xilocated}), extending the partition theorem of Beiglb$\ddot{o}$ck in [Be] for commutative semigroups corresponding to the case of finite ordinals (see Corollaries~\ref{thm:block-Ramsey006} and ~\ref{thm: 1011}), and finally Nash-Williams type partition theorems for infinite sequences in a commutative and in a non-commutative semigroup (in Theorems~\ref{thm: 11011} and ~\ref{cor:ncentralNW}).

The essentially stronger nature of this Ramsey theory for $\omega$-$\zat^\ast$-words developed in this paper makes it reasonable to expect that will find substantial applications in Ramsey ergodic theory, Analysis, Logic and in various other branches of mathematics.

We will use the following notation.
\begin{note}
Let $\nat=\{1,2,\ldots\}$ be the set of natural numbers, $\mathbb{Z}=\{\ldots,-2,-1,0,1,2,\ldots\}$ be the set of integer numbers,  $\mathbb{Z}^{-}=\{n\in\mathbb{Z} : n<0\}$ and $\zat^\ast=\zat\setminus\{0\}.$ For a non-empty set $X$ we denote by $[X]^{<\omega}$ the set of all the finite subsets of $X$ and by $[X]_{>0}^{<\omega}$ the set of all the non-empty finite subsets of $X$.
\end{note}

\section{A Partition Theorem for $\omega$-$\mathbb{Z}^\ast$-located words}

The purpose of this section is to prove partition theorems for $\omega$-$\zat^\ast$-located words over an alphabet $\Sigma=\{\alpha_n: \;n\in \zat^\ast\}$ dominated by a two-sided sequence $\vec{k}=(k_n)_{n\in \zat^\ast}$ of natural numbers (Theorems~\ref{thm:block-Ramsey02} and ~\ref{thm:block-Ramsey2} in stronger form) using methods introduced in [C] and [BBH]. These theorems extends fundamental results for words supported on $\nat$ as the partition theorems for $\omega$-(located) words and words proved in [C] (Theorem 15, Lemma 5.9) and consequently the partition theorem for located words proved in [BBH] (Theorem 4.1).

As consequences of Theorem~\ref{thm:block-Ramsey02} we get a strengthened Hales-Jewett theorem involving the $\omega$-$\zat^\ast$-located words in Corollary~\ref{thm:HJ} and also strong partition theorems for the set of rational numbers and more generally for an arbitrary semigroup in Theorem~\ref{thm:block-Ramsey05} and Theorems~\ref{thm:block-Ramsey06}, ~\ref{thm: 011} respectively.

An \textit{$\omega$-$\mathbb{Z}^\ast$-located word} over the alphabet $\Sigma=\{\alpha_n: \;n\in \zat^\ast\}$ dominated by $\vec{k}=(k_n)_{n\in \zat^\ast},$ with $k_n\in \nat$ for every $n\in \zat^\ast,$ is a function $w$ from a non-empty,  finite subset $F$ of $\mathbb{Z}^\ast$ into the alphabet $\Sigma$ such that $w(n)=w_n\in \{\alpha_1, \ldots, \alpha_{k_n}\}$ for every $n\in F\cap \nat$ and $w_n\in \{\alpha_{-k_{n}}, \ldots,\alpha_{-1}\}$ for every $n\in F\cap \mathbb{Z}^{-}$. So, the set
$\widetilde{L}(\Sigma, \vec{k})$ of all (constant) $\omega$-$\mathbb{Z}^\ast$-located words over $\Sigma$ dominated by $\vec{k}$ is:
\begin{center}$\widetilde{L}(\Sigma, \vec{k})= \{w=w_{n_1}\ldots w_{n_l} : l\in\nat, n_1<\ldots<n_l\in \mathbb{Z}^\ast$ and $ w_{n_i}\in
\{\alpha_1, \ldots, \alpha_{k_{n_i}}\}$ if $\;\;\;$\\
 $n_i> 0$, $ w_{n_i}\in \{\alpha_{-k_{n_i}}, \ldots,\alpha_{-1}\}$ if $n_i< 0$
for every $1\leq i \leq l \}.$
\end{center}
Let $\upsilon \notin \Sigma$ be an entity which is called a \textit{variable}. The set of \textit{variable  $\omega$-$\mathbb{Z}^\ast$-located words} over $\Sigma$ dominated by $\vec{k}$ is:
\begin{center}$\widetilde{L}(\Sigma, \vec{k} ; \upsilon) = \{w=w_{n_1}\ldots w_{n_l} : l\in\nat, n_1<\ldots<n_l\in \mathbb{Z}^\ast,$ $ w_{n_i}\in \{\upsilon, \alpha_1, \ldots, \alpha_{k_{n_i}} \}$ if $\;\;\;$\\
 $\;\;\;\;\;\;\;\;\;\;\;\;\;\;\;\;\;\;n_i> 0$, $ w_{n_i}\in \{\upsilon,\alpha_{-k_{n_i}}, \ldots,\alpha_{-1}\}$ if $n_i< 0$ for all $1\leq i \leq l $ and there\\ exists $1\leq i \leq l$  with $w_{n_i}=\upsilon \}.\;\;\;\;\;\;\;\;\;\;\;\;\;\;\;\;\;\; \;\;\;\;\;\;\;\;\;\;\;\;\;\;\;\;\;\; \;\;\;\;$\end{center}
We set $\widetilde{L}(\Sigma\cup\{\upsilon\}, \vec{k})=\widetilde{L}(\Sigma, \vec{k})\cup \widetilde{L}(\Sigma, \vec{k} ; \upsilon)$. For $w=w_{n_1}\ldots w_{n_l}\in \widetilde{L}(\Sigma\cup\{\upsilon\}, \vec{k})$ the set $dom(w)=\{n_1,\ldots ,n_l\}$ is the \textit{domain} of $w$. Let $dom^-(w)=\{n\in dom(w):\;n<0\}$ and $dom^+(w)=\{n\in dom(w):\;n>0\}$.
We define the sets
\begin{center}
$\;\;\widetilde{L}_{0}(\Sigma,\vec{k};\upsilon)=\{w\in \widetilde{L}(\Sigma,\vec{k};\upsilon): w_{i_1}=\upsilon=w_{i_2}$ for some $i_1\in dom^-(w),i_2\in dom^+(w)\}$,\\
$\widetilde{L}_{0}(\Sigma,\vec{k})=\{w\in \widetilde{L}(\Sigma,\vec{k}): dom^-(w)\neq \emptyset$ and $ dom^+(w)\neq\emptyset \},$ and $\;\;\;\;\;\;\;\;\;\;\;\;\;\;\;\;\;\;\;\;\;\;\;\;\;\;\;\;$\\
$\widetilde{L}_0(\Sigma\cup\{\upsilon\}, \vec{k})=\widetilde{L}_0(\Sigma, \vec{k})\cup \widetilde{L}_0(\Sigma, \vec{k} ; \upsilon).\;\;\;\;\;\;\;\;\;\;\;\;\;\;\;\;\;\;\;\;\;\;\;\;\;\;\;\;\;\; \;\;\;\;\;\;\;\;\;\;\;\;\;\;\;\;\;\;\;\;\;\;\;\;\;\;\;\;\;\;\;\;\;\;\;\;\;\;$\end{center}
For $w=w_{n_1}\ldots w_{n_r},  u=u_{m_1}\ldots u_{m_l}\in \widetilde{L}(\Sigma \cup\{\upsilon\}, \vec{k})$ with $dom(w)\cap dom(u)=\emptyset$ we  define the concatenating word:
\begin{center}
$w \star u =z_{q_1}\ldots z_{q_{r+l}}\in \widetilde{L}(\Sigma \cup\{\upsilon\}, \vec{k}),$
\end{center}
where $\{q_1<\ldots<q_{r+l}\}=dom(w)\cup dom(u),\;z_i=w_i$ if $i \in dom(w)$ and $z_i=u_i$ if $i \in dom(u)$.

We endow the set $\widetilde{L}_0(\Sigma\cup\{\upsilon\}, \vec{k})$ with a relation $<_{\textsl{R}_1}$ defining for $w,u\in \widetilde{L}_0(\Sigma\cup\{\upsilon\}, \vec{k})$
\begin{center}
$w <_{\textsl{R}_1} u \Longleftrightarrow \;dom(u)=A_1\cup A_2$ with $A_1,A_2\neq \emptyset$ such that \\ $\;\;\;\;\;\;\;\;\;\;\;\;\;\;\;\;\;\;\;\;\;\;\;\; \;\;\;\;\max A_1<\min dom(w)\leq \max dom(w)<\min A_2.$
\end{center}
Let $\widetilde{L}^\infty (\Sigma, \vec{k} ; \upsilon) = \{\vec{w} = (w_n)_{n\in\nat} : w_n\in \widetilde{L}_0(\Sigma, \vec{k} ; \upsilon)$
and  $w_n<_{\textsl{R}_1}w_{n+1}$ for every $ n\in\nat\}$.

We define the functions $T_{(p,q)} : \widetilde{L}(\Sigma\cup\{\upsilon\}, \vec{k}) \longrightarrow \widetilde{L}(\Sigma\cup\{\upsilon\}, \vec{k})$ for every $(p,q)\in \nat\times \nat\cup\{(0,0)\}$, setting,  for $w=w_{n_1}\ldots w_{n_l}\in \widetilde{L}(\Sigma\cup\{\upsilon\}, \vec{k})$, $T_{(0,0)}(w)=w$ and, for  $(p,q)\in\nat\times \nat$, $T_{(p,q)}(w)=u_{n_1}\ldots u_{n_l}$, where, for $1\leq i\leq l$
\newline
$u_{n_i}=w_{n_i}$ if $w_{n_i}\in \Sigma$,
\newline
$u_{n_i}=\alpha_p$ if $w_{n_i}=\upsilon,$ $n_i> 0$ and $p\leq k_{n_i},$
\newline
$u_{n_i}=\alpha_{k_{n_i}}$ if $w_{n_i}=\upsilon,$ $n_i> 0$ and $p> k_{n_i},$
\newline
$u_{n_i}=\alpha_{-q}$ if $w_{n_i}=\upsilon,$ $n_i< 0$ and $q\leq k_{n_i},$ and
\newline
$u_{n_i}=\alpha_{-k_{n_i}}$ if $w_{n_i}=\upsilon,$ $n_i< 0$ and $q> k_{n_i}$.
\newline
We remark that for every $(p,q)\in \nat\times \nat\cup\{(0,0)\}$ we have  $dom(T_{(p,q)} (w))=dom(w)$ for $w\in $ $\widetilde{L}(\Sigma\cup\{\upsilon\}, \vec{k}),$ $T_{(p,q)} (w)=w$ for $w\in \widetilde{L}(\Sigma, \vec{k})$ and $T_{(p,q)} (w\star u)=T_{(p,q)} (w)\star T_{(p,q)} (u)$ for every $w,u\in \widetilde{L}(\Sigma\cup\{\upsilon\}, \vec{k})$ with $dom(w)\cap dom(u)=\emptyset$. Also,
$T_{(p,q)} (\widetilde{L}(\Sigma\cup\{\upsilon\}, \vec{k}))\subseteq\widetilde{L}(\Sigma, \vec{k})$ for every $(p,q)\in \nat\times \nat$.

With the previous terminology we can state the following partition theorem for $\omega$-$\mathbb{Z}^\ast$-located words. The theorem is a natural generalization of a corresponding result for a simply infinite alphabet (Corollary~\ref{thm:block-Ramsey1}, below), which follows easily from Carlson's Theorem 15 in [C], Section 10, and included here for the purpose of making the paper self-contained (cf. also Theorem 1.1 in [F4]). The fact that Carlson's results are for $\omega$-words, while our results are phrased for $\omega$-$\mathbb{Z}^\ast$-located words does not make much difference for the essential arguments involved.

\begin{thm}
[\textsf{Partition theorem for $\omega$-$\mathbb{Z}^\ast$-located words}]
\label{thm:block-Ramsey02}
Let $\Sigma=\{\alpha_n:\;n\in \mathbb{Z}^\ast\}$ be an alphabet, $\upsilon \notin \Sigma$ a variable and $\vec{k}=(k_n)_{n\in\mathbb{Z}^\ast}\subseteq \nat$. If $\widetilde{L}(\Sigma, \vec{k} ; \upsilon)=A_1\cup\ldots\cup A_r$ and $\widetilde{L}(\Sigma, \vec{k})=C_1\cup\ldots\cup C_s,$ where $r,s \in \nat,$ then for each sequence  $(F_n)_{n\in\nat}$ of non-empty, finite subsets of $\nat\times \nat$ and $w \in \widetilde{L}_0(\Sigma, \vec{k} ; \upsilon)$ there exists a sequence $(w_{n})_{n\in\nat}\in \widetilde{L}^\infty(\Sigma, \vec{k} ; \upsilon)$ with $w<_{\textsl{R}_1}w_1$ and $1\leq i_{0} \leq r$, $1\leq j_{0} \leq s$ such that
\begin{center}
$T_{(p_1,q_1)}(w_{n_1})\star \ldots \star T_{(p_\lambda,q_\lambda)}(w_{n_\lambda})\in A_{i_{0}}$,
\end{center}
for every $\lambda\in\nat$, $n_1<\ldots<n_\lambda\in\nat$, $(p_i,q_i) \in F_{n_i}\cup\{(0,0)\}$ for every $1\leq i \leq \lambda$ with $(0,0) \in \{ (p_1,q_1),\ldots,(p_\lambda,q_\lambda)\}$; and
\begin{center}
$T_{(p_1,q_1)}(w_{n_1})\star \ldots \star T_{(p_\lambda,q_\lambda)}(w_{n_\lambda})\in C_{j_{0}}$,
\end{center}
for every $\lambda\in\nat$, $n_1<\ldots<n_\lambda\in\nat$ and $(p_i,q_i) \in F_{n_i}$ for every $1\leq i \leq \lambda$.
\end{thm}

In the proof of Theorem~\ref{thm:block-Ramsey02} we will apply some results of the theory of left compact semigroups, which we mention below.

\subsection*{Left compact semigroups}

A \textit{left compact semigroup} is a non-empty semigroup $(X, +),$ endowed with a topology $\Tau$ such that $(X,\Tau)$ is a compact Hausdorff space and the functions $f_y : X\rightarrow X$ with $f_y(x)=x+y$ for $x\in X$ are continuous for every $y\in X$.

An element $x$ of a semigroup $(X,+)$ such that $x+x=x$ is an \textit{idempotent} of $(X, +).$ According to a fundamental result due to Ellis, every non-empty, left compact semigroup contains an idempotent. On the set of all the idempotents of $(X, +)$ is defined a partial order $\leq$ by the rule
\begin{center}
$x_1 \leq x_2 \Longleftrightarrow x_1+x_2=x_2+x_1=x_1$.
\end{center}
An idempotent $x$ of $(X, +)$ is \textit{minimal for $X$} if every idempotent $x_1$ of $X$ satisfying the relation $x_1\leq x $ is equal to $x$. According to \cite{FuK}, a left compact semigroup $(X, +)$ contains an idempotent $x$ minimal for $X$, moreover, for every idempotent $x_1$ of $X$ there exists an idempotent $x$ of $X$ which is minimal for $X$ and $x\leq x_1$. Every two-sided ideal of a semigroup $(X,+)$ (a subset $I$ of $X$ is called \textit{two-sided ideal} of $(X,+)$ if $X+I\subseteq I$ and $I+X\subseteq I$) contains all the minimal for $X$ idempotents of $X$.

\subsection*{Ultrafilters}

Let $X$ be a non-empty set. An \textit{ultrafilter} on the set $X$ is a zero-one finite additive measure $\mu$ defined on all the subsets of $X$. The set of all ultrafilters on the set $X$ is denoted by $\beta X$. So, $\mu\in\beta X$ if and only if
\begin{itemize}
\item[{(i)}] $\mu(A)\in\{0,1\}$ for every $A\subseteq X$, $\mu(X)=1$, and
\item[{(ii)}] $\mu(A\cup B)=\mu(A)+\mu(B)$ for every $A,B\subseteq X$ with $A\cap B=\emptyset$.
\end{itemize}
For $\mu\in\beta X,$ it is easy to see that $\mu(A\cap B)=1,$ if $\mu(A)=1$ and $\mu(B)=1$. For every $x\in X$ is defined the \textit{principal ultrafilter} $\mu_x$ on $X$ which corresponds a set $A\subseteq X$ to $\mu_x(A)=1$ if $x\in A$ and  $\mu_x(A)=0$ if $x\notin A$.  So, $\mu$ is a non-principal ultrafilter on $X$ if and only if  $\mu(A)=0$ for every finite subset $A$ of $X$.

The set $\beta X$ becomes a compact Hausdorff space if it be endowed with the topology $\Tau$ which has basis the family $\{A^* :  A\subseteq X \}$, where $A^*=\{\mu\in\beta X : \mu(A)=1 \}$. It is easy to see that $(A\cap B)^*=A^*\cap B^*$, $(A\cup B)^*=A^*\cup B^*$ and $(X\setminus A)^*=\beta X\setminus A^*$ for every $A,B\subseteq X$. We always consider the set $\beta X$ endowed with the topology $\Tau$. Hence, for a function $T : X\rightarrow Y$ is defined the function
\begin{center}
$\beta T : \beta X\rightarrow \beta Y$  with  $\beta T(\mu)(B)=\mu(T^{-1}(B))$ for $\mu\in\beta X$ and $B\subseteq Y$,
\end{center}
which is always continuous.

If $(X,+)$ is a semigroup, then a binary operation $+$ is defined on $\beta X$ corresponding to every $\mu_1, \mu_2\in \beta X$ the ultrafilter $\mu_1 + \mu_2\in \beta X$ given by
\begin{center}
$(\mu_1 + \mu_2)(A)= \mu_1(\{x\in X : \mu_2(\{y\in X : x+y\in A\})=1\})$ for every $A\subseteq X$.
\end{center}
With this operation the set $\beta X$ becomes a semigroup and for every $\mu\in\beta X$ the function $f_\mu : \beta X\rightarrow \beta X$ with $f_\mu(\mu_1)=\mu_1+\mu$ is continuous. Hence, if $(X,+)$ is a semigroup, then $(\beta X, +)$ becomes a left compact semigroup.

\begin{proof}[Proof of Theorem~\ref{thm:block-Ramsey02}] Let $X=\widetilde{L}(\Sigma\cup\{\upsilon\},\vec{k}).$
We endow $X$ with an operation defining for $w=w_{n_1}\ldots w_{n_r},  u=u_{m_1}\ldots u_{m_l}\in X$
\begin{center}
$w+u=z_{q_1}\ldots z_{q_s}\in X$,
\end{center}
where $\{q_1,\ldots,q_s\}=dom(w)\cup dom(u)$, $q_1<\ldots<q_s$ and, for $1\leq i\leq s$, $z_{q_i}=w_{q_i}$ if $q_i\notin dom(u)$, $z_{q_i}=u_{q_i}$ if $q_i\notin dom(w),$ $z_{q_i}=\upsilon$ if either $w_{q_i}=\upsilon$ or $u_{q_i}=\upsilon$,
 $z_{q_i}=\alpha_{\max\{\mu, \nu\}}$ if $ q_i\in dom^+(w)\cap dom^+(u)$ and $w_{q_i}=\alpha_{\mu},\;u_{q_i}=\alpha_{\nu}$, and $z_{q_i}=\alpha_{-\max\{\mu, \nu\}}$ if $q_i\in dom^-(w)\cap dom^-(u)$ and $w_{q_i}=\alpha_{-\mu},\;u_{q_i}=\alpha_{-\nu}$.
\newline
Observe that $(X,+)$ is a semigroup, $C=\widetilde{L}(\Sigma,\vec{k})$ is a non-empty subsemigroup of $(X,+)$ and $V_0=\widetilde{L}_0(\Sigma,\vec{k};\upsilon)$ is a two-sided ideal of $(X, +)$. Also, $w+ u=w\star u$ for $w, u\in X$ with $w<_{\textsl{R}_1}u$.

Since $(X, +)$ is a semigroup, $(\beta X, +)$ has the structure of a left compact semigroup. For every $A\subseteq X$ and $w\in X$ we set
\begin{center}
$A_{w}=\{u\in A :\; w<_{\textsl{R}_1}u\}\;\;$ and
$\;\;\theta A=\bigcap \{(A_{w})^* : w\in X \}$,
\end{center}
where $(A_{w})^*=\{\mu\in\beta X : \mu(A_{w})=1 \}$.

\noindent {\bf Claim 1}
 $\theta X, \theta C$ and  $\theta V_0$ are non-empty left compact subsemigroups of $\beta X$ consisted of non-principal ultrafilters on $X$.

Indeed, let $A\in \{X, C, V_0\}$. For every $w\in X$ the set $(A_{w})^*=\beta X\setminus(X\setminus A_{w})^*$ is a compact subset of $\beta X$ , so $\theta A$ is a compact subset of $\beta X$. For every $w\in X$ we have $A_{w}\neq\emptyset$ and consequently $(A_{w})^*\neq\emptyset$. Since  the family
$\{(A_{w})^* :\; w\in X\}$ has the finite intersection property, we have $\theta A\neq\emptyset$. Using
the fact that if $w<_{\textsl{R}_1}y$ and $w+y<_{\textsl{R}_1}z,$ then $w<_{\textsl{R}_1}y+z$, it can be proved that $(\theta A, +)$ is a semigroup. Indeed, for $\mu_1,\mu_2\in \theta A$ and $w\in X$, we have that
\begin{center}$(\mu_1 +\mu_2)(A_{w})=\mu_1(\{u_1\in A_{w} : \mu_2(\{u_2\in A_{w+u_1} : u_1+u_2\in A_{w}\})=1 \})\;\;\;\;\;\;\;\;\;\;\;\;\;\;\;\;\;\;$\\
$=\mu_1(\{u_1\in A_{w} : \mu_2( A_{w+u_1})=1 \})=\mu_1(A_{w})=1.\;\;\;\;\;\;\;$
\end{center}
Hence, $\theta A$ is a non-empty left compact subsemigroup of $\beta X$ and it contains only non-principal ultrafilters on $X$, since $\mu_{w}\notin (A_w)^\ast$ for every $w\in X$.

 $\theta V_0  \subseteq \theta X $ is a two sided ideal of $\theta X$. Indeed, for $\mu_1\in \theta V_0$, $\mu_2\in \theta X$ and $w\in X,$
\begin{center} $(\mu_1+\mu_2)((V_{0})_{w})=\mu_1(\{u_1\in (V_{0})_{w} : \mu_2(\{u_2\in X_{w+u_1} : u_1+u_2\in (V_{0})_{w}\})=1 \})\;\;\;\;\;\;\;\;\;\;\;$\\

$=\mu_1(\{u_1\in (V_{0})_{w} : \mu_2(X_{w+u_1})=1 \})\;\;\;\;\;\;\;\;\;\;\;\;\;\;\;\;\;\;\;\;\;\;\;$\\ $=\mu_1((V_{0})_{w})=1=(\mu_2+\mu_1)((V_{0})_{w}).\;\;\;\;\;\;\;\;\;\;\;\;\;\;\;\;\;\;\;\;\;\;\;$\end{center}

\noindent Let $(p,q)\in \nat\times \nat$ and let the function $\beta T_{(p,q)} : \beta X\rightarrow \beta X$ with
$\beta T_{(p,q)}(\mu)(A)=\mu((T_{(p,q)})^{-1}(A))$ for every $\mu\in\beta X$ and $A\subseteq X$. We note that:

i) $\beta T_{(p,q)}(\mu)=\mu$ for every $\mu\in\theta C$, since $T_{(p,q)}(w)=w$ for every $w\in C$.

ii) $\beta T_{(p,q)}(\theta X)\subseteq \theta C$, since for $\mu\in \theta X$ and $w\in X$ we have \begin{center}
$\beta T_{(p,q)}(\mu)(C_{w})=\mu(\{u\in X_{w} : T_{(p,q)}(u)\in C_{w} \})=\mu(X_{w})=1,\;\;\;\;\;\;\;\;\;\;\;\;\;\;\;\;\;\;\;\;\;\;\;\;\;\;\;\;\;\;\;\;\;\;\;\;$\end{center}

iii) the restriction $\beta T_{(p,q)}|_{\theta X}$ of $\beta T_{(p,q)}$ to $\theta X$ is a homomorphism, since for $\mu_1, \mu_2\in \theta X$\\ \indent \indent $\;$ and $A\subseteq X$
\begin{center}$\beta T_{(p,q)}(\mu_1+\mu_2)(A)= \mu_1(\{u_1\in X : \mu_2(\{u_2\in X_{u_1} : T_{(p,q)}(u_1+u_2)\in A\})=1\})\;\;\;\;\;\;\;\;\;\;\;$\\
$\;\;\;\;\;\;\;\;\;\;\;\;\;\;\;\;\;\;\;\;\;\;\;\;\;\;\;\;\;=\mu_1(\{u_1\in X : \mu_2(\{u_2\in X_{u_1} : T_{(p,q)}(u_1)+T_{(p,q)}(u_2)\in A\})=1\})$\\ $=(\beta T_{(p,q)}(\mu_1)+ \beta T_{(p,q)}(\mu_2))(A).\;\;\;\;\;\;\;\;\;\;\;\;\;\;\;\;\;\;\;\;\;$ \end{center}
According to \cite{FuK}, there exists an idempotent $\mu_1$ in the non-empty, left compact semigroup $\theta C,$ minimal for $\theta C$. Since $\theta V_0$ is a two-sided ideal of the left compact semigroup $\theta X$ and
$\mu_1\in \theta C\subseteq\theta X$ is an idempotent of $\theta X$ there exists an idempotent $\mu\in \theta V_0\subseteq \beta X$ minimal for $\theta X$ with $\mu\leq \mu_1$. Since for each $(p,q)\in \nat\times \nat$ the restriction of $\beta T_{(p,q)}$ to $\theta X$ is an homomorphism, we have that $\beta T_{(p,q)}(\mu)\leq \beta T_{(p,q)}(\mu_1)$ for every $(p,q)\in \nat\times \nat$. But $\beta T_{(p,q)}(\mu_1)= \mu_1$ for every $(p,q)\in \nat\times \nat$, since $\mu_1\in\theta C$, hence, $\beta T_{(p,q)}(\mu)\leq \mu_1$ for every $(p,q)\in \nat\times \nat$. Now, since $\mu_1$ is minimal for $\theta C$ and
$\beta T_{(p,q)}(\mu)\in \theta C$ for every $(p,q)\in \nat\times \nat$, we have $\beta T_{(p,q)}(\mu)=\mu_1$ for every $(p,q)\in \nat\times \nat$. In conclusion, there exist non-principal ultrafilters
$\mu\in \theta V_0 \subseteq\theta X\subseteq \beta X$ and $\mu_1\in\theta C\subseteq \theta X\subseteq \beta X$ such that:

$(1)\;\mu+\mu=\mu$, $\mu_1+\mu_1=\mu_1 $,

$(2)\;\mu_1=\beta T_{(p,q)}(\mu)$ for every $(p,q)\in \nat\times \nat$, and

$(3)\;\mu+\mu_1=\mu_1+\mu=\mu $.\\
\noindent {\bf Claim 2}
For every $w\in V_0$, $B\subseteq V_0$ with $\mu(B)=1$, $D\subseteq C$ with $\mu_1(D)=1$ and $F$ a non-empty, finite subset of $\nat\times \nat$ there exist $w_{0}\in V_0$, $w<_{\textsl{R}_1}w_{0}$, $B_{1}\subseteq B\subseteq V_0$ and $D_{1}\subseteq D\subseteq C$ with $\mu(B_{1})=1$, $\mu_1(D_{1})=1$ such that:

$w_{0}\in B$, $w<_{\textsl{R}_1}w_{0}$, $T_{(p,q)}(w_{0})\in D$ for every $(p,q) \in F$,

$B_{1}=\{u\in B_{w_{0}} : w_{0} + u\in B$ and $ T_{(p,q)}(w_{0})+ u\in B$ for all $(p,q) \in F$\}, and

$D_{1}=\{z\in D_{w_{0}} : w_{0}+z\in B$ and $T_{(p,q)}(w_{0})+ z\in D$ for all $ (p,q) \in F$\}.
\newline
The proof of Claim $2$ follows from the properties (1), (2) and (3) of the ultrafilters
$\mu\in\theta V_0 \subseteq \beta X$ and $\mu_1\in\theta C\subseteq \beta X$, since $ \mu_1=\beta T_{(p,q)}(\mu)=\beta T_{(p,q)}(\mu)+\mu_1$ and $\mu=\mu+\mu=\beta T_{(p,q)}(\mu)+\mu=\mu+\mu_1$ for every
$(p,q)\in F$.

We will construct, by induction on $n$, the required sequence $(w_{n})_{n\in\nat}\subseteq V_0$. Since, $V=A_1\cup\ldots\cup A_r$ and $C=C_1\cup\ldots\cup C_s$, there exist $1\leq i_{0} \leq r$, $1\leq j_{0} \leq s$ such that $\mu(A_{i_{0}})=1$ and $\mu_1(C_{j_{0}})=1$. Let $w\in V_0$. According to Claim $2$, starting with $w\in V_0$, $B_{1}=A_{i_{0}}\cap V_0$ and $D_{1}=C_{j_{0}}$, we can construct an increasing sequence $w<_{\textsl{R}_1}w_{1}<_{\textsl{R}_1}w_{2}<_{\textsl{R}_1}\ldots$ in $V_0$ and two decreasing sequences $V_0\supseteq B_{1}\supseteq B_{2}\supseteq \ldots$, and $C\supseteq D_{1}\supseteq D_{2}\supseteq \ldots$ such that for every $n\in\nat$ to satisfy:

$\mu(B_{n})=1$ and $\mu_1(D_{n})=1$,
$w_{n}\in B_{n}$ and $T_{(p,q)}(w_{n})\in D_{n}$ for every $(p,q)\in F_n$,

$B_{n+1}=\{u\in (B_{n})_{w_{n}} : w_{n} + u\in B_{n}$ and $ T_{(p,q)}(w_{n})+ u\in B_{n}$ for all $ (p,q)\in F_n \}$, and

$D_{n+1}=\{z\in (D_{n})_{w_{n}} : w_{n}+z\in B_{n}$ and $T_{(p,q)}(w_{n})+ z\in D_{n}$ for all $ (p,q)\in F_n \}$.

\noindent The sequence $(w_{n})_{n\in\nat}$ has the required properties. We will prove, by induction on $\lambda$, that, for every $\lambda\in\nat$,
\begin{center}
$T_{(p_1,q_1)}(w_{n_1})\star \ldots \star T_{(p_\lambda,q_\lambda)}(w_{n_\lambda})\in B_{n_1}\subseteq B_{1}\subseteq A_{i_{0}}$,
\end{center}
for every $n_1<\ldots<n_\lambda\in\nat$, $(p_i,q_i) \in F_{n_i}\cup\{(0,0)\}$ for every $1\leq i \leq \lambda$ and $(0,0) \in \{ (p_1,q_1),\ldots,(p_\lambda,q_\lambda)\}$; and that
\begin{center}
$T_{(p_1,q_1)}(w_{n_1})\star \ldots \star T_{(p_\lambda,q_\lambda)}(w_{n_\lambda})\in D_{n_1}\subseteq D_{1}=C_{j_{0}}$,
\end{center}
for every $n_1<\ldots<n_\lambda\in\nat$, $(p_i,q_i) \in F_{n_i}$ for every $1\leq i \leq \lambda$.

Indeed, for $n_1\in\nat$ and $(p_1,q_1)\in F_{n_1}$, we have $w_{n_1}\in B_{n_1}$ and $T_{(p_1,q_1)}(w_{n_1})\in D_{n_1}$.  Assume that the assertion  holds for $\lambda\geq 1$ and let $n_1<\ldots<n_\lambda<n_{\lambda+1}\in\nat$ and $ (p_i,q_i) \in F_{n_i}\cup\{(0,0)\}$ for every $1\leq i \leq \lambda+1$.

\noindent {\bf Case 1} Let $(0,0)\in \{(p_2,q_2),\ldots,(p_{\lambda+1},q_{\lambda+1})\}$. \\ Then $u=T_{(p_2,q_2)}(w_{n_2})\star \ldots \star T_{(p_{\lambda+1},q_{\lambda+1})}(w_{n_{\lambda+1}})\in B_{n_2}\subseteq B_{n_1 +1}\;$ according to the induction hypothesis. Hence,
$ T_{(p_1,q_1)}(w_{n_1})+ u=T_{(p_1,q_1)}(w_{n_1})\star \ldots \star T_{(p_{\lambda+1},q_{\lambda+1})}(w_{n_{\lambda+1}})\in B_{n_1}$.

\noindent {\bf Case 2} Let $(0,0)\notin \{(p_2,q_2),\ldots,(p_{\lambda+1},q_{\lambda+1})\}$. \\ Then  $z=T_{(p_2,q_2)}(w_{n_2})\star \ldots \star T_{(p_{\lambda+1},q_{\lambda+1})}(w_{n_{\lambda+1}})\in D_{n_2}\subseteq D_{n_1 +1}\;$ according to the induction hypothesis.
If $(p_1,q_1)=(0,0)$, then
$ w_{n_1}+ z=T_{(p_1,q_1)}(w_{n_1})\star \ldots \star T_{(p_{\lambda+1},q_{\lambda+1})}(w_{n_{\lambda+1}})\in B_{n_1}$.

\noindent If $(p_1,q_1)\in F_{n_1}$, then
$ T_{(p_1,q_1)}(w_{n_1})+ z=T_{(p_1,q_1)}(w_{n_1})\star \ldots \star T_{(p_{\lambda+1},q_{\lambda+1})}(w_{n_{\lambda+1}})\in D_{n_1}$.
\newline
This finishes the proof.
\end{proof}

For an alphabet $\Sigma=\{\alpha_1, \alpha_2, \ldots \}$, an increasing sequence $\vec{k}=(k_n)_{n\in\nat}\subseteq\nat$ and $\upsilon \notin \Sigma$ let the sets of \textit{$\omega$-located words} over $\Sigma$ dominated by $\vec{k}$ be
\begin{center}
$L(\Sigma, \vec{k})= \{w=w_{n_1}\ldots w_{n_l} :\; l\in\nat, n_1<\ldots<n_l\in \nat$, $ w_{n_i}\in \{\alpha_1,\ldots,\alpha_{k_{n_i}}\}\;\;\;\;\;\;\;\;\;\;\;\;\;$\\ $\;$ for every $1\leq i \leq l \},\;\;\;\;\;\;\;\;\;\;\;\;\;\;\;\;\;\;\;\;\;\;\;\;\;\;\;\;\;\;\;\;\;\;\; \;\;\;\;\;\;\;\;\;\;\;\;\;\;\;\;\;\;\;\;\;\;\;\;\;\;\;\;$\\
$L(\Sigma, \vec{k} ; \upsilon) = \{w=w_{n_1}\ldots w_{n_l} :\; l\in\nat, n_1<\ldots<n_l\in \nat, w_{n_i}\in \{\upsilon, \alpha_1, \ldots, \alpha_{k_{n_i}} \}\;\;\;\;\;\;$\\
 $\;\;\;\;\;\;\;\;\;\;\;\;\;\;$ for every $ 1\leq i \leq l$ and there exists $1\leq i \leq l$ with $w_{n_i}=\upsilon \}$, and\\
$L(\Sigma\cup\{\upsilon\}, \vec{k})=L(\Sigma, \vec{k})\cup L(\Sigma, \vec{k} ; \upsilon).\;\;\;\;\;\;\;\;\;\;\;\;\;\; \;\;\;\;\;\;\;\;\;\;\;\;\;\;\;\;\;\;\;\;\;\;\;\;\;\;\;\;\;\;\;\;\;\;\; \;\;\;\;\;\;\;\;\;\;\;\;\;\;\;\;\;\;\;\;\;\;$
\end{center}
For $w,u\in L(\Sigma\cup\{\upsilon\}, \vec{k})$ we write $w <_{\textsl{R}_2}u\;\Longleftrightarrow\;\max dom(w)< \min dom(u)$.
\newline
Let $L^\infty(\Sigma, \vec{k} ; \upsilon)=\{(w_n)_{n\in\nat}\subseteq L(\Sigma, \vec{k} ; \upsilon) : w_n<_{\textsl{R}_2}w_{n+1}$ for every $n\in\nat\}$.
\newline
For every $p\in\nat\cup\{0\}$ are defined the functions $T_p : L(\Sigma\cup\{\upsilon\}, \vec{k}) \rightarrow L(\Sigma\cup\{\upsilon\}, \vec{k})$ setting for $w=w_{n_1}\ldots w_{n_l}\in L(\Sigma\cup\{\upsilon\}, \vec{k})$: $T_0(w)=w$ and, for $p\in\nat$, $T_p(w)=u_{n_1}\ldots u_{n_l}$, where, for $1\leq i\leq l$, $u_{n_i}=w_{n_i}$ if $w_{n_i}\in \Sigma$, $u_{n_i}=\alpha_{p} $ if $w_{n_i}=\upsilon$ and $p\leq k_{n_i}$ and finally $u_{n_i}=\alpha_{k_{n_i}} $ if $w_{n_i}=\upsilon$ and $p>k_{n_i}$.
\newline
Let $\vec{w}=(w_n)_{n\in\nat}\in L^\infty(\Sigma, \vec{k} ; \upsilon)$. The set $EV(\vec{w})$ of all the \textit{extracted variable $\omega$-located words} of $\vec{w}$ contains the words of the form
$T_{p_1}(w_{n_1})\star \ldots \star T_{p_\lambda}(w_{n_\lambda})$,
where $\lambda\in\nat$, $n_1<\ldots<n_\lambda\in\nat$ and $p_1,\ldots,p_\lambda\in\nat\cup\{0\}$ such that $0\leq p_i \leq k_{n_i}$ for every $1\leq i \leq \lambda$ and $0 \in \{ p_1,\ldots,p_\lambda\}$, and the set $E(\vec{w})$ of all the \textit{extracted $\omega$-located words} of $\vec{w}$ contains the words of the form
$T_{p_1}(w_{n_1})\star \ldots \star T_{p_\lambda}(w_{n_\lambda})$,
where $\lambda\in\nat$, $n_1<\ldots<n_\lambda\in\nat$ and $p_1,\ldots,p_\lambda\in\nat$ such that $1\leq p_i \leq k_{n_i}$ for every $1\leq i \leq \lambda$. Let

$EV^{\infty}(\vec{w}) = \{\vec{u}=(u_n)_{n\in\nat} \in L^\infty (\Sigma, \vec{k} ; \upsilon) : u_n\in EV(\vec{w})$ for every $n\in\nat \}$.
\newline
If $\vec{u} \in EV^{\infty}(\vec{w})$, then we say that $\vec{u}$ is an \textit{extraction} of $\vec{w}$ and we write $\vec{u} \prec \vec{w}$. Notice that $\vec{u} \prec \vec{w}$ if and only if $EV(\vec{u})\subseteq EV(\vec{w}).$

\begin{cor}
[\textsf{Partition theorem for $\omega$-located words}(Carlson, \cite{C})]
\label{thm:block-Ramsey1}
Let $\Sigma=\{\alpha_n:\;n\in \nat\}$ be an infinite countable alphabet, $\upsilon \notin \Sigma$ a variable, $\vec{k}=(k_n)_{n\in\nat}\subseteq \nat$ an increasing sequence and $r,s\in\nat$. If $L(\Sigma, \vec{k} ; \upsilon)=A_1\cup\ldots\cup A_r$ and
$L(\Sigma, \vec{k})=C_1\cup\ldots\cup C_s$, then there exists a sequence $\vec{w}=(w_n)_{n\in\nat}\in L^\infty(\Sigma, \vec{k} ; \upsilon)$ and $1\leq i_0 \leq r$, $1\leq j_0 \leq s$ such that
\begin{center}
$EV(\vec{w})\in A_{i_0}$ and  $E(\vec{w})\in C_{j_0}$.
\end{center}
\end{cor}

\begin{proof}
 We set $\widetilde{\Sigma}=\{\alpha_n:n\in \zat^\ast\}$  and $\vec{k}_\ast=(\widetilde{k}_{n})_{n \in \zat^\ast}\subseteq \mathbb{N},$ where $\alpha_{-n}=\alpha_n$ and $\widetilde{k}_{-n}=\widetilde{k}_{n}=k_{n}$ for every $n\in  \mathbb{N}.$ Let $\varphi:\widetilde{L}_0(\widetilde{\Sigma}\cup\{\upsilon\},\vec{k}_\ast)\rightarrow L(\Sigma\cup\{\upsilon\},\vec{k})$ with \begin{center} $\varphi(w_{n_{1}}\ldots w_{n_{l}})=w_{n_{i_0}}\ldots w_{n_{l}},$ where $n_{i_0}=\min dom^+(w_{n_{1}}\ldots w_{n_{l}}).$\end{center}
  We have that $\varphi(\widetilde{L}_0(\widetilde{\Sigma},\vec{k}_\ast))=L(\Sigma,\vec{k})$ and $\varphi(\widetilde{L}_0(\widetilde{\Sigma},\vec{k}_\ast;\upsilon))=L(\Sigma,\vec{k};\upsilon),$ so
$\widetilde{L}_0(\widetilde{\Sigma},\vec{k}_\ast;\upsilon)$ $=\bigcup^{r}_{i=1}\varphi^{-1}(A_{i})$ and $ \widetilde{L}_0(\widetilde{\Sigma},\vec{k}_\ast)=\bigcup^{s}_{j=1}\varphi^{-1}(C_{j}).$ According to Theorem~\ref{thm:block-Ramsey02}, there exists $(\widetilde{w}_{n})_{n \in \mathbb{N}}\subseteq \widetilde{L}_0(\widetilde{\Sigma},\vec{k}_\ast;\upsilon)$ with $\widetilde{w}_{n}<_{\textsl{R}_1}\widetilde{w}_{n+1}$ for every $n \in \mathbb{N}$ and $1\leq i_{0}\leq r,1\leq j_{0}\leq s$ such that
 $T_{(p_{1},q_1)}(\widetilde{w}_{n_{1}})\star\ldots\star T_{(p_{\lambda},q_{\lambda})}(\widetilde{w}_{n_{\lambda}}) \in \varphi^{-1}(A_{i_{0}})$
for every $\lambda \in \mathbb{N},\; n_{1}<\ldots<n_{\lambda} \in \mathbb{N},\; (p_{i},q_i) \in \mathbb{N}\times\nat\cup \{(0,0)\}$ with $0 \leq p_{i},q_i\leq k_{n_{i}}$ for every $1\leq i\leq \lambda$ and $(0,0) \in \{(p_{1},q_1),\ldots,(p_{\lambda},q_{\lambda})\},$ and
$T_{(p_{1},q_1)}(\widetilde{w}_{n_{1}})\star\ldots\star T_{(p_{\lambda},q_{\lambda})}(\widetilde{w}_{n_{\lambda}}) \in \varphi^{-1}(C_{j_{0}})$
for every $\lambda \in \mathbb{N},$ $\; n_{1}<\ldots<n_{\lambda} \in \mathbb{N},\; (p_{i},q_i) \in \mathbb{N}\times\nat$ with $1 \leq p_{i},q_i\leq k_{n_{i}}$ for every $1\leq i\leq \lambda.$

We set $w_n= \varphi(\widetilde{w}_{n})\in L(\Sigma,\vec{k};\upsilon)$ for every $n\in \nat.$ Observe that $w_n<_{\textsl{R}_2}w_{n+1}$ for every $n\in \nat,$ since $\widetilde{w}_{n}<_{\textsl{R}_1}\widetilde{w}_{n+1}$ for every $n \in \mathbb{N},$ hence $(w_n)_{n\in \nat}\in L^{\infty}(\Sigma,\vec{k};\upsilon).$ The sequence $(w_n)_{n\in \nat}$ satisfies the conclusion, since $T_p(w_n)=\varphi(T_{(p,p)}(\widetilde{w}_{n}))$ for $0\leq p\leq k_n.$
\end{proof}

Now, we will prove, in Theorem~\ref{thm:block-Ramsey2} below, a stronger version of Theorem~\ref{thm:block-Ramsey02}, using the notion of extracted $\omega$-$\mathbb{Z}^\ast$-located words of a given orderly sequence of variable $\omega$-$\mathbb{Z}^\ast$-located words.

\subsection*{Extracted $\omega$-$\mathbb{Z}^\ast$-located words, Extractions}

Let $\Sigma=\{\alpha_n:\;n\in\mathbb{Z}^\ast\}$ be an alphabet,  $\vec{k}=(k_n)_{n\in\mathbb{Z}^\ast}\subseteq\nat$ such that
$(k_n)_{n\in\nat}$, $(k_{-n})_{n\in\nat}$ are increasing sequences ($(k_{-n})_{n\in\nat}$ is \textit{increasing} if $k_{-n}\leq k_{-(n+1)}$ for every $n\in \nat$) and let $\upsilon \notin \Sigma$.
We fix a sequence $\vec{w} = (w_n)_{n\in\nat}\in \widetilde{L}^\infty (\Sigma, \vec{k} ; \upsilon)$.

An \textit{extracted variable $\omega$-$\mathbb{Z}^\ast$-located word} of $\vec{w}$ is a variable $\omega$-$\mathbb{Z}^\ast$-located word $u\in \widetilde{L}_0(\Sigma, \vec{k} ; \upsilon)$ such that
$u=T_{(p_1,q_1)}(w_{n_1})\star \ldots \star T_{(p_\lambda,q_\lambda)}(w_{n_\lambda})$,
where $\lambda\in\nat$, $n_1<\ldots<n_\lambda\in\nat$, $(p_i,q_i) \in \nat\times \nat\cup\{(0,0)\}$ with $0\leq p_i\leq k_{n_i}$, $0\leq q_i \leq k_{-n_i}$ for every $1\leq i \leq \lambda$ and $(0,0) \in \{ (p_1,q_1),\ldots,(p_\lambda,q_\lambda)\}$.
The set of all the extracted variable $\omega$-$\mathbb{Z}^\ast$-located words of $\vec{w}$ is denoted by $\widetilde{EV}(\vec{w})$.

An \textit{extracted $\omega$-$\mathbb{Z}^\ast$-located word} of $\vec{w}$ is an $\omega$-$\mathbb{Z}^\ast$-located word $z\in \widetilde{L}_0(\Sigma, \vec{k})$ with
$z=T_{(p_1,q_1)}(w_{n_1})\star \ldots \star T_{(p_\lambda,q_\lambda)}(w_{n_\lambda})$,
where $\lambda\in\nat$, $n_1<\ldots<n_\lambda\in\nat$ and $(p_i,q_i) \in \nat\times \nat$ with $1\leq p_i\leq k_{n_i}$, $1\leq q_i \leq k_{-n_i}$ for every $1\leq i \leq \lambda$.
The set of all the extracted $\omega$-$\mathbb{Z}^\ast$-located words of $\vec{w}$ is denoted by $\widetilde{E}(\vec{w})$.

Let $\widetilde{EV}^{\infty}(\vec{w}) = \{\vec{u}=(u_n)_{n\in\nat} \in \widetilde{L}^\infty (\Sigma, \vec{k} ; \upsilon) : u_n\in \widetilde{EV}(\vec{w})$ for every $n\in\nat \}$.
\newline
If $\vec{u} \in \widetilde{EV}^{\infty}(\vec{w})$, then we say that $\vec{u}$ is an \textit{extraction} of $\vec{w}$ and we write $\vec{u} \prec \vec{w}$. Notice that for $\vec{u},\vec{w}\in \widetilde{L}^{\infty} (\Sigma, \vec{k} ; \upsilon)$ we have $\vec{u} \prec \vec{w}$ if and only if $\widetilde{EV}(\vec{u})\subseteq \widetilde{EV}(\vec{w})$.

Now we can state and prove a strengthening of Theorem~\ref{thm:block-Ramsey02}.
\begin{thm}
\label{thm:block-Ramsey2}
Let $\Sigma=\{\alpha_n \;:\;n \in \mathbb{Z}^{\ast}\}$ be an alphabet,  $\vec{k}=(k_n)_{n\in\mathbb{Z}^\ast}\subseteq \nat$ such that
$(k_n)_{n\in\nat}$, $(k_{-n})_{n\in\nat}$ are increasing sequences, $\upsilon \notin \Sigma$ and let $\vec{w}=(w_n)_{n\in \nat}\in \widetilde{L}^{\infty} (\Sigma, \vec{k} ; \upsilon)$, $r,s\in\nat$. If $\widetilde{L}(\Sigma, \vec{k} ; \upsilon)=A_1\cup\ldots\cup A_r$ and $\widetilde{L}(\Sigma, \vec{k})=C_1\cup\ldots\cup C_s$, then there exists an extraction $\vec{u}=(u_n)_{n\in\nat}$ of $\vec{w}$ and $1\leq i_0 \leq r$, $1\leq j_0 \leq s$ such that
\begin{center}
$\widetilde{EV}(\vec{u})\subseteq A_{i_0}$ and $\widetilde{E}(\vec{u})\subseteq C_{j_0}$.
\end{center}
\end{thm}
\begin{proof}[Proof]
We will order the set $\nat\times \nat$. For $q\in \nat$ set $i(q)=1$ if $q\leq k_{-1}$ and $i(q)=n+1$ if $k_{-n}<q\leq k_{-n-1}$ for $n\geq 1.$ Let $(p_1,q_1), (p_2,q_2)\in \nat\times \nat$. We write $(p_1,q_1)<_\ast(p_2,q_2)$ in cases (1) $i(q_1)<i(q_2)$ or (2) $i(q_1)=i(q_2)$ and $p_1<p_2$ or (3) $i(q_1)=i(q_2)$ and $p_1=p_2$ and $q_1< q_2$. Let  $\nat\times \nat=\{\beta_1<_\ast\beta_2<_\ast\beta_3<_\ast\ldots\}$ and the increasing sequence $\vec{l}=(l_n)_{n\in\nat}\subseteq\nat$ such that $\beta_{l_n}=(k_n, k_{-n})$. We set $\widetilde{\Sigma}=\{\beta_n:\;n\in \nat\}$ and we define the function
$h:L(\widetilde{\Sigma}\cup\{\upsilon\}, \vec{l} \;)\rightarrow \widetilde{EV}(\vec{w})\cup \widetilde{E}(\vec{w})$ which sends $\;$ $t_{n_1}\ldots t_{n_\lambda}\in L(\widetilde{\Sigma}\cup\{\upsilon\}, \vec{l} \;)$ to
$$h(t_{n_1}\ldots t_{n_\lambda})=T_{(p_1,q_1)}(w_{n_1})\star \ldots \star T_{(p_\lambda,q_\lambda)}(w_{n_\lambda}),$$ where for $1\leq i \leq \lambda$, $(p_i,q_i)=(0,0)$ if $t_{n_i}=\upsilon$ and $(p_i,q_i)=(\mu_1, \mu_2)$ if $t_{n_i}=\beta_{\mu}=(\mu_1,\mu_2)$. The function $h$ is one-to-one and onto $\widetilde{EV}(\vec{w})\cup \widetilde{E}(\vec{w})$. Also, $h(L(\widetilde{\Sigma}, \vec{l} ; \upsilon))=\widetilde{EV}(\vec{w})$ and $h(L(\widetilde{\Sigma}, \vec{l}\; ))=\widetilde{E}(\vec{w})$.

According to Theorem~\ref{thm:block-Ramsey1}, there exist a sequence $\vec{s}=(s_n)_{n\in\nat}\in
L^\infty (\widetilde{\Sigma}, \vec{l} ; \upsilon)$ and $1\leq i_0 \leq r$, $1\leq j_0 \leq s$ such that $EV(\vec{s})\subseteq h^{-1}(A_{i_0})$ and $E(\vec{s})\subseteq h^{-1}(C_{j_0})$. Set $u_n=h(s_n)\in \widetilde{EV}(\vec{w})$ for every $n\in\nat$ and $\vec{u}=(u_n)_{n\in\nat}$. Then $\vec{u}=(u_n)_{n\in\nat}$ is an extraction of $\vec{w}$ and $\widetilde{EV}(\vec{u})\subseteq h(EV(\vec{s}))\subseteq A_{i_0}$ and $\widetilde{E}(\vec{u})\subseteq h(E(\vec{s}))\subseteq C_{j_0}$.
\end{proof}

 We will now prove a finitistic form of Theorem~\ref{thm:block-Ramsey02}, which can be considered as a strengthened Hales-Jewett theorem. We will need the following notation.

\begin{note} We denote by $\widetilde{L}(\Sigma,\vec{k},n)$ the set of all (constant) $\omega$-$\zat^\ast$-located words over $\Sigma$ dominated by $\vec{k}$ with length $n\in \nat,$ and if $\upsilon\notin \Sigma$ be a variable, we denote by $\widetilde{L}(\Sigma,\vec{k};\upsilon,n)$ the set of variable $\omega$-$\zat^\ast$-located words over $\Sigma$ dominated by $\vec{k}$ with length $n\in \nat.$ So, if $|A|$ is the cardinality of a finite set $A,$ then,

 $\widetilde{L}(\Sigma,\vec{k},n)=\{w\in \widetilde{L}(\Sigma,\vec{k}):\;|dom(w)|=n\},$ and

$\widetilde{L}(\Sigma,\vec{k};\upsilon,n)=\{w\in \widetilde{L}(\Sigma,\vec{k};\upsilon):\;|dom(w)|=n\}.$\end{note}

\begin{cor}\label{thm:HJ}
Let $\Sigma=\{\alpha_n:\;n\in \zat^\ast\}$ be an alphabet, $\upsilon\notin \Sigma$ a variable, $r,m\in \nat,$ $n_1,\ldots,n_m\in \nat$ and $\vec{k}=(k_n)_{n\in \zat^\ast}\subseteq \nat$ such that $(k_n)_{n\in \nat},$ $(k_{-n})_{n\in \nat}$ are increasing sequences. Then, there exists $n_0\equiv n_0(r,m,n_1,\ldots,n_m,\vec{k})\in \nat$ such that if $\widetilde{L}(\Sigma,\vec{k},n_0)=C_1\cup\ldots\cup C_r,$ there exist $t_1<_{\textsl{R}_1}\ldots<_{\textsl{R}_1}t_m\in \widetilde{L}_0(\Sigma,\vec{k};\upsilon)$ with $t_1\star\ldots\star t_m\in \widetilde{L}(\Sigma,\vec{k};\upsilon,n_0)$ such that for some $1\leq i_0\leq r$ to satisfy $$T_{(p_1,q_1)}(t_1)\star\ldots\star T_{(p_m,q_m)}(t_m)\in C_{i_0}$$ for all $1\leq p_i\leq k_{n_i},\;1\leq q_i\leq k_{-n_i}$ for $1\leq i\leq m.$
\end{cor}

\begin{proof}
Let the conclusion does not hold. Then, for every $n\in \nat$ we have that there exists a partition $\widetilde{L}(\Sigma,\vec{k},n)=C_1^n\cup\ldots\cup C_r^n,$ such that for every $t_1<_{\textsl{R}_1}\ldots<_{\textsl{R}_1}t_m$ $\in \widetilde{L}_0(\Sigma,\vec{k};\upsilon)$ with $t_1\star\ldots\star t_m\in \widetilde{L}(\Sigma,\vec{k};\upsilon,n)$ the conclusion does not hold. For $1\leq i\leq r$ let $B_i=\cup_{n\in \nat}C_i^n.$ Then $\widetilde{L}(\Sigma,\vec{k})=B_1\cup\ldots\cup B_r.$ According to Theorem~\ref{thm:block-Ramsey02} there exist $(w_n)_{n\in \nat}\in \widetilde{L}^{\infty}(\Sigma,\vec{k};\upsilon)$ and $1\leq i_0\leq r$ such that $T_{(p_1,q_1)}(w_{n_1})\star\ldots\star T_{(p_m,q_m)}(w_{n_m})\in B_{i_0}$ for all $1\leq p_i\leq k_{n_i},\;1\leq q_i\leq k_{-n_i}$ for $1\leq i\leq m.$ If $|dom(w_{n_1}\star\ldots\star w_{n_m})|=n,$ then $T_{(p_1,q_1)}(w_{n_1})\star\ldots\star T_{(p_m,q_m)}(w_{n_m})\in C_{i_0}^n$ for all $1\leq p_i\leq k_{n_i},\;1\leq q_i\leq k_{-n_i}$ for $1\leq i\leq m,$
a contradiction.
\end{proof}

\section{Extended Ramsey type partition theorems for $\omega$-$\mathbb{Z}^\ast$-located words}

The main result of this section is Theorem~\ref{thm:block-Ramsey}, where we prove an extended to every countable order Ramsey-type partition theorem for variable and constant $\omega$-$\mathbb{Z}^\ast$-located words over an alphabet $\Sigma=\{ \alpha_n :\; n \in \mathbb{Z}^{\ast}\},$ dominated by a sequence $\vec{k}=(k_n)_{n\in\mathbb{Z}^\ast}\subseteq\nat.$ It is an extension to every countable order $\xi$ of Theorem~\ref{thm:block-Ramsey2} corresponding to the case $\xi=1.$ In the proof we apply some technics, introduced in [FN2] and [F4].

As consequences of Theorem~\ref{thm:block-Ramsey} we get an extended to every countable order Ramsey-type partition theorem for $\omega$-located words (see Corollary~\ref{cor:nat}) and consequently for located words over a finite alphabet.

The vehicle for proving this extended Ramsey type partition theorem (Theorem~\ref{thm:block-Ramsey}) is the Schreier systems  $(\widetilde{L}^\xi(\Sigma, \vec{k}))_{\xi<\omega_1}$ and $(\widetilde{L}^\xi(\Sigma, \vec{k} ; \upsilon))_{\xi<\omega_1}$ (Definition~\ref{recursivethinblock}), consisting of families of finite orderly sequences of (constant and variable respectively) $\omega$-$\mathbb{Z}^\ast$-located words over the alphabet $\Sigma$ dominated by the sequence
$\vec{k}$. Instrumental for this definition are the Schreier sets $\A_\xi,$ consisting of finite subsets of $\nat,$ which are defined below, employing (in case 3(iii)) the Cantor normal form of ordinals
(cf. \cite{KM}, \cite{L}). Schreier sets were systematically studied in [F1] and [F3].

For $s_1, s_2 \in [\nat]^{<\omega}_{>0}$ we write $s_1 < s_2\;\Longleftrightarrow\;\max s_1 < \min s_2$.

\begin{defn}[\textsf{The Schreier system},
{[F1, Def. 7], [F2, Def. 1.5], [F3, Def. 1.4]}]
\label{Irecursivethin}
For every non-zero, countable, limit ordinal $\lambda$ choose and fix a strictly
increasing sequence $(\lambda_n)_{n \in \nat}$ of successor ordinals smaller than $\lambda$
with $\sup_n \lambda_n = \lambda$.
The system $(\A_\xi)_{\xi<\omega_1}$ is defined recursively as follows: \begin{center}(1) $\A_0 = \{\emptyset\}$ and $\A_1 = \{\{n\} : n\in \nat\};\;\;\;\;\;\;\;\;\;\;\;\;\;\;\;\;\;\; \;\;\;\;\;\;\;\;\;\;\;\;\;\;\;\;\;\;\;\;\;\;\;\;\;\;\;\;\;\;\;\;\;\;\;\; \;\;\;\;\;\;\;\;\;\;\;\;\;\;\;\;$\\
\item[(2)] $\A_{\zeta+1} = \{s\in [\nat]^{<\omega}_{>0} : s= \{n\} \cup s_1$,
where $n\in \nat$, $\{n\} <s_1$ and $s_1\in \A_\zeta\};\;\;\;\;\;\;\;\;\;\;\;\;$\\
\item[(3i)] $\A_{\omega^{\beta+1}} = \{s\in [\nat]^{<\omega}_{>0} :
s = \bigcup_{i=1}^n s_i$, where $n= \min s_1$, $s_1<\ldots < s_n$ and $\;\;\;\;\;\;\;\;\;\;\;\;\;$\\
$s_1,\ldots, s_n\in \A_{\omega^\beta}\};\;\;\;\;\;\;\;\;\;\;\;\;\;\;\;\;\;\;\;\;\;\;\;\; \;\;\;\;\;\;\;\;\;\;\;\;\;\;\;\;\;\;\;\;\;\;\;\;\;\;\;\;\;\;\;\;\;\;\;\;$\\
 \item[(3ii)] for a non-zero, countable limit ordinal $\lambda,\;\;\;\;\;\;\;\;\;\;\;\;\;\;\;\;\;\;\;\;\;\;\;\;\;\;\;\;\;\;\;\;\;\;\;\;\;\;\;\;\;\;\;\;\;\;\;\; \;\;\;\;\;\;\;\;\;\;\;$\\
$\A_{\omega^\lambda} = \{s\in [\nat]^{<\omega}_{>0} : s\in \A_{\omega^{\lambda_n}}$
with $n= \min s\}$; and $\;\;\;\;\;\;\;\;\;\;\;\;\;\;\;\;\;\;\;\;\;\;\;\;\;\;\;\;\;\;\;\;\;\;\;\;$\\
\item[(3iii)] for a limit ordinal $\xi$ such that $\omega^{\alpha}< \xi < \omega^{\alpha +1}$ for some $0< \alpha <\omega_1$, if $\;\;\;\;\;\;\;\;\;\;\;\;\;$\\
$\;\;\;\;\;\;\;\;\;\;\;\;\xi = \omega^{\alpha} p
+ \sum_{i=1}^m \omega^{a_i} p_i$, where $m\in \nat$ with $m\ge0$,
 $p,p_1,\ldots,p_m$ are natural $\;\;\;\;\;\;$\\ $\;\;\;\;$ numbers with $p,p_1,\ldots,p_m\ge1$
(so that either $p>1$, or $p=1$ and  $m\ge 1$) \\ and
$a,a_1,\ldots,a_m$ are ordinals with $a>a_1>\ldots >a_m >0,\;\;\;\;\;\;\;\;\;\;\;\;\;\;\;\;\;\;\;\;$\\
$\;\;\;\;\;\;\;\;\;\;\;\;\A_\xi = \{s\in [\nat]^{<\omega}_{>0} :s= s_0 \cup (\bigcup_{i=1}^m s_i)$
with $s_m < \ldots < s_1 <s_0$,
$s_0= s_1^0\cup\ldots\cup s_p^0$\\ $\;\;\;\;\;\;\;\;\;\;\;\;\;\;\;\;\;\;\;$ with $s_1^0<\ldots < s_p^0\in \A_{\omega^a}$,
and $s_i = s_1^i \cup\ldots\cup s_{p_i}^i$ with
$s_1^i <\ldots<\ s_{p_i}^i\in \;$\\ $\A_{\omega^{a_i}}$ for every $1\le i\le m\}.\;\;\;\;\;\;\;\;\;\;\;\;\;\;\;\;\;\;\;\;\;\;\;\;\;\;\;\;\;\;\;\;\;\;\;\;\;\;\;\;\;\; \;\;\;\;\;\;$
\end{center}
\end{defn}

Let $\Sigma=\{ \alpha_n\;:\; n \in \mathbb{Z}^{\ast}\}$ be an alphabet, $\upsilon \notin \Sigma$ a variable and $\vec{k}=(k_n)_{n\in\mathbb{Z}^\ast}$ a two-sided sequence of natural numbers. We  define the \textit{finite orderly sequences of $\omega$-$\mathbb{Z}^\ast$-located words} over $\Sigma$ dominated by $\vec{k}$ as follows:
\begin{center}
$\widetilde{L}^{<\infty}(\Sigma, \vec{k}) = \{\bw = (w_1,\ldots,w_l) : l\in\nat, $
$w_1<_{\textsl{R}_1}\ldots <_{\textsl{R}_1} w_l\in \widetilde{L}_0(\Sigma, \vec{k}) \} \cup \{\emptyset \},$ and
\end{center}
\begin{center}
$\widetilde{L}^{<\infty}(\Sigma, \vec{k} ; \upsilon) = \{\bw = (w_1,\ldots,w_l) : l\in\nat, $
$w_1<_{\textsl{R}_1}\ldots <_{\textsl{R}_1} w_l\in \widetilde{L}_0(\Sigma, \vec{k} ; \upsilon) \} \cup \{\emptyset \}$.\end{center}

Finally we set $\widetilde{L}^{<\infty}(\Sigma\cup\{\upsilon\}, \vec{k})=\widetilde{L}^{<\infty}(\Sigma, \vec{k})\cup \widetilde{L}^{<\infty}(\Sigma, \vec{k} ; \upsilon).$

We will define now the Schreier systems $(\widetilde{L}^\xi(\Sigma, \vec{k}))_{\xi<\omega_1}$ and $(\widetilde{L}^\xi(\Sigma, \vec{k} ; \upsilon))_{\xi<\omega_1}$.
\begin{defn}
[\textsf{The Schreier systems $(\widetilde{L}^\xi(\Sigma, \vec{k}))_{\xi<\omega_1}$ and $(\widetilde{L}^\xi(\Sigma, \vec{k} ; \upsilon))_{\xi<\omega_1}$}]
\label{recursivethinblock}
Let $\Sigma=\{\alpha_n:\;n \in \mathbb{Z}^\ast\}$ be an alphabet, $\upsilon \notin \Sigma$ a variable and $\vec{k}=(k_n)_{n\in\mathbb{Z}^\ast}\subseteq\nat$. We define
\newline
$\widetilde{L}^0(\Sigma, \vec{k})= \{\emptyset\}=\widetilde{L}^0(\Sigma, \vec{k} ; \upsilon) $, and
for every countable ordinal $\xi\ge1$,
\newline
$\widetilde{L}^\xi(\Sigma, \vec{k}) = \{(w_1,\ldots,w_l)\in \widetilde{L}^{<\infty}(\Sigma, \vec{k}) :
\{\min dom^{+}(w_1),\ldots, \min dom^{+}(w_l)\}\in\A_\xi \}$,
\newline
$\widetilde{L}^\xi(\Sigma, \vec{k}; \upsilon) = \{(w_1,\ldots,w_l)\in \widetilde{L}^{<\infty}(\Sigma, \vec{k}; \upsilon) :
\{\min dom^{+}(w_1),\ldots, \min dom^{+}(w_l)\}\in\A_\xi \}$.
\end{defn}

\begin{remark}\label{rem1.4}
(i)
$\emptyset \notin \widetilde{L}^\xi(\Sigma, \vec{k}; \upsilon)$ for every $\xi\ge1$.

(ii) $\widetilde{L}^m(\Sigma, \vec{k} ; \upsilon) = \{(w_1,\ldots, w_m) : w_1<_{\textsl{R}_1}\ldots <_{\textsl{R}_1} w_m \in \widetilde{L}_0(\Sigma, \vec{k} ; \upsilon)\}$ for $m\in \nat$.

(iii) $\widetilde{L}^\omega(\Sigma, \vec{k} ; \upsilon) = \{(w_1,\ldots, w_n)\in \widetilde{L}^{<\infty}(\Sigma, \vec{k} ; \upsilon) : n\in \nat$,
and $\min dom^{+}(w_1) = n\}$.

(iv) Alternatively we could define the sets $\widetilde{L}^\xi(\Sigma, \vec{k}),\;\widetilde{L}^\xi(\Sigma, \vec{k};\upsilon)$ via the negative part of the domain of the words as follows
\begin{center}
$\widetilde{L}^\xi(\Sigma, \vec{k}) = \{(w_1,\ldots,w_l)\in \widetilde{L}^{<\infty}(\Sigma, \vec{k}) :
\{|\max dom^{-}(w_1)|,\ldots, |\max dom^{-}(w_l)|\}\in\A_\xi \},\;\;$\\
$\widetilde{L}^\xi(\Sigma, \vec{k};\upsilon) = \{(w_1,\ldots,w_l)\in \widetilde{L}^{<\infty}(\Sigma, \vec{k};\upsilon) :\{|\max dom^{-}(w_1)|,\ldots, |\max dom^{-}(w_l)|\}\in\;\;$ \\ $\A_\xi \}.\;\;\;\;\;\; \;\;\;\;\;\; \;\;\;\;\;\; \;\;\;\;\;\; \;\;\;\;\;\; \;\;\;\;\;\; \;\;\;\;\;\;\;\;\;\;\;\;\;\;\;\;\;\; \;\;\;\;\;\; \;\;\;\;\;\;\;\;\;\;\;\;\;\;\;\;\;\;\;\;\;$\end{center}
\end{remark}

The following proposition justifies the recursiveness of the systems $(\widetilde{L}^\xi(\Sigma, \vec{k}))_{\xi<\omega_1}$ and $(\widetilde{L}^\xi(\Sigma, \vec{k} ; \upsilon))_{\xi<\omega_1}$.
\newline
For a family $\F \subseteq \widetilde{L}^{<\infty} (\Sigma\cup\{\upsilon\}, \vec{k})$ and $t\in \widetilde{L}_0(\Sigma\cup\{\upsilon\}, \vec{k})$, we set

\begin{center}
$\F(t) = \{\bw\in \widetilde{L}^{<\infty} (\Sigma\cup\{\upsilon\}, \vec{k}) : $ either $\bw=(w_1,\ldots,w_l)\neq \emptyset$ and $(t,w_1, \ldots,w_l) \in \F$\\
or $\bw=\emptyset$ and $(t) \in\F\},\;\;\;\;\;\;\;\;\;\;\;\;\;\;\;\;\;\;\;\;\;\;\;\;\;\;\;\;\;\;\;\;\;\;\;\;\;\;\;\;\;\;\;\;\;\;\;\; \;\;\;\;\;\;\;\;\;\;\;\;\;\;\;\;\;\;$\\
$\F -t = \{\bw \in \F :$ either $\bw = (w_1,\ldots,w_l)\neq \emptyset$ and $t <_{\textsl{R}_1}w_1$, or
$\bw =\emptyset\}.\;\;\;\;\;\;\;\;\;\;\;\;\;\;\;\;\;\;$
\end{center}

\begin{prop}\label{justification}
For every countable ordinal $\xi\ge 1$, there exists a concrete sequence
$(\xi_n)_{n\in\nat}$ of countable ordinals with $\xi_n<\xi$ such that for $\Sigma=\{ \alpha_n:\;n \in \mathbb{Z}^\ast\}$ an alphabet, $\upsilon \notin \Sigma$ a variable, $\vec{k}=(k_n)_{n\in\mathbb{Z}^\ast}\subseteq\nat,$ $s\in \widetilde{L}_0(\Sigma, \vec{k})$ and $t\in \widetilde{L}_0(\Sigma, \vec{k}; \upsilon)$, with $\min dom^+(s)=\min dom^+(t)=n$, we have
\begin{center} $\widetilde{L}^\xi(\Sigma, \vec{k}) (s) = \widetilde{L}^{\xi_n}(\Sigma, \vec{k})\cap (\widetilde{L}^{<\infty}(\Sigma, \vec{k}) - s),\;\;$ and \\
$\widetilde{L}^\xi(\Sigma, \vec{k}; \upsilon) (t) = \widetilde{L}^{\xi_n}(\Sigma, \vec{k}; \upsilon)\cap (\widetilde{L}^{<\infty}(\Sigma, \vec{k}; \upsilon) - t)$.
\end{center}

Moreover, $\xi_n =\zeta$ for every $n\in\nat$ if $\xi = \zeta+1$, and
$(\xi_n)_{n\in\nat}$ is a strictly increasing sequence with $\sup_n \xi_n=\xi$ if $\xi$
is a limit ordinal.
\end{prop}

\begin{proof}
It follows from Theorem~1.6 in \cite{F3}, according to which for every countable ordinal $\xi >0$ there exists a concrete sequence
$(\xi_n)_{n\in\nat}$ of countable ordinals with  $\xi_n <\xi$, such that $\A_\xi (n) = \A_{\xi_n} \cap [\{n+1,n+2,\ldots\}]^{<\omega}$ for every $ n\in \nat$, where,
$\A_\xi(n) = \{s\in [\nat]^{<\omega} : s \in [\nat]^{<\omega}_{>0}, n < \min s$ and $\{n\}\cup s\in \A_\xi$ or $s=\emptyset$ and $\{n\}\in \A_\xi \}$.
Moreover, $\xi_n = \zeta$ for every $n\in\nat$ if $\xi = \zeta+1$, and
$(\xi_n)_{n\in\nat}$ is a strictly increasing sequence with $\sup_n \xi_n=\xi$ if $\xi$
is a limit ordinal.
\end{proof}

In order to state and prove the principal result of this section, a Ramsey type partition theorem for $\omega$-$\mathbb{Z}^\ast$-located words extended to every countable order, we need the following notation:

\subsection*{Notation}
Let $\Sigma=\{\alpha_n:\; n \in \mathbb{Z}^\ast\}$ be an alphabet, $\upsilon \notin \Sigma$ a variable and  $\vec{k}=(k_n)_{n\in\mathbb{Z}^\ast}\subseteq \nat$ such that  $(k_n)_{n\in \nat}$ and $(k_{-n})_{n\in \nat}$ are increasing sequences. For $\vec{w}=(w_n)_{n\in\nat} \in \widetilde{L}^\infty (\Sigma, \vec{k} ; \upsilon)$, $\bw = (w_1,\ldots,w_l)\in \widetilde{L}^{<\infty}(\Sigma, \vec{k} ; \upsilon)$ and $t\in \widetilde{L}_0(\Sigma, \vec{k} ; \upsilon)$, we set:

$\widetilde{E}^{<\infty}(\vec{w}) = \{\bu=(u_1,\ldots,u_l)\in \widetilde{L}^{<\infty} (\Sigma, \vec{k}) :\; l\in\nat, u_1,\ldots,u_l\in \widetilde{E}(\vec{w})\}\cup \{\emptyset \}$,

$\widetilde{EV}^{<\infty}(\vec{w}) = \{\bu=(u_1,\ldots,u_l)\in \widetilde{L}^{<\infty} (\Sigma, \vec{k} ; \upsilon) :\; l\in\nat, u_1,\ldots,u_l\in \widetilde{EV}(\vec{w})\}\cup \{\emptyset \}$,

$\widetilde{E}(\bw)=\{T_{(p_1,q_1)}(w_{n_1})\star \ldots \star T_{(p_\lambda,q_\lambda)}(w_{n_\lambda})\in \widetilde{L}_0(\Sigma, \vec{k})  :\; 1\leq n_1<\ldots<n_\lambda\leq l$ and \begin{center}
$\;\;\;\;\;\;\;\;\;(p_i,q_i)\in\nat\times \nat$ with $1\leq p_i \leq k_{n_i},\;1\leq q_i\leq k_{-n_i}$ for every $1\leq i \leq \lambda \}$,
\end{center}

$\widetilde{EV}(\bw)=\{T_{(p_1,q_1)}(w_{n_1})\star \ldots \star T_{(p_\lambda,q_\lambda)}(w_{n_\lambda})\in \widetilde{L}_0(\Sigma, \vec{k} ; \upsilon)  :\; 1\leq n_1<\ldots<n_\lambda\leq l$ and \begin{center}
$\;\;\;\;\;\;\;\;\;\;\;\;\;\;(p_i,q_i)\in(\nat\times\nat)\cup\{(0,0)\}$ with $0\leq p_i \leq k_{n_i},\;0\leq q_i \leq k_{-n_i}$ for every\\ $1\leq i \leq \lambda$ and $(0,0) \in \{ (p_1,q_1),\ldots,(p_\lambda,q_\lambda)\} \},$ $\;\;\;\;\;\;\;\;\;\;\;\;\;\;\;\;\;\;\;\;$
\end{center}

$\widetilde{E}^{<\infty}(\bw) = \{\bu=(u_1,\ldots,u_l)\in \widetilde{L}^{<\infty} (\Sigma, \vec{k}) : l\in\nat, u_1,\ldots,u_l\in \widetilde{E}(\bw)\}\cup \{\emptyset \},$ and

$\widetilde{EV}^{<\infty}(\bw) = \{\bu=(u_1,\ldots,u_l)\in \widetilde{L}^{<\infty} (\Sigma, \vec{k} ; \upsilon) : l\in\nat, u_1,\ldots,u_l\in \widetilde{EV}(\bw)\}\cup \{\emptyset \}$.
\newline
Observe that the sets $\widetilde{EV}(\bw)$, $\widetilde{E}(\bw)$ are finite. Also, we set

$\vec{w} -t = (w_n)_{n\geq l}\in \widetilde{L}^\infty (\Sigma, \vec{k} ; \upsilon)$, where $l=\min \{n\in\nat : t<_{\textsl{R}_1}w_n\}$, and

 $\vec{w}-\bw =  \vec{w} -w_l$.

\begin{thm}
[\textsf{Ramsey type partition theorem on Schreier families for $\omega$-$\mathbb{Z}^\ast$-located words}]
\label{thm:block-Ramsey}
Let $\xi\geq 1$ be a countable ordinal, $\Sigma=\{\alpha_n :\;n \in \mathbb{Z}^\ast\}$ be an alphabet, $\upsilon \notin \Sigma$ a variable and $\vec{k}=(k_n)_{n\in\mathbb{Z}^\ast}\subseteq \nat$ such that  $(k_n)_{n\in \nat},\;(k_{-n})_{n\in \nat}$ are increasing sequences. For every $\G\subseteq \widetilde{L}^{<\infty}(\Sigma, \vec{k}),$ $\F \subseteq \widetilde{L}^{<\infty}(\Sigma, \vec{k} ; \upsilon)$ and every infinite orderly sequence $\vec{w} \in \widetilde{L}^\infty (\Sigma, \vec{k} ; \upsilon)$ of variable $\omega$-$\mathbb{Z}^\ast$-located words there exists a variable extraction $\vec{u}\prec \vec{w}$ of $\vec{w}$ such that:

either $\widetilde{L}^\xi(\Sigma, \vec{k}) \cap \widetilde{E}^{<\infty}(\vec{u})\subseteq \G$, or
$\widetilde{L}^\xi(\Sigma, \vec{k}) \cap \widetilde{E}^{<\infty}(\vec{u})\subseteq \widetilde{L}^{<\infty}(\Sigma, \vec{k})\setminus \G,$ and

 either $\widetilde{L}^\xi(\Sigma, \vec{k} ; \upsilon) \cap \widetilde{EV}^{<\infty}(\vec{u})\subseteq \F$, or
$\widetilde{L}^\xi(\Sigma, \vec{k} ; \upsilon) \cap \widetilde{EV}^{<\infty}(\vec{u})\subseteq \widetilde{L}^{<\infty}(\Sigma, \vec{k} ; \upsilon)\setminus \F$.
\end{thm}

In the proof of this partition theorem we will make use of a diagonal argument, contained in the following lemmas.

\begin{lem}\label{lem:block-Ramsey0}
Let $\Sigma=\{\alpha_n:\;n \in \mathbb{Z}^\ast\}$ be an alphabet, $\upsilon \notin \Sigma$ a variable, $\vec{k}=(k_n)_{n\in\mathbb{Z}^\ast}\subseteq \nat$ such that  $(k_n)_{n\in \nat},\;(k_{-n})_{n\in \nat}$ are increasing sequences, $\vec{w} = (w_n)_{n\in\nat} \in \widetilde{L}^\infty (\Sigma, \vec{k} ; \upsilon)$, and
$$\Pi = \{(t,\vec{s}): t\in \widetilde{L}_0(\Sigma, \vec{k}),\vec{s} = (s_n)_{n\in\nat}\in \widetilde{L}^\infty (\Sigma, \vec{k} ; \upsilon)\;\text{with}\;\vec{s}\prec \vec{w}\;\text{and}\;t<_{\textsl{R}_1}s_1\}.$$

If a subset $\R$ of $\Pi$ satisfies
\begin{itemize}
\item[(i)] for every $(t,\vec{s})\in\Pi$, there exists $(t,\vec{s}_1)\in\R$ with
$\vec{s}_1 \prec \vec{s}$; and
\item[(ii)] for every $(t,\vec{s}) \in\R$ and $\vec{s}_1 \prec \vec{s}$, we have $(t,\vec{s}_1)\in\R$,
\end{itemize}
then there exists $\vec{u} \prec \vec{w}$, such that
$(t,\vec{s})\in \R$ for all $t\in \widetilde{E}(\vec{u})$ and $\vec{s} \prec \vec{u}-t$.
\end{lem}

\begin{proof}
Let $u_0 = w_1$.
According to condition $(i)$, there exists $\vec{s}_1 = (s^1_n)_{n\in\nat} \in
\widetilde{L}^\infty (\Sigma, \vec{k} ; \upsilon)$
with $\vec{s}_1\prec\vec{w}-u_0$ such that
$(T_{(p,q)}(u_0),\vec{s}_1)\in \R$ for every $1\leq p\leq k_{1},$ $1\leq q\leq k_{-1}.$
Let $u_1 = s^1_1$.
Of course, $u_0<_{\textsl{R}_1}u_1$ and $u_0, u_1 \in \widetilde{EV}(\vec{w})$.
We assume now that there have been constructed
$\vec{s}_1,\ldots,\vec{s}_n \in \widetilde{L}^\infty (\Sigma, \vec{k} ; \upsilon)$ and $u_0,u_1,\ldots,u_n \in \widetilde{EV}(\vec{w})$, with
$\vec{s}_n \prec \ldots \prec \vec{s}_1 \prec \vec{w}$, $u_0<_{\textsl{R}_1}u_1<_{\textsl{R}_1}\ldots <_{\textsl{R}_1} u_n$ and
$(t,\vec{s}_i)\in\R$ for all $t\in \widetilde{E}((u_0,\ldots,u_{i-1}))$ and $1\le i\le n.$

We will construct $\vec{s}_{n+1}$ and $u_{n+1}$.
Let $\{ t_1,\ldots, t_l\} = \widetilde{E} ((u_0,\ldots, u_n))$.
According to condition $(i)$, there exists
$\vec{s}_{n+1}^1,\ldots, \vec{s}_{n+1}^l \in \widetilde{L}^\infty (\Sigma, \vec{k} ; \upsilon)$ such that
$\vec{s}_{n+1}^l \prec\ldots \prec \vec{s}_{n+1}^1 \prec \vec{s}_n -u_n$ and
$(t_i, \vec{s}_{n+1}^i)\in \R$ for every $1\le i\le l$.
Set $\vec{s}_{n+1} = \vec{s}_{n+1}^l$. If $\vec{s}_{n+1} = (s^{n+1}_n)_{n\in\nat}$, set $u_{n+1} = s^{n+1}_1$.
Of course $u_n <_{\textsl{R}_1} u_{n+1}$, $u_{n+1} \in \widetilde{EV}(\vec{w})$ and, according to condition (ii),
$(t_i, \vec{s}_{n+1})\in \R$ for all $1\le i\le l$.

Set $\vec{u} = (u_0,u_1,u_2,\ldots) \in \widetilde{L}^\infty (\Sigma, \vec{k} ; \upsilon)$. Then
$\vec{u} \prec \vec{w}$, since $u_0<_{\textsl{R}_1}u_1<_{\textsl{R}_1} \ldots \in \widetilde{EV}(\vec{w})$.
Let $t\in \widetilde{E}(\vec{u})$ and $\vec{s} \prec \vec{u}-t$.
Set $n_0 = \min \{n\in\nat  :\; t\in \widetilde{E}((u_0,u_1,\ldots,u_n))\}$. Since $t\in \widetilde{E}((u_0,u_1,\ldots,u_{n_0}))$,
we have $(t,\vec{s}_{n_0+1})\in \R$. Then, according to $(ii)$, we have $(t,\vec{u}-u_{n_0})\in\R$, since
$\vec{u}-u_{n_0}\prec \vec{s}_{n_0+1}$, and also $(t,\vec{s}) \in\R$, since $\vec{s}\prec\vec{u}-u_{n_0}=\vec{u}-t$.
\end{proof}

\begin{lem}\label{lem:block-Ramsey}
Let $\Sigma=\{\alpha_n:\;n \in \mathbb{Z}^\ast\}$ be an alphabet, $\upsilon \notin \Sigma$ a variable, $\vec{k}=(k_n)_{n\in\mathbb{Z}^\ast}\subseteq \nat$ such that  $(k_n)_{n\in \nat},\;(k_{-n})_{n\in \nat}$ are increasing sequences, $\vec{w} = (w_n)_{n\in\nat} \in \widetilde{L}^\infty (\Sigma, \vec{k} ; \upsilon)$, and
$$\Pi = \{(t,\vec{s}): t\in \widetilde{L}_0(\Sigma, \vec{k} ; \upsilon),\vec{s} = (s_n)_{n\in\nat}\in \widetilde{L}^\infty (\Sigma, \vec{k} ; \upsilon)\;\text{with}\;\vec{s}\prec \vec{w}\;\text{and}\;t<_{\textsl{R}_1}s_1\}.$$

If a subset $\R$ of $\Pi$ satisfies
\begin{itemize}
\item[(i)] for every $(t,\vec{s})\in\Pi$, there exists $(t,\vec{s}_1)\in\R$ with
$\vec{s}_1 \prec \vec{s}$; and
\item[(ii)] for every $(t,\vec{s}) \in\R$ and $\vec{s}_1 \prec \vec{s}$, we have $(t,\vec{s}_1)\in\R$,
\end{itemize}
then there exists $\vec{u} \prec \vec{w}$, such that
$(t,\vec{s})\in \R$ for all $t\in \widetilde{EV}(\vec{u})$ and $\vec{s} \prec \vec{u}-t$.
\end{lem}

\begin{proof}
Let $u_0 = w_1$.
According to condition $(i)$, there exists $\vec{s}_1 = (s^1_n)_{n\in\nat} \in
\widetilde{L}^\infty (\Sigma, \vec{k} ; \upsilon)$
with $\vec{s}_1\prec\vec{w}-u_0$ such that
$(u_0,\vec{s}_1)\in \R$.
Let $u_1 = s^1_1$.
Of course, $u_0<_{\textsl{R}_1}u_1$ and $u_0, u_1 \in \widetilde{EV}(\vec{w})$. The proof can be continued analogously to the proof of Lemma~\ref{lem:block-Ramsey0}.
\end{proof}

We will now prove Theorem~\ref{thm:block-Ramsey}.

\begin{proof}[Proof of Theorem~\ref{thm:block-Ramsey}]
Let $\G\subseteq \widetilde{L}^{<\infty}(\Sigma, \vec{k}),$ $\F\subseteq \widetilde{L}^{<\infty}(\Sigma, \vec{k} ; \upsilon)$ and $\vec{w} \in \widetilde{L}^\infty (\Sigma, \vec{k} ; \upsilon)$. For $\xi=1$ the theorem is valid, according to Theorem~\ref{thm:block-Ramsey2}. Let $\xi>1$. Assume that the theorem is valid for every $\zeta <\xi$. Let $t\in \widetilde{L}_0(\Sigma, \vec{k} ; \upsilon)$ with $\min dom^+(t)=n$ and $\vec{s} = (s_n)_{n\in\nat}\in \widetilde{L}^\infty (\Sigma, \vec{k} ; \upsilon)$ with $\vec{s}\prec \vec{w}$ and $t<_{\textsl{R}_1}s_1.$ According to Proposition~\ref{justification}, there exists $\xi_n<\xi$ such that
\begin{center}
$\widetilde{L}^\xi(\Sigma, \vec{k}; \upsilon) (t) = \widetilde{L}^{\xi_n}(\Sigma, \vec{k}; \upsilon)\cap (\widetilde{L}^{<\infty}(\Sigma, \vec{k}; \upsilon) - t)$.
\end{center}
Using the induction hypothesis, there exists $\vec{s}_1 \prec\vec{s}$ such that

either $\widetilde{L}^{\xi_n}(\Sigma, \vec{k} ; \upsilon) \cap \widetilde{EV}^{<\infty}(\vec{s}_1)\subseteq \F(t)$,

 or $\widetilde{L}^{\xi_n}(\Sigma, \vec{k} ; \upsilon) \cap \widetilde{EV}^{<\infty}(\vec{s}_1)\subseteq \widetilde{L}^{<\infty}(\Sigma, \vec{k} ; \upsilon)\setminus \F(t)$.
\newline
Then $\vec{s}_1 \prec \vec{s} \prec \vec{w}$, and

either $\widetilde{L}^{\xi}(\Sigma, \vec{k} ; \upsilon)(t) \cap \widetilde{EV}^{<\infty}(\vec{s}_1)\subseteq \F(t),$

or $\widetilde{L}^{\xi}(\Sigma, \vec{k} ; \upsilon)(t) \cap \widetilde{EV}^{<\infty}(\vec{s}_1)\subseteq \widetilde{L}^{<\infty}(\Sigma, \vec{k} ; \upsilon)\setminus \F(t)$.
\begin{center} Let $\R_1= \{(t,\vec{s}):\; t\in \widetilde{L}_0(\Sigma, \vec{k} ; \upsilon)$,
$\vec{s} = (s_n)_{n\in\nat}\in \widetilde{L}^\infty (\Sigma, \vec{k} ; \upsilon),\vec{s}\prec \vec{w}$, $t<_{\textsl{R}_1}s_1,$ and \\  either $\widetilde{L}^{\xi}(\Sigma, \vec{k} ; \upsilon)(t) \cap \widetilde{EV}^{<\infty}(\vec{s})\subseteq \F(t)\;\;\;\;\;\;\;\;\;\;\;\;\;\;\;\;\;\;\;\; \;\;\;\;\;\;\;\;\; \;\;\;$ \\
 or $\widetilde{L}^{\xi}(\Sigma,\vec{k}; \upsilon)(t)\cap \widetilde{EV}^{<\infty}(\vec{s})\subseteq \widetilde{L}^{<\infty}(\Sigma,\vec{k}; \upsilon)\setminus\F(t)\}.\;\;\;\;\;\;\;\;\;\;\;\;$\end{center}

The family $\R_1$ satisfies the conditions $(i)$ (by the above arguments) and
$(ii)$ (obviously) of Lemma~\ref{lem:block-Ramsey}.
Hence, there exists $\vec{w}_1=(w_n^1)_{n\in \nat} \prec \vec{w}$ such that
$(t,\vec{s})\in \R_1$ for all $t\in \widetilde{EV}(\vec{w}_1)$ and $\vec{s} \prec \vec{w}_1-t$.

Analogously, defining the family

 \begin{center} $\R_2= \{(t,\vec{s}):\; t\in \widetilde{L}_0(\Sigma, \vec{k})$,
$\vec{s} = (s_n)_{n\in\nat}\in \widetilde{L}^\infty (\Sigma, \vec{k} ; \upsilon),\vec{s}\prec \vec{w_1}$, $t<_{\textsl{R}_1}s_1,$ and \\ either $\widetilde{L}^{\xi}(\Sigma, \vec{k})(t) \cap \widetilde{E}^{<\infty}(\vec{s})\subseteq \G(t)\;\;\;\;\;\;\;\;\;\;\;\;\;\;\;\;\;\;\;\;\;\;\;\;\;\;\;\;\; \;\;\;\;\;\;\;\;\; \;\;\;\;\;$ \\
or $\widetilde{L}^{\xi}(\Sigma,\vec{k})(t)\cap \widetilde{E}^{<\infty}(\vec{s})\subseteq \widetilde{L}^{<\infty}(\Sigma,\vec{k})\setminus\G(t)\}.\;\;\;\;\;\;\;\;\;\;\;\;\;\;
\;\;\;\;\;\;\;\;\;\;\;\;\;$\end{center}

\noindent We have that the family $\R_2$ satisfies the conditions $(i)$ and $(ii)$  of Lemma~\ref{lem:block-Ramsey0} for the sequence $\vec{w}_1.$ Hence, there exists $\vec{w}_2=(w_n^2)_{n\in \nat} \prec \vec{w_1}$ such that
$(t,\vec{s})\in \R_1$ for all $t\in \widetilde{E}(\vec{w}_2)$ and $\vec{s} \prec \vec{w}_2-t$. Let

$\G_1 = \{t\in \widetilde{E}(\vec{w}_2):\; \widetilde{L}^{\xi}(\Sigma, \vec{k})(t) \cap \widetilde{E}^{<\infty}(\vec{w}_2-t)\subseteq \G(t)\}$, and

$\F_1 = \{t\in \widetilde{EV}(\vec{w}_2):\; \widetilde{L}^{\xi}(\Sigma, \vec{k} ; \upsilon)(t) \cap \widetilde{EV}^{<\infty}(\vec{w}_2-t)\subseteq \F(t)\}$.

\noindent We use the induction hypothesis for $\xi=1$ (Theorem~\ref{thm:block-Ramsey2}).
Then, there exists a variable extraction $\vec{u} \prec \vec{w}_2$ of $\vec{w}_2$ such that:

either $\widetilde{E}(\vec{u})\subseteq \G_1$, or
$\widetilde{E}(\vec{u})\subseteq \widetilde{L}_0(\Sigma, \vec{k})\setminus \G_1;$ and,

 either $\widetilde{EV}(\vec{u})\subseteq \F_1$, or
$\widetilde{EV}(\vec{u})\subseteq \widetilde{L}_0(\Sigma, \vec{k} ; \upsilon)\setminus \F_1$.

\noindent Since $\vec{u} \prec \vec{w}_2$ we have that $\widetilde{E}(\vec{u})\subseteq \widetilde{E}(\vec{w}_2)$ and $\widetilde{EV}(\vec{u})\subseteq \widetilde{EV}(\vec{w}_2).$ Thus

either $\widetilde{L}^{\xi}(\Sigma, \vec{k})(t) \cap \widetilde{E}^{<\infty}(\vec{u}-t)\subseteq \G(t)$
for all $t\in \widetilde{E}(\vec{u})$,

 or $\widetilde{L}^{\xi}(\Sigma, \vec{k})(t) \cap \widetilde{E}^{<\infty}(\vec{u}-t)\subseteq \widetilde{L}^{<\infty}(\Sigma, \vec{k})\setminus \G(t)$ for all $t\in \widetilde{E}(\vec{u});$ and,

 either $\widetilde{L}^{\xi}(\Sigma, \vec{k} ; \upsilon)(t) \cap \widetilde{EV}^{<\infty}(\vec{u}-t)\subseteq \F(t)$
for all $t\in \widetilde{EV}(\vec{u})$,

 or $\widetilde{L}^{\xi}(\Sigma, \vec{k}; \upsilon )(t) \cap \widetilde{EV}^{<\infty}(\vec{u}-t)\subseteq \widetilde{L}^{<\infty}(\Sigma, \vec{k} ; \upsilon)\setminus \F(t)$ for all $t\in \widetilde{EV}(\vec{u})$.

\noindent Hence,

either $\widetilde{L}^\xi(\Sigma, \vec{k}) \cap \widetilde{E}^{<\infty}(\vec{u})\subseteq \G$, or
$\widetilde{L}^\xi(\Sigma, \vec{k}) \cap \widetilde{E}^{<\infty}(\vec{u})\subseteq \widetilde{L}^{<\infty}(\Sigma, \vec{k})\setminus \G;$ and,

 either $\widetilde{L}^\xi(\Sigma, \vec{k} ; \upsilon) \cap \widetilde{EV}^{<\infty}(\vec{u})\subseteq \F$, or
$\widetilde{L}^\xi(\Sigma, \vec{k} ; \upsilon) \cap \widetilde{EV}^{<\infty}(\vec{u})\subseteq \widetilde{L}^{<\infty}(\Sigma, \vec{k} ; \upsilon)\setminus \F$.
\end{proof}

The particular cases of Theorem~\ref{thm:block-Ramsey} for $\xi$ a finite ordinal and $\xi=\omega$ has the following statements.

\begin{cor}
[\textsf{Ramsey type partition theorem for variable $\omega$-$\mathbb{Z}^\ast$-located words}]
\label{cor:k-Ramsey}
Let $m\in\nat$, $\Sigma=\{\alpha_n:\;n \in \mathbb{Z}^\ast\}$ be an alphabet, $\upsilon \notin \Sigma$ a variable, $\vec{k}=(k_n)_{n\in\mathbb{Z}^\ast}\subseteq \nat$ such that $(k_n)_{n\in \nat},\;(k_{-n})_{n\in \nat}$ are increasing sequences, $r,s\in\nat$ and $\vec{w} \in \widetilde{L}^\infty (\Sigma, \vec{k} ; \upsilon)$. If $\widetilde{L}^m(\Sigma, \vec{k} ; \upsilon)=A_1\cup\ldots\cup A_r$ and $\widetilde{L}^m(\Sigma, \vec{k})=C_1\cup\ldots\cup C_s$, then there exists an extraction $\vec{u}\prec \vec{w}$ of $\vec{w}$ and $1\leq i_0 \leq r$, $1\leq j_0 \leq s$ such that
\begin{center}
$ \{(z_1,\ldots,z_m)\in \widetilde{L}^{<\infty}(\Sigma, \vec{k} ; \upsilon)) :\; z_1,\ldots,z_m\in \widetilde{EV}(\vec{u})\}\subseteq A_{i_0},$ and
\end{center}
\begin{center}
 $ \{(z_1,\ldots,z_m)\in \widetilde{L}^{<\infty}(\Sigma, \vec{k}) :\; z_1,\ldots,z_m\in \widetilde{E}(\vec{u})\}\subseteq C_{j_0}.\;\;\;\;\;\;\;\;\;\;\;\;\;\;\;$
\end{center}
\end{cor}

\begin{cor}
[\textsf{$\omega$-Ramsey type partition theorem for $\omega$-$\mathbb{Z}^\ast$-located words}]
\label{cor:kk-Ramsey}
Let $\Sigma=\{\alpha_n:\;n \in \mathbb{Z}^\ast\}$ be an alphabet, $\upsilon \notin \Sigma$ a variable,  $\vec{k}=(k_n)_{n\in\mathbb{Z}^\ast}\subseteq \nat$ such that $(k_n)_{n\in \nat},\;(k_{-n})_{n\in \nat}$ are increasing sequences, $r,s\in\nat$ and $\vec{w} \in \widetilde{L}^\infty (\Sigma, \vec{k} ; \upsilon)$. If $\widetilde{L}^{<\infty}(\Sigma, \vec{k} ; \upsilon)=A_1\cup\ldots\cup A_r$ and $\widetilde{L}^{<\infty}(\Sigma, \vec{k})=C_1\cup\ldots\cup C_s$, then there exists an extraction $\vec{u}\prec \vec{w}$ of $\vec{w}$ and $1\leq i_0 \leq r$, $1\leq j_0 \leq s$ such that
$$\{(z_1,\ldots,z_n)\in \widetilde{L}^{<\infty}(\Sigma, \vec{k} ; \upsilon)) :\; n\in \nat, \min dom^+(z_1) = n \text{ and}\; z_1,\ldots,z_n\in \widetilde{EV}(\vec{u})\}\subseteq A_{i_0},$$
\begin{center}
$ \{(z_1,\ldots,z_n)\in \widetilde{L}^{<\infty}(\Sigma, \vec{k})) :\; n\in \nat, \min dom^+(z_1) = n $ and $z_1,\ldots,z_n\in \widetilde{E}(\vec{u})\}\subseteq C_{j_0}.\;\;\;\;\;$
\end{center}
\end{cor}

For each countable ordinal $\xi$ Theorem~\ref{thm:block-Ramsey} has the following finitistic form (Corollary~\ref{thm:xiHJ}), which can be proved  analogously to Corollary~\ref{thm:HJ}, corresponding to the case $\xi=1$ and $l=1.$

\begin{note} Let $\Sigma=\{\alpha_n:\;n \in \mathbb{Z}^\ast\}$ be an alphabet, $\upsilon \notin \Sigma$ a variable and $\vec{k}=(k_n)_{n\in\mathbb{Z}^\ast}\subseteq\nat.$ For every countable ordinal $\xi\geq 1$ and $n\in \nat$ we set

$\widetilde{L}^{\xi}(\Sigma,\vec{k},n)=\{(w_1,\ldots,w_l)\in \widetilde{L}^{\xi}(\Sigma,\vec{k}):\; |dom(w_1)|+\ldots+|dom(w_l)|=n\}.$\end{note}

\begin{cor}\label{thm:xiHJ}
Let $\xi\geq 1$ be a countable ordinal, $\Sigma=\{\alpha_n:\;n\in \zat^\ast\}$ be an alphabet, $\upsilon\notin \Sigma$ a variable, $r,l\in \nat$ and $\vec{k}=(k_n)_{n\in \zat^\ast}\subseteq \nat$ such that $(k_n)_{n\in \nat},$ $(k_{-n})_{n\in \nat}$ are increasing sequences. Then, there exists $n_0\equiv n_0(\xi,r,l,\vec{k})\in \nat$ such that if $\widetilde{L}^{\xi}(\Sigma,\vec{k},n_0)=C_1\cup\ldots\cup C_r,$ there exists ${\bf{t}}=(t_1,\ldots,t_l)\in \widetilde{L}^{<\infty}(\Sigma,\vec{k};\upsilon)$ such that for some $1\leq i_0\leq r$ to satisfy \begin{center} $\widetilde{L}^{\xi}(\Sigma,\vec{k},n_0)\cap \widetilde{E}^{<\infty}({\bf{t}})\subseteq C_{i_0}.$\end{center}
\end{cor}

From Theorem~\ref{thm:block-Ramsey} we get immediately the corresponding theorem for $\omega$-located words.

\begin{note}
Let $\Sigma=\{\alpha_1, \alpha_2, \ldots \}$ be an infinite countable alphabet, $\upsilon \notin \Sigma$ a variable and $\vec{k}=(k_n)_{n\in\nat}\subseteq \nat$ an increasing sequence. We  define the \textit{finite orderly sequences of $\omega$-located words} over $\Sigma$ dominated by $\vec{k}$ as follows:

 $L^{<\infty}(\Sigma, \vec{k} ; \upsilon) = \{\bw = (w_1,\ldots,w_l) : l\in\nat, $
$w_1<_{\textsl{R}_2}\ldots <_{\textsl{R}_2} w_l\in L(\Sigma, \vec{k} ; \upsilon) \} \cup \{\emptyset \}$,

 $L^{<\infty}(\Sigma, \vec{k}) = \{\bw = (w_1,\ldots,w_l) : l\in\nat, $
$w_1<_{\textsl{R}_2}\ldots <_{\textsl{R}_2} w_l\in L(\Sigma, \vec{k}) \} \cup \{\emptyset \}$.

\noindent For every countable ordinal $\xi\geq 1,$ we set

$L^\xi(\Sigma, \vec{k}; \upsilon) = \{(w_1,\ldots,w_l)\in L^{<\infty}(\Sigma, \vec{k}; \upsilon) :
\{\min dom(w_1),\ldots, \min dom(w_l)\}\in\A_\xi \},$

$L^\xi(\Sigma, \vec{k}) = \{(w_1,\ldots,w_l)\in L^{<\infty}(\Sigma, \vec{k}) :
\{\min dom(w_1),\ldots, \min dom(w_l)\}\in\A_\xi \}.$

\noindent For $\vec{w}=(w_n)_{n\in\nat} \in L^\infty (\Sigma, \vec{k} ; \upsilon)$ we set:

$EV^{<\infty}(\vec{w}) = \{\bu=(u_1,\ldots,u_l)\in L^{<\infty} (\Sigma, \vec{k} ; \upsilon) : l\in\nat, u_1,\ldots,u_l\in EV(\vec{w})\}\cup \{\emptyset \},$ and

$E^{<\infty}(\vec{w}) = \{\bu=(u_1,\ldots,u_l)\in L^{<\infty} (\Sigma, \vec{k}) : l\in\nat, u_1,\ldots,u_l\in E(\vec{w})\}\cup \{\emptyset \}.$
\end{note}

\begin{cor}
[\textsf{Ramsey type partition theorem on Schreier families for $\omega$-located words}]
\label{cor:nat}
Let $\xi\geq 1$ be a countable ordinal, $\Sigma=\{\alpha_n :\;n \in \nat\}$ be an infinite countable alphabet, $\upsilon \notin \Sigma$ a variable and $\vec{k}=(k_n)_{n\in\nat}\subseteq \nat$ an increasing sequence. For every $\G\subseteq L^{<\infty}(\Sigma, \vec{k}),$ $\F \subseteq L^{<\infty}(\Sigma, \vec{k} ; \upsilon)$ and every infinite orderly sequence $\vec{w} \in L^\infty (\Sigma, \vec{k} ; \upsilon)$ of variable $\omega$-located words there exists an extraction $\vec{u}\prec \vec{w}$ of $\vec{w}$ such that:

either $L^\xi(\Sigma, \vec{k}) \cap E^{<\infty}(\vec{u})\subseteq \G$, or
$L^\xi(\Sigma, \vec{k}) \cap E^{<\infty}(\vec{u})\subseteq L^{<\infty}(\Sigma, \vec{k})\setminus \G,$ and

 either $L^\xi(\Sigma, \vec{k} ; \upsilon) \cap EV^{<\infty}(\vec{u})\subseteq \F$, or
$L^\xi(\Sigma, \vec{k} ; \upsilon) \cap EV^{<\infty}(\vec{u})\subseteq L^{<\infty}(\Sigma, \vec{k} ; \upsilon)\setminus \F$.
\end{cor}

\begin{proof}
Let $\G\subseteq L^{<\infty}(\Sigma, \vec{k}),$ $\F \subseteq L^{<\infty}(\Sigma, \vec{k} ; \upsilon)$ and $\vec{w}=(w_n)_{n\in \nat} \in L^\infty (\Sigma, \vec{k} ; \upsilon).$ We set $\widetilde{\Sigma}=\{\alpha_n:n\in \zat^\ast\},$ $\vec{k}_\ast=(\widetilde{k}_{n})_{n \in \zat^\ast}\subseteq \mathbb{N}$ and $\vec{w}_{\ast}=(\widetilde{w}_n)\in \widetilde{L}^\infty (\widetilde{\Sigma}, \vec{k}_\ast ; \upsilon)$ where $\alpha_{-n}=\alpha_n,$ $\widetilde{k}_{-n}=\widetilde{k}_{n}=k_{n}$ and $\widetilde{w}_n=\upsilon_{-n}\star w_n$ for every $n\in  \mathbb{N}.$ Let $\widetilde{\varphi}:\widetilde{L}^{<\infty}(\widetilde{\Sigma}\cup\{\upsilon\},\vec{k}_\ast)\rightarrow L^{<\infty}(\Sigma\cup\{\upsilon\},\vec{k})$ with \begin{center} $\widetilde{\varphi}(u_1,\ldots, u_l)=(\varphi(u_1),\ldots,\varphi(u_l)),$\end{center} where $\varphi:\widetilde{L}_0(\widetilde{\Sigma}\cup\{\upsilon\},\vec{k}_\ast)\rightarrow L(\Sigma\cup\{\upsilon\},\vec{k})$ is defined in Corollary~\ref{thm:block-Ramsey1}. Then, we apply Theorem~\ref{thm:block-Ramsey} for the families $\widetilde{\varphi}^{-1}(\G),$ $\widetilde{\varphi}^{-1}(\F),$ and the sequence $\vec{w}_{\ast}.$
\end{proof}

\section{Partition theorems for sequences of variable $\omega$-$\mathbb{Z}^\ast$-located words}

The main result of this Section is Theorem~\ref{block-NashWilliams2}, which strengthens the Ramsey type partition theorem on Schreier families for variable $\omega$-$\mathbb{Z}^\ast$-located words  (Theorem~\ref{thm:block-Ramsey}) in case the partition family is a tree. Specifically, given $\xi<\omega_1$ and a partition family $\F \subseteq \widetilde{L}^{<\infty}(\Sigma, \vec{k} ; \upsilon)$ of finite orderly sequences of variable $\omega$-$\mathbb{Z}^\ast$-located words over an alphabet $\Sigma=\{\alpha_n:\;n \in \mathbb{Z}^\ast\}$ dominated by a sequence $\vec{k}=(k_n)_{n\in\mathbb{Z}^\ast}\subseteq \nat$ such that $(k_n)_{n\in \nat},\;(k_{-n})_{n\in \nat}$ are increasing sequences,  Theorem~\ref{thm:block-Ramsey} provides no information on how to decide whether the $\xi$-homogeneous family falls in $\F$ or in its complement, while Theorem~\ref{block-NashWilliams2} in case the partition family $\F$ is a tree, does provide a criterion, in terms of a Cantor-Bendixson type index of $\F$.

 From Theorem~\ref{block-NashWilliams2} follows a partition theorem for the infinite orderly sequences of variable $\omega$-$\mathbb{Z}^\ast$-located words (Theorem~\ref{cor:blockNW}), which we may regard as a Nash-Williams type partition theorem for variable $\omega$-$\mathbb{Z}^\ast$-located words. Also, as a consequence of Theorem~\ref{block-NashWilliams2} we can get a strengthened Nash-Williams type partition theorem for variable $\omega$-located words analogous to Corollary~\ref{cor:tree2} (see Corollary~\ref{cor:tree3}) and consequently a theorem analogous to Theorem~\ref{cor:blockNW}.

\subsection*{Notation}
Let $\Sigma=\{\alpha_n:\;n \in \mathbb{Z}^\ast\}$ be an alphabet, $\upsilon \notin \Sigma$ a variable and $\vec{k}=(k_n)_{n\in\mathbb{Z}^\ast}\subseteq \nat$. A finite orderly sequence $\bw=(w_1,\ldots,w_l)\in \widetilde{L}^{<\infty}(\Sigma, \vec{k} ; \upsilon)$ is an \textit{initial segment} of $\bu=(u_1,\ldots,u_k)\in \widetilde{L}^{<\infty}(\Sigma, \vec{k} ; \upsilon)$ if and only if $l\leq k$ and $w_i= u_i$ for  every $i=1,\ldots,l$ and $\bw$ is an \textit{initial segment} of $\vec{u}=(u_n)_{n\in\nat}\in \widetilde{L}^{\infty}(\Sigma, \vec{k} ; \upsilon)$ if $w_i=u_i$ for all $i=1,\ldots,l$. In these cases we write $\bw\propto \bu$ and  $\bw\propto \vec{u}$, respectively, and we set $\bu\setminus \bw = (u_{l+1},\ldots,u_k)$ and $\vec{u}\setminus \bw = (u_n)_{n>l}$.

\begin{defn}\label{def:Fthin}
Let $\Sigma=\{\alpha_n:\;n \in \mathbb{Z}^\ast\}$ be an alphabet, $\upsilon \notin \Sigma$ a variable, $\vec{k}=(k_n)_{n\in\mathbb{Z}^\ast}\subseteq \nat$ such that $(k_n)_{n\in \nat},\;(k_{-n})_{n\in \nat}$ are increasing sequences and $\F\subseteq \widetilde{L}^{<\infty}(\Sigma, \vec{k} ; \upsilon)$.
\begin{itemize}
\item[(i)] $\F$ is {\em thin\/} if there are no elements $\bw,\bu\in\F$
with $\bw \propto \bu$ and $\bw\ne \bu$.
\item[(ii)] $\F^* = \{\bw \in \widetilde{L}^{<\infty}(\Sigma,\vec{k} ; \upsilon):\; \bw\propto \bu$ for some
$\bu\in \F\}\cup \{\emptyset\}$.
\item[(iii)] $\F$ is a {\em tree\/} if $\F^* = \F$.
\item[(iv)] $\F_* = \{\bw \in \widetilde{L}^{<\infty}(\Sigma, \vec{k} ; \upsilon):\; \bw\in \widetilde{EV}^{<\infty}(\bu)$
for some $\bu\in \F\}\cup\{\emptyset\}$.
\item[(v)] $\F$ is {\em hereditary\/} if $\F_* = \F$.
\end{itemize}
\end{defn}

\begin{prop}\label{prop:thinfamily}
Every family $\widetilde{L}^\xi(\Sigma, \vec{k} ; \upsilon)$, for $\xi<\omega_1$ is thin.
\end{prop}
\begin{proof}
It follows from the fact that the families $\A_\xi$ are thin (cf. \cite{F3})(which means that if
$s,t\in \A_\xi$ and $s\propto t$, then $s=t$).
\end{proof}

\begin{prop}\label{prop:canonicalrep}
Let $\xi$ be a non-zero countable ordinal number, $\Sigma=\{\alpha_n:\;n \in \mathbb{Z}^\ast\}$ be an alphabet, $\upsilon \notin \Sigma$ a variable and  $\vec{k}=(k_n)_{n\in\mathbb{Z}^\ast}\subseteq \nat$. Then

$(i)$ every infinite orderly sequence  $\vec{s} = (s_n)_{n\in\nat}\in \widetilde{L}^{\infty}(\Sigma, \vec{k} ; \upsilon)$ has canonical representation with respect to $\widetilde{L}^\xi(\Sigma, \vec{k} ; \upsilon)$, which means that there exists a unique strictly increasing sequence $(m_n)_{n\in\nat}$ in $\nat$
so that $(s_1,\ldots,s_{m_1}) \in \widetilde{L}^\xi(\Sigma, \vec{k} ; \upsilon)$ and $(s_{m_{n-1}+1},\ldots,s_{m_n}) \in \widetilde{L}^\xi(\Sigma, \vec{k} ; \upsilon)$ for every $n > 1$; and,

$(ii)$ every non-empty finite orderly sequence $\bs = (s_1,\ldots,s_k)\in \widetilde{L}^{<\infty}(\Sigma, \vec{k} ; \upsilon)$ has canonical representation with respect to $\widetilde{L}^\xi(\Sigma, \vec{k} ; \upsilon)$, so either $\bs\in (\widetilde{L}^\xi(\Sigma, \vec{k} ; \upsilon))^* \setminus \widetilde{L}^\xi(\Sigma, \vec{k} ; \upsilon)$ or there exists unique $n\in\nat$, and  $m_1, \ldots,m_n \in\nat$ with $m_1 < \ldots < m_n\leq k$ such that
either $(s_1,\ldots,s_{m_1}),\ldots,(s_{m_{n-1}+1},\ldots,s_{m_n}) \in \widetilde{L}^\xi(\Sigma, \vec{k} ; \upsilon)$ and $m_n = k$,
or $(s_1,\ldots,s_{m_1}),\ldots,$ $(s_{m_{n-1}+1},\ldots,s_{m_n}) \in \widetilde{L}^\xi(\Sigma, \vec{k} ; \upsilon)$,  $(s_{{m_n}+1},\ldots,s_k)\in (\widetilde{L}^\xi(\Sigma, \vec{k} ; \upsilon))^* \setminus \widetilde{L}^\xi(\Sigma, \vec{k} ; \upsilon)$.
\end{prop}
\begin{proof}
It follows from the fact that every non-empty increasing sequence (finite or infinite) in $\nat$
has canonical representation with respect to $\A_\xi$ (cf. \cite{F3}) and that
the family $\widetilde{L}^\xi(\Sigma, \vec{k} ; \upsilon)$ is thin (Proposition~\ref{prop:thinfamily}).
\end{proof}

Now, using Proposition~\ref{prop:canonicalrep}, we can give an alternative description of the second horn of the dichotomy described in Theorem~\ref{thm:block-Ramsey} in case the partition family is a tree.

\begin{prop}\label{prop:tree}
Let $\xi\geq 1$ be a countable ordinal, $\Sigma=\{\alpha_n:\;n \in \mathbb{Z}^\ast\}$ be an alphabet, $\upsilon \notin \Sigma$ a variable, $\vec{k}=(k_n)_{n\in\mathbb{Z}^\ast}\subseteq \nat$ such that $(k_n)_{n\in \nat},\;(k_{-n})_{n\in \nat}$ are increasing sequences, $\vec{u}\in \widetilde{L}^{\infty}(\Sigma, \vec{k}; \upsilon)$ and $\F\subseteq \widetilde{L}^{<\infty}(\Sigma, \vec{k} ; \upsilon)$ be a tree. Then

$\widetilde{L}^\xi(\Sigma, \vec{k} ; \upsilon) \cap \widetilde{EV}^{<\infty}(\vec{u})\subseteq \widetilde{L}^{<\infty}(\Sigma, \vec{k} ; \upsilon)\setminus \F\;\;$ if and only if

$\F\cap \widetilde{EV}^{<\infty}(\vec{u}) \subseteq (\widetilde{L}^\xi(\Sigma, \vec{k} ; \upsilon))^* \setminus \widetilde{L}^\xi(\Sigma, \vec{k} ; \upsilon)$.
\end{prop}

\begin{proof}
Let $\widetilde{L}^\xi(\Sigma, \vec{k} ; \upsilon) \cap \widetilde{EV}^{<\infty}(\vec{u})\subseteq \widetilde{L}^{<\infty}(\Sigma, \vec{k} ; \upsilon)\setminus \F$ and $\bs = (s_1,\ldots,s_k)\in \F\cap \widetilde{EV}^{<\infty}(\vec{u})$.
Then $\bs$ has canonical representation with respect to $\widetilde{L}^\xi(\Sigma, \vec{k} ; \upsilon)$
(Proposition~\ref{prop:canonicalrep}), hence
either $\bs \in (\widetilde{L}^\xi(\Sigma, \vec{k} ; \upsilon))^* \setminus \widetilde{L}^\xi(\Sigma, \vec{k} ; \upsilon)$, as required, or there exists $\bs_1\in \widetilde{L}^\xi(\Sigma, \vec{k} ; \upsilon)$ such that $\bs_1 \propto \bs$.
The second case is impossible.
Indeed, since $\F$ is a tree and $\bs\in \F\cap \widetilde{EV}^{<\infty}(\vec{u})$, we have
$\bs_1 \in \F\cap \widetilde{EV}^{<\infty}(\vec{u}) \cap \widetilde{L}^\xi(\Sigma, \vec{k} ; \upsilon)$; a contradiction to our
assumption.
Hence, $\F\cap \widetilde{EV}^{<\infty}(\vec{u}) \subseteq (\widetilde{L}^\xi(\Sigma, \vec{k} ; \upsilon))^* \setminus \widetilde{L}^\xi(\Sigma, \vec{k} ; \upsilon)$.
\end{proof}

\begin{defn}\label{def:aclosed}
Let $\Sigma=\{\alpha_n:\;n \in \mathbb{Z}^\ast\}$ be an alphabet, $\upsilon \notin \Sigma$ a variable and $\vec{k}=(k_n)_{n\in\mathbb{Z}^\ast}\subseteq \nat$. We set

\begin{center}$D=\big\{(n,\alpha):\;$ either $n \in \mathbb{Z}^-$ and $\alpha \in \{\upsilon, \alpha_{-k_n}, \ldots, \alpha_{-1}\}$\\ or $ n\in \nat$ and  $\alpha \in \{\upsilon, \alpha_1, \ldots, \alpha_{k_{n}} \} \big\}.\;\;\;\;\;\;\;\;\;\;\;\;\;$
\end{center}
Note that $D$ is a countable set. Let $[D]^{<\omega}$ be the set of all finite subsets of $D$. Identifying every $s\in \widetilde{L}(\Sigma, \vec{k}; \upsilon)$ with the corresponding element of $[D]^{<\omega}$ and consequently every $\bs\in \widetilde{L}^{<\infty}(\Sigma, \vec{k}; \upsilon)$ and every $\vec{s}\in \widetilde{L}^{\infty}(\Sigma, \vec{k}; \upsilon)$) with their characteristic functions $x_{\sigma(\bs)} \in \{0,1\}^{[D]^{<\omega}}$ and $x_{\sigma(\vec{s})} \in \{0,1\}^{[D]^{<\omega}}$ respectively (where $\sigma(\bs) = \{s_1,\ldots,s_k\}$ for $\bs = (s_1,\ldots,s_k)\in \widetilde{L}^{<\infty}(\Sigma, \vec{k}; \upsilon)$, $\sigma(\vec{s}) = \{s_n : n\in\nat\}$ for  $\vec{s} = (s_n)_{n\in\nat}\in \widetilde{L}^{\infty}(\Sigma, \vec{k}; \upsilon)$ and $\sigma(\emptyset) = \emptyset$), we say that a family $\F\subseteq \widetilde{L}^{<\infty}(\Sigma, \vec{k}; \upsilon)$ is {\em pointwise closed\/} if and only if the family $\{x_{\sigma(\bs)} :\bs \in \F\}$ is closed in the product topology (equivalently by the pointwise convergence topology) of $\{0,1\}^{[D]^{<\omega}}$ and in analogy a family $\U\subseteq \widetilde{L}^{\infty}(\Sigma, \vec{k}; \upsilon)$ is {\em pointwise closed\/} if and only if $\{x_{\sigma(\vec{s})} :\;\vec{s}\in \U\}$ is closed in $\{0,1\}^{[D]^{<\omega}}$ with the product topology.
\end{defn}

\begin{prop}\label{prop:finitefamily}
Let $\Sigma=\{\alpha_n:\;n \in \mathbb{Z}^\ast\}$ be an alphabet, $\upsilon \notin \Sigma$ a variable and $\vec{k}=(k_n)_{n\in\mathbb{Z}^\ast}\subseteq \nat$ such that $(k_n)_{n\in \nat},\;(k_{-n})_{n\in \nat}$ are increasing sequences.

{\rm (i)} If $\F\subseteq \widetilde{L}^{<\infty}(\Sigma, \vec{k}; \upsilon)$ is a tree, then $\F$ is pointwise closed if and only if there does not exist an infinite sequence $(\bs_n)_{n\in\nat}$ in $\F$
such that $\bs_n \propto \bs_{n+1}$ and $\bs_n\ne \bs_{n+1}$ for all $n\in\nat$.

{\rm (ii)} If $\F\subseteq \widetilde{L}^{<\infty}(\Sigma, \vec{k}; \upsilon)$  is hereditary, then $\F$ is pointwise closed if and only if there does not exist $\vec{s}\in \widetilde{L}^{\infty}(\Sigma, \vec{k} ; \upsilon)$ such that $\widetilde{EV}^{<\infty}(\vec{s}) \subseteq \F$.

{\rm (iii)} The hereditary family $(\widetilde{L}^\xi(\Sigma, \vec{k} ; \upsilon) \cap \widetilde{EV}^{<\infty}(\vec{u}) )_*$ is pointwise closed for every countable ordinal $\xi$ and $\vec{u} \in \widetilde{L}^{\infty}(\Sigma, \vec{k} ; \upsilon)$.
\end{prop}
\begin{proof}
This follows directly from the definitions (for details cf. \cite{F3}, \cite{FN1}).
\end{proof}

Let $\vec{s} \in \widetilde{L}^{\infty}(\Sigma, \vec{k} ; \upsilon)$. For a hereditary and pointwise closed family $\F\subseteq \widetilde{L}^{<\infty}(\Sigma, \vec{k}; \upsilon)$ we will define the strong Cantor-Bendixson index $sO_{\vec{s}}(\F)$ of $\F$ with respect to $\vec{s} $.

\begin{defn}\label{def:Cantor-Bendix}
Let $\Sigma=\{\alpha_n:\;n \in \mathbb{Z}^\ast\}$ be an alphabet, $\upsilon \notin \Sigma$ a variable, $\vec{k}=(k_n)_{n\in\mathbb{Z}^\ast}\subseteq \nat$ such that $(k_n)_{n\in \nat},\;(k_{-n})_{n\in \nat}$ are increasing sequences, $\vec{s}\in \widetilde{L}^{\infty}(\Sigma, \vec{k} ; \upsilon)$ and let $\F\subseteq \widetilde{L}^{<\infty}(\Sigma, \vec{k}; \upsilon)$ be a hereditary and pointwise closed family. For every $\xi <\omega_1$ we define the families $(\F)_{\vec{s}}^\xi$ inductively as follows:
\newline
For every $\bw = (w_1,\ldots,w_l)\in \F\cap \widetilde{EV}^{<\infty}(\vec{s})$ we set

$A_{\bw}=\{t\in \widetilde{EV}(\vec{s}) :\;(w_1,\ldots,w_l,t)\notin \F\}$ and $A_{\emptyset}=\{t\in \widetilde{EV}(\vec{s}) :(t)\notin \F\} $.
\newline
We define
\begin{center}
$(\F)_{\vec{s}}^1 = \{\bw \in \F\cap \widetilde{EV}^{<\infty}(\vec{s})\cup \{\emptyset\}
 : A_{\bw}$ does not contain an infinite orderly sequence$\}$.
\end{center}
It is easy to verify that $(\F)_{\vec{s}}^1$ is hereditary, hence it is pointwise closed since $\F$ is pointwise closed (Proposition~\ref{prop:finitefamily}).
So, we can define for every $\xi >1$ the $\xi$-derivatives of $\F$
recursively as follows:
\begin{itemize}
\item[]  $ (\F)_{\vec{s}}^{\zeta +1} = ((\F)_{\vec{s}}^\zeta  )_{\vec{s}}^1$ for all $ \zeta <\omega_1$, and
\item[] $(\F)_{\vec{s}}^\xi = \bigcap_{\beta<\xi} (\F)_{\vec{s}}^\beta$ for $\xi$ a limit ordinal.
\end{itemize}

The  {\em strong Cantor-Bendixson index} $sO_{\vec{s}}(\F)$ of $\F$ on $\vec{s}$ is the smallest countable ordinal $\xi$ such that $(\F)_{\vec{s}}^\xi = \emptyset$.
\end{defn}

\begin{remark}\label{rem:Cantor-Bendix}
Let $\vec{s}\in \widetilde{L}^{\infty}(\Sigma, \vec{k}; \upsilon)$ and let $\F_1, \R_1,\subseteq \widetilde{L}^{<\infty}(\Sigma, \vec{k}; \upsilon)$ be hereditary and pointwise closed families.
\begin{itemize}
\item[(i)] $sO_{\vec{s}}(\F_1)$ is a countable successor ordinal less than or
equal to the ``usual'' Cantor-Bendixson index $O(\F_1)$ of $\F_1$ into
$\{0,1\}^{[D]^{<\omega}}$ (cf. \cite{KM}).

\item[(ii)] $sO_{\vec{s}} (\F_1\cap \widetilde{EV}^{<\infty}(\vec{s})) =$ $sO_{\vec{s}}(\F_1)$.

\item[(iii)] $sO_{\vec{s}}(\F_1) \leq sO_{\vec{s}}(\R_1)$ if $\F_1\subseteq \R_1$.

\item[(iv)] If $\vec{s}_1 \prec \vec{s}$ and $\bw \in (\F_1)_{\vec{s}}^\xi$, then for every $\bw_1\in \widetilde{EV}^{<\infty}(\vec{s}_1)$ such that $\sigma(\bw_1)=\sigma(\bw) \cap \widetilde{EV}(\vec{s}_1) $ we have that $\bw_1\in (\F_1)_{\vec{s}_1}^\xi$ , since $\widetilde{EV}(\vec{s}_1)\subseteq \widetilde{EV}(\vec{s})$.

\item[(v)] If $\vec{s}_1\prec \vec{s}$, then $sO_{\vec{s}_1} (\F_1) \geq sO_{\vec{s}}(\F_1)$,
according to (iv).

\item[(vi)] If $\sigma(\vec{s}_1)\setminus \sigma(\vec{s})$ is a finite set, then $sO_{\vec{s}_1}(\F_1)
\geq $$sO_{\vec{s}}(\F_1)$.
\end{itemize}
\end{remark}

The corresponding strong Cantor-Bendixson index to $\widetilde{L}^\xi(\Sigma, \vec{k} ; \upsilon)$ is equal to $\xi+1$, with respect any sequence $\vec{s}\in \widetilde{L}^{\infty}(\Sigma, \vec{k} ; \upsilon)$.

\begin{prop}
\label{prop:Cantor-Bendix}
Let $\xi<\omega_1$ be an ordinal, $\Sigma=\{\alpha_n:\;n \in \mathbb{Z}^\ast\}$ be an alphabet, $\upsilon \notin \Sigma$ a variable, $\vec{k}=(k_n)_{n\in\mathbb{Z}^\ast}\subseteq \nat$ such that $(k_n)_{n\in \nat},\;(k_{-n})_{n\in \nat}$ are increasing sequences and $\vec{s}\in \widetilde{L}^{\infty}(\Sigma, \vec{k} ; \upsilon)$. Then,
\begin{center}
 $sO_{\vec{s}_1}\Big((\widetilde{L}^\xi(\Sigma, \vec{k} ; \upsilon) \cap \widetilde{EV}^{<\infty}(\vec{s}))_*\Big)= \xi+1$ for every $\vec{s}_1\prec \vec{s}$.
\end{center}
\end{prop}

\begin{proof}
The family $(\widetilde{L}^\xi(\Sigma, \vec{k}  ; \upsilon) \cap \widetilde{EV}^{<\infty}(\vec{s}))_*$ is hereditary and pointwise closed (Proposition~\ref{prop:finitefamily}). We will prove by induction on $\xi$ that
$\Big((\widetilde{L}^\xi(\Sigma, \vec{k}  ; \upsilon) \cap \widetilde{EV}^{<\infty}(\vec{s}))_*\Big)_{\vec{s}_1}^\xi = \{\emptyset\}$ for every $\vec{s}_1\prec \vec{s}$ and $\xi<\omega_1$.
Since
$(\widetilde{L}^1(\Sigma, \vec{k}; \upsilon) \cap \widetilde{EV}^{<\infty}(\vec{s}))_* = \{(t) :\; t\in \widetilde{EV}(\vec{s})\} \cup \{\emptyset\}$, we have $\Big((\widetilde{L}^1(\Sigma, \vec{k} ; \upsilon) \cap \widetilde{EV}^{<\infty}(\vec{s}))_*\Big)_{\vec{s}_1}^1 = \{\emptyset\}$ for every $\vec{s}_1\prec \vec{s}$.

Let $\xi>1$ and assume that
$\Big((\widetilde{L}^\zeta(\Sigma, \vec{k} ; \upsilon) \cap \widetilde{EV}^{<\infty}(\vec{s}))_* \Big)_{\vec{s}_1}^\zeta = \{\emptyset\}$ for every $\vec{s}_1\prec \vec{s}$ and $\zeta <\xi$.
For every $t\in \widetilde{EV}(\vec{s})$ with $\min dom^+(t)=n$, according to
Proposition~\ref{justification}, we
\begin{center}
have $(\widetilde{L}^\xi(\Sigma, \vec{k} ; \upsilon)\cap \widetilde{EV}^{<\infty}(\vec{s})) (t) = \widetilde{L}^{\xi_n}(\Sigma, \vec{k} ; \upsilon)\cap \widetilde{EV}^{<\infty}(\vec{s}-t)$ for some $\xi_n <\xi.$
\end{center}
Hence, for every $\vec{s}_1\prec \vec{s}$ and $t\in \widetilde{EV}(\vec{s}_1)$ with $\min dom^+(t)=n$ we have that

$\Big((\widetilde{L}^\xi(\Sigma, \vec{k} ; \upsilon)\cap \widetilde{EV}^{<\infty}(\vec{s}))(t)_* \Big)_{\vec{s}_1}^{\xi_n}
= \Big((\widetilde{L}^{\xi_n}(\Sigma, \vec{k} ; \upsilon)\cap  \widetilde{EV}^{<\infty}(\vec{s}-t))_* \Big)_{\vec{s}_1}^{\xi_n}
= \{\emptyset\}.$
\newline
This gives that
$(t) \in \Big((\widetilde{L}^\xi(\Sigma, \vec{k} ; \upsilon)\cap \widetilde{EV}^{<\infty}(\vec{s}))_* \Big)_{\vec{s}_1}^{\xi_n}$.
So, $\emptyset \in \Big((\widetilde{L}^\xi(\Sigma, \vec{k} ; \upsilon)\cap \widetilde{EV}^{<\infty}(\vec{s}))_* \Big)_{\vec{s}_1}^\xi$, since if $\xi = \zeta+1$, then $(t) \in \Big((\widetilde{L}^\xi(\Sigma, \vec{k} ; \upsilon)\cap \widetilde{EV}^{<\infty}(\vec{s}))_* \Big)_{\vec{s}_1}^\zeta$
for every $t\in \widetilde{EV}(\vec{s}_1)$ and if
$\xi$ is a limit ordinal, then $\emptyset\in \Big((\widetilde{L}^\xi(\Sigma, \vec{k} ; \upsilon)\cap \widetilde{EV}^{<\infty}(\vec{s}))_* \Big)_{\vec{s}_1}^{\xi_n}$ for every $n\in\nat$
and $\sup \xi_n  =\xi$.

If $\{\emptyset\} \ne \Big((\widetilde{L}^\xi(\Sigma, \vec{k} ; \upsilon)\cap \widetilde{EV}^{<\infty}(\vec{s}))_* \Big)_{\vec{s}_1}^\xi$ for $\vec{s}_1\prec \vec{s}$,
then, according to Lemma 2.8 in [F5], using the fact that the set $\{y\in \widetilde{EV}(\vec{s}_1):\;\max dom^+(y)<n_0\}$ for $n_0\in \nat$ is finite, can be constructed $\vec{s}_2\prec \vec{s}_1$ and  $s\in \widetilde{EV}(\vec{s}_2)$ such that
$\Big((\widetilde{L}^\xi(\Sigma, \vec{k} ; \upsilon)\cap \widetilde{EV}^{<\infty}(\vec{s}))(s)_* \Big)_{\vec{s}_2}^\xi \ne \emptyset$, a contradiction to the induction hypothesis.
Hence, $\Big((\widetilde{L}^\xi(\Sigma, \vec{k} ; \upsilon)\cap \widetilde{EV}^{<\infty}(\vec{s}))_* \Big)_{\vec{s}_1}^\xi = \{\emptyset\} $ and
$sO_{\vec{s}_1}((\widetilde{L}^\xi(\Sigma, \vec{k} ; \upsilon) \cap \widetilde{EV}^{<\infty}(\vec{s}))_* )
= \xi+1$ for every $\xi<\omega_1$.
\end{proof}

\begin{cor}\label{cor:tree}
Let $\xi_1,\xi_2$ be countable ordinals with $\xi_1<\xi_2$ and $\vec{w}\in \widetilde{L}^\infty (\Sigma, \vec{k} ; \upsilon)$. Then there exists $\vec{u}\prec \vec{w}$ such that
\begin{center} $(\widetilde{L}^{\xi_1}(\Sigma, \vec{k} ; \upsilon))_* \cap \widetilde{EV}^{<\infty}(\vec{u})\subseteq (\widetilde{L}^{\xi_2}(\Sigma, \vec{k} ; \upsilon))^* \setminus \widetilde{L}^{\xi_2}(\Sigma, \vec{k} ; \upsilon)$.\end{center}
\end{cor}

\begin{proof}
The family $(\widetilde{L}^{\xi_1}(\Sigma, \vec{k}; \upsilon))_* \subseteq \widetilde{L}^{<\infty}(\Sigma, \vec{k} ; \upsilon)$ is a tree. According to Theorem~\ref{thm:block-Ramsey} and Proposition~\ref{prop:tree} there exists $\vec{u}\prec\vec{w}$ such that:

either $\widetilde{L}^{\xi_2}(\Sigma, \vec{k} ; \upsilon) \cap \widetilde{EV}^{<\infty}(\vec{u})\subseteq (\widetilde{L}^{\xi_1}(\Sigma, \vec{k} ; \upsilon))_*$,

or $(\widetilde{L}^{\xi_1}(\Sigma, \vec{k} ; \upsilon))_* \cap \widetilde{EV}^{<\infty}(\vec{u}) \subseteq (\widetilde{L}^{\xi_2}(\Sigma, \vec{k} ; \upsilon))^* \setminus \widetilde{L}^{\xi_2}(\Sigma, \vec{k} ; \upsilon)$.
\newline
The first alternative of the dichotomy is impossible, since, according to Proposition~\ref{prop:Cantor-Bendix}
\newline
$\xi_2 +1 =$ $sO_{\vec{u}}((\widetilde{L}^{\xi_2}(\Sigma, \vec{k} ; \upsilon) \cap \widetilde{EV}^{<\infty}(\vec{u}))_*) \le$ $sO_{\vec{u}} ((\widetilde{L}^{\xi_1}(\Sigma, \vec{k} ; \upsilon))_*) = \xi_1 +1$.
\end{proof}

The following Theorem~\ref{block-NashWilliams2}, the main result in this Section, refines  Theorem~\ref{thm:block-Ramsey} in case the partition family is a tree.

\begin{defn}\label{def:blockNW}
Let $\F \subseteq \widetilde{L}^{<\infty}(\Sigma, \vec{k} ; \upsilon)$ be a family of finite orderly sequences of variable $\omega$-$\mathbb{Z}^\ast$-located words over an alphabet $\Sigma=\{\alpha_n:\;n\in \zat^\ast\},$ dominated by a two-sited sequence $\vec{k}=(k_n)_{n\in\mathbb{Z}^\ast}\subseteq \nat$ such that $(k_n)_{n\in \nat},\;(k_{-n})_{n\in \nat}$ are increasing sequences. We set

$\F_h = \{\bw \in \F:\; \widetilde{EV}^{<\infty}(\bw)\subseteq\F \}\cup \{\emptyset\}$.

Of course, $\F_h$ is the largest subfamily of $\F\cup \{\emptyset\}$ which is hereditary.
\end{defn}

\begin{thm}
\label{block-NashWilliams2}
Let $\Sigma=\{\alpha_n:\;n \in \mathbb{Z}^\ast\}$ be an alphabet, $\upsilon \notin \Sigma$ a variable, $\vec{k}=(k_n)_{n\in\mathbb{Z}^\ast}\subseteq \nat$ such that $(k_n)_{n\in \nat},\;(k_{-n})_{n\in \nat}$ are increasing sequences, $\F \subseteq \widetilde{L}^{<\infty}(\Sigma, \vec{k} ; \upsilon)$ a family which is a tree and $\vec{w} \in \widetilde{L}^\infty (\Sigma, \vec{k} ; \upsilon).$ Then we have the following cases:
\newline
\noindent {\bf [Case 1]} The family $\F_h\cap \widetilde{EV}^{<\infty}(\vec{w})$ is not pointwise closed.

Then, there exists $\vec{u}\prec \vec{w}$ such that
$\widetilde{EV}^{<\infty}(\vec{u})\subseteq \F$.
\newline
\noindent {\bf [Case 2]}
The family $\F_h \cap \widetilde{EV}^{<\infty}(\vec{w})$ is pointwise closed.

Then, setting
\begin{center}
$\zeta_{\vec{w}}^{\F} = \sup \{sO_{\vec{u}} (\F_h) :\; \vec{u}\prec \vec{w}\},$
\end{center}
which is a countable ordinal, the following subcases obtain:
\begin{itemize}
\item[2(i)] If $\xi+1 <\zeta_{\vec{w}}^{\F}$, then there exists $\vec{u}\prec \vec{w}$
such that
\begin{center}
$\widetilde{L}^\xi(\Sigma, \vec{k} ; \upsilon) \cap \widetilde{EV}^{<\infty}(\vec{u})\subseteq \F;$
\end{center}
\item[2(ii)] if $\xi+1>\xi>\zeta_{\vec{w}}^{\F}$, then for every $\vec{w}_1 \prec \vec{w}$ there exists
$\vec{u}\prec \vec{w}_1$ such that

$\widetilde{L}^\xi(\Sigma, \vec{k} ; \upsilon) \cap \widetilde{EV}^{<\infty}(\vec{u})\subseteq \widetilde{L}^{<\infty}(\Sigma, \vec{k} ; \upsilon)\setminus \F;$

$($equivalently $\F\cap \widetilde{EV}^{<\infty}(\vec{u}) \subseteq (\widetilde{L}^\xi(\Sigma, \vec{k} ; \upsilon))^* \setminus \widetilde{L}^\xi(\Sigma, \vec{k}; \upsilon))$ and
\item[2(iii)] if $\xi+1 = \zeta_{\vec{w}}^{\F}$ or $\xi = \zeta_{\vec{w}}^{\F}$, then there exists $\vec{u}\prec \vec{w}$ such that
\item[{}] either $\widetilde{L}^\xi(\Sigma, \vec{k} ; \upsilon) \cap \widetilde{EV}^{<\infty}(\vec{u})\subseteq \F, $ or $
\widetilde{L}^\xi(\Sigma, \vec{k} ; \upsilon) \cap \widetilde{EV}^{<\infty}(\vec{u})\subseteq \widetilde{L}^{<\infty}(\Sigma, \vec{k} ; \upsilon)\setminus \F.$
\end{itemize}
\end{thm}
\begin{proof}
\noindent{\bf [Case 1]}
If the hereditary family $\F_h \cap \widetilde{EV}^{<\infty}(\vec{w})$ is not pointwise
closed, then, according to Proposition~\ref{prop:finitefamily}, there exists $\vec{u}\in \widetilde{L}^\infty (\Sigma, \vec{k}; \upsilon)$ such that $\widetilde{EV}^{<\infty}(\vec{u})\subseteq \F_h\cap \widetilde{EV}^{<\infty}(\vec{w})\subseteq \F$. Of course, $\vec{u}\prec \vec{w}$.

\noindent{\bf [Case 2]}
If the hereditary family $\F_h \cap \widetilde{EV}^{<\infty}(\vec{w})$ is pointwise closed, then $\zeta_{\vec{w}}^{\F}$
is a countable ordinal, since the ``usual'' Cantor-Bendixson
index $O(\F_h)$ of $\F_h$ into $\{0,1\}^{[D]^{<\omega}}$ is countable
(Remark~\ref{rem:Cantor-Bendix}(i)) and also $sO_{\vec{u}}(\F_h)\leq O(\F_h)$ for every $\vec{u}\prec \vec{w}$.

2(i)
Let $\xi+1<\zeta_{\vec{w}}^{\F}$. Then there exists $\vec{u}_1\prec\vec{w}$ such that
$\xi+1 <$ $sO_{\vec{u}_1}(\F_h)$.
According to Theorem~\ref{thm:block-Ramsey} and Proposition~\ref{prop:tree},
there exists $\vec{u}\prec \vec{u}_1$ such that

either $\widetilde{L}^\xi(\Sigma, \vec{k} ; \upsilon) \cap \widetilde{EV}^{<\infty}(\vec{u})\subseteq \F_h\subseteq \F$,

or $\F_h \cap \widetilde{EV}^{<\infty}(\vec{u}) \subseteq (\widetilde{L}^\xi(\Sigma, \vec{k} ; \upsilon))^* \setminus \widetilde{L}^\xi(\Sigma, \vec{k} ; \upsilon)\subseteq (\widetilde{L}^\xi(\Sigma, \vec{k} ; \upsilon))^* \subseteq (\widetilde{L}^\xi(\Sigma, \vec{k} ; \upsilon))_*$.
\newline
The second alternative is impossible.
Indeed, if $\F_h \cap \widetilde{EV}^{<\infty}(\vec{u}) \subseteq (\widetilde{L}^\xi(\Sigma, \vec{k} ; \upsilon))_*$, then, according
to Remark ~\ref{rem:Cantor-Bendix} and Proposition~\ref{prop:Cantor-Bendix},

$sO_{\vec{u}_1}(\F_h) \leq$ $sO_{\vec{u}} (\F_h) =$ $sO_{\vec{u}}(\F_h\cap \widetilde{EV}^{<\infty}(\vec{u})) \leq$ $ sO_{\vec{u}}((\widetilde{L}^\xi(\Sigma, \vec{k} ; \upsilon))_*) = \xi+1$;
\newline
a contradiction. Hence, $\widetilde{L}^\xi(\Sigma, \vec{k} ; \upsilon) \cap \widetilde{EV}^{<\infty}(\vec{u})\subseteq \F$.

2(ii)
Let $\xi +1>\xi >\zeta_{\vec{w}}^{\F}$ and $\vec{w}_1\prec \vec{w}$. According to Theorem~\ref{thm:block-Ramsey}, there exists $\vec{u_1}\prec \vec{w}_1$ such that

 either $\widetilde{L}^{\zeta_{\vec{w}}^{\F}}(\Sigma, \vec{k} ; \upsilon) \cap \widetilde{EV}^{<\infty}(\vec{u}_1)\subseteq \F_h$,

  or
$\widetilde{L}^{\zeta_{\vec{w}}^{\F}}(\Sigma, \vec{k} ; \upsilon) \cap \widetilde{EV}^{<\infty}(\vec{u_1})\subseteq \widetilde{L}^{<\infty}(\Sigma, \vec{k} ; \upsilon)\setminus \F_h$.
\newline
The first alternative is impossible.
Indeed, if $\widetilde{L}^{\zeta_{\vec{w}}^{\F}}(\Sigma, \vec{k} ; \upsilon) \cap \widetilde{EV}^{<\infty}(\vec{u}_1)\subseteq \F_h$, then, according
to Remark ~\ref{rem:Cantor-Bendix} and Proposition~\ref{prop:Cantor-Bendix},
we have that

${\zeta_{\vec{w}}^{\F}}+1 =$ $sO_{\vec{u}_1} ((\widetilde{L}^{\zeta_{\vec{w}}^{\F}}(\Sigma, \vec{k} ; \upsilon) \cap \widetilde{EV}^{<\infty}(\vec{u}_1))_*) \le$ $sO_{\vec{u}_1}
(\F_h) \le \zeta_{\vec{w}}^{\F};$
\newline
a contradiction. Hence,

$(1)\;\;\;\;\;\;\;\;\; \;\;\;\;\;\;\;\;\;\;\;\;\;\;\;\;\;\;\widetilde{L}^{\zeta_{\vec{w}}^{\F}}(\Sigma, \vec{k} ; \upsilon) \cap \widetilde{EV}^{<\infty}(\vec{u}_1)\subseteq \widetilde{L}^{<\infty}(\Sigma, \vec{k} ; \upsilon)\setminus \F_h\ .\;\;\;\;\;\;\;\;\;\;\;\;\;\;\;\;\;\;\;\;\;\;\;\;\;\;\;$

\noindent According to Theorem~\ref{thm:block-Ramsey},
there exists $\vec{u}\prec \vec{u}_1$ such that

 either $\widetilde{L}^\xi(\Sigma, \vec{k} ; \upsilon) \cap \widetilde{EV}^{<\infty}(\vec{u})\subseteq \F$,

 or
$\widetilde{L}^\xi(\Sigma, \vec{k} ; \upsilon) \cap \widetilde{EV}^{<\infty}(\vec{u})\subseteq \widetilde{L}^{<\infty}(\Sigma, \vec{k} ; \upsilon)\setminus \F$.
\newline
We claim that the first alternative does not hold.
Indeed, if $\widetilde{L}^\xi(\Sigma, \vec{k} ; \upsilon) \cap \widetilde{EV}^{<\infty}(\vec{u})\subseteq \F$, then
$(\widetilde{L}^\xi(\Sigma, \vec{k} ; \upsilon) \cap \widetilde{EV}^{<\infty}(\vec{u}))^* \subseteq \F^* = \F$.
Using the canonical representation of every infinite orderly sequence of variable located words
with respect to $\widetilde{L}^\xi(\Sigma, \vec{k} ; \upsilon)$ (Proposition~\ref{prop:canonicalrep}) it is
easy to check that

$(\widetilde{L}^\xi(\Sigma, \vec{k} ; \upsilon))^* \cap \widetilde{EV}^{<\infty}(\vec{u}) = (\widetilde{L}^\xi(\Sigma, \vec{k} ; \upsilon) \cap \widetilde{EV}^{<\infty}(\vec{u}))^*$ .
\newline
Hence, $(\widetilde{L}^\xi(\Sigma, \vec{k} ; \upsilon))^* \cap \widetilde{EV}^{<\infty}(\vec{u})\subseteq \F$.

Since $\xi>\zeta_{\vec{w}}^{\F}$, according to Corollary~\ref{cor:tree}, there exists
$\vec{t}\prec \vec{u}$ such that

$(\widetilde{L}^{\zeta_{\vec{w}}^{\F}}(\Sigma, \vec{k} ; \upsilon))_*\cap \widetilde{EV}^{<\infty}(\vec{t}) \subseteq
(\widetilde{L}^\xi(\Sigma, \vec{k} ; \upsilon))^* \cap \widetilde{EV}^{<\infty}(\vec{u})\subseteq \F$.
\newline
So, $(\widetilde{L}^{\zeta_{\vec{w}}^{\F}}(\Sigma, \vec{k} ; \upsilon))_*\cap \widetilde{EV}^{<\infty}(\vec{t}) \subseteq \F_h$.
This is a contradiction to the relation $(1)$.
Hence, $\widetilde{L}^\xi(\Sigma, \vec{k} ; \upsilon) \cap \widetilde{EV}^{<\infty}(\vec{u})\subseteq \widetilde{L}^{<\infty}(\Sigma, \vec{k} ; \upsilon)\setminus \F$ and  $\F\cap \widetilde{EV}^{<\infty}(\vec{u}) \subseteq (\widetilde{L}^\xi(\Sigma, \vec{k} ; \upsilon))^* \setminus \widetilde{L}^\xi(\Sigma, \vec{k} ; \upsilon)$.

2(iii)
In the cases $\zeta_{\vec{w}}^{\F} = \xi+1$ or $\zeta_{\vec{w}}^{\F} = \xi$, we
use Theorem~\ref{thm:block-Ramsey}.
\end{proof}

The following immediate corollary to Theorem~\ref{block-NashWilliams2} is more useful for applications and can be considered as a strengthened Nash-Williams type partition theorem for variable $\omega$-$\mathbb{Z}^\ast$-located words.

\begin{cor}
\label{cor:tree2}
Let $\F \subseteq \widetilde{L}^{<\infty}(\Sigma, \vec{k} ; \upsilon)$ which is a
tree and let $\vec{w} \in \widetilde{L}^\infty (\Sigma, \vec{k} ; \upsilon)$. Then
\begin{itemize}
\item[(i)] either
there exists $\vec{u}\prec \vec{w}$ such that $\widetilde{EV}^{<\infty}(\vec{u})\subseteq \F,$
\item[(ii)] or for every countable ordinal $\xi > \zeta_{\vec{w}}^{\F}$ there exists
$\vec{u}\prec \vec{w}$, such that for every $\vec{u}_1 \prec \vec{u}$ the unique initial segment of $\vec{u}_1$
which is an element of $\widetilde{L}^\xi(\Sigma, \vec{k} ; \upsilon)$ belongs to $ \widetilde{L}^{<\infty}(\Sigma, \vec{k} ; \upsilon)\setminus \F$.
\end{itemize}
\end{cor}

 Corollary~\ref{cor:tree2} implies the following Nash-Williams type partition theorem for variable $\omega$-$\mathbb{Z}^\ast$-located words.

\begin{thm}
[\textsf{Partition theorem for infinite orderly sequences of variable $\omega$-$\mathbb{Z}^\ast$-located words}]
\label{cor:blockNW}
Let $\Sigma=\{\alpha_n:\;n \in \mathbb{Z}^\ast\}$ be an alphabet, $\upsilon \notin \Sigma$ a variable and $\vec{k}=(k_n)_{n\in\mathbb{Z}^\ast}\subseteq \nat$ such that $(k_n)_{n\in \nat},\;(k_{-n})_{n\in \nat}$ are increasing sequences. If $\U \subseteq \widetilde{L}^{\infty}(\Sigma, \vec{k} ; \upsilon)$ is a pointwise closed family and $\vec{w} \in \widetilde{L}^\infty (\Sigma, \vec{k} ; \upsilon),$ then there exists $\vec{u}\prec \vec{w}$ such that
\begin{center}either  $\widetilde{EV}^{\infty}(\vec{u})\subseteq \U,\;\;$ or $\;\;\widetilde{EV}^{\infty}(\vec{u})\subseteq \widetilde{L}^{\infty}(\Sigma, \vec{k} ; \upsilon) \setminus \U$.\end{center}
\end{thm}

\begin{proof}
Let $\F_{\U} = \{\bw \in \widetilde{L}^{<\infty}(\Sigma, \vec{k} ; \upsilon)$: there exist $\vec{s}\in \U$
such that $\bw\propto \vec{s}\}$.
Since the family $\F_{\U}$ is a tree, we can use Corollary~\ref{cor:tree2}.
So, we have the following two cases:

\noindent { [Case 1]}
There exists $\vec{u}\prec \vec{w}$ such that $\widetilde{EV}^{<\infty}(\vec{u})\subseteq \F_{\U}$.
Then, $\widetilde{EV}^{\infty}(\vec{u})\subseteq \U$.
Indeed, if $\vec{z} = (z_n)_{n\in\nat} \in \widetilde{EV}^{\infty}(\vec{u})$, then
$(z_1,\ldots,z_n)\in \F_{\U}$ for every $n\in\nat$.
Hence, for each $n\in\nat$ there exists $\vec{s}_n\in\U$ such that
$(z_1,\ldots,z_n)\propto \vec{s}_n$.
Since $\U$ is pointwise closed,
we have that $\vec{z}\in\U$ and consequently that $\widetilde{EV}^{\infty}(\vec{u})\subseteq \U$.

\noindent { [Case 2]}
There exists $\vec{u}\prec \vec{w}$ such that for every $\vec{u}_1 \prec \vec{u}$ there exists an initial segment
of $\vec{u}_1$ which belongs to $\widetilde{L}^{<\infty}(\Sigma, \vec{k} ; \upsilon)\setminus \F_{\U}$.
Hence, $\widetilde{EV}^{\infty}(\vec{u})\subseteq \widetilde{L}^{\infty}(\Sigma, \vec{k} ; \upsilon) \setminus \U$.
\end{proof}

Using Corollary~\ref{cor:tree2}, we can get the following strengthened Nash-Williams type partition theorem for variable $\omega$-located words, which extends the analogous to Theorem~\ref{cor:blockNW} result for variable $\omega$-located words that follows from Carlson's Theorem 15 in [C], an Ellentuck-type theorem.

\begin{cor}
\label{cor:tree3}
Let $\F \subseteq L^{<\infty}(\Sigma, \vec{k} ; \upsilon)$ which is a
tree (see Definition~\ref{def:Fthin} (iii)) and let $\vec{w} \in L^\infty (\Sigma, \vec{k} ; \upsilon)$. Then
\begin{itemize}
\item[(i)] either
there exists $\vec{u}\prec \vec{w}$ such that $EV^{<\infty}(\vec{u})\subseteq \F,$
\item[(ii)] or there exists $\xi_0<\omega_1$ such that for every countable ordinal $\xi > \xi_0$ there exists
$\vec{u}\prec \vec{w}$, such that for every $\vec{u}_1 \prec \vec{u}$ the unique initial segment of $\vec{u}_1$
which is an element of $L^\xi(\Sigma, \vec{k} ; \upsilon)$ belongs to $ L^{<\infty}(\Sigma, \vec{k} ; \upsilon)\setminus \F$.
\end{itemize}
\end{cor}

\section{Partition theorems for sequences of rational numbers}

T. Budak, N. I\c{s}ik and J. Pym in \cite{BIP} (Theorem $4.2$) introduced a representation of rational numbers with specific properties, according to which a non-zero rational number can be identified with a $\omega$-$\zat^\ast$-located word over the alphabet $\Sigma=\{\alpha_n: \;n\in \zat^\ast\},$ where $\alpha_{-n}=\alpha_n=n$ for $n\in \nat,$ dominated by $(k_n)_{n\in \zat^\ast},$ where $k_{-n}=k_n=n$ for $n\in \nat.$ Hence, all the results concerning $\omega$-$\zat^\ast$-located words over an alphabet $\Sigma=\{\alpha_n: \;n\in \zat^\ast\}$ dominated by a $(k_n)_{n\in \zat^\ast}\subseteq \nat,$ proved in the previous sections, can be formulating to statements concerning rational numbers.

In this section we present a strengthened van der Waerden theorem for the set of rational numbers in Theorem~\ref{thm:block-Ramsey05} using Theorem~\ref{thm:block-Ramsey02}, an extended to every countable order Ramsey-type partition theorem for the set of rational numbers in Theorem~\ref{thm:block-Ramsey005} as a consequence of Theorem~\ref{thm:block-Ramsey}, and a Nash-Williams type partition theorem for infinite orderly sequences of rational numbers in Theorem~\ref{thm:block-Ramsey0005} as a consequence of Theorem~\ref{cor:blockNW}.

Analytically, according to \cite{BIP}, every rational number $q$  has a unique expression in the form $$\sum^{\infty}_{s=1}q_{-s}\frac{(-1)^{s}}{(s+1)!}\;+\;\sum^{\infty}_{r=1}q_{r}(-1)^{r+1}r! $$ where $(q_n)_{n \in \mathbb{Z}^\ast}\subseteq \nat\cup\{0\}$ with $\;0\leq q_{-s}\leq s$ for every $s>0$, $ 0\leq q_r\leq r$ for every $r> 0$ and $q_{-s}=q_r=0$ for all but finite many $r,s$. So, for a non-zero rational number $q$, there exist unique $l\in \nat,\;\{t_1<\ldots<t_l\}=dom(q)\in [\zat^\ast]^{<\omega}_{>0}$ and $\{q_{t_1},\ldots,q_{t_l}\}\subseteq \nat$ with $1\leq q_{t_i}\leq -t_i$ if $t_i<0$ and $1\leq q_{t_i}\leq t_{i}$ if $t_i> 0$ for every $1\leq i\leq l,$ such that defining $dom^-(q)=\{t\in dom(q):\;t<0\}$ and $dom^+(q)=\{t\in dom(q):\;t>0\}$ to have $$ q=\sum_{t\in dom^-(q)}q_t\frac{(-1)^{-t}}{(-t+1)!}\;+\;\sum_{t\in dom^+(q)}q_t(-1)^{t+1}t!\;\;\;(\text{we set}\;\;\sum_{t\in \emptyset}=0).$$ Observe that $$e^{-1}-1=-\sum^{\infty}_{t=1}\frac{2t-1}{(2t)!}< \sum_{t\in dom^-(q)}q_t\frac{(-1)^{-t}}{(-t+1)!} <\sum^{\infty}_{t=1}\frac{2t}{(2t+1)!}=e^{-1}$$ and that $$ \sum_{t\in dom^+(q)}q_t(-1)^t(t+1)! \in \zat^\ast\;\;\text{if}\;\;dom^+(q)\neq\emptyset.$$ Let $\alpha_{-n}=\alpha_n=n$ and $k_{-n}=k_n=n$ for $n\in \nat.$ We set $\Sigma=\{\alpha_n:\;n\in \mathbb{Z}^\ast\}$ and $\vec{k}=(k_n)_{n\in\mathbb{Z}^\ast}.$ For $\upsilon=0$ we define the function $g:\widetilde{L}(\Sigma\cup\{0\},\vec{k})\rightarrow\mathbb{Q},$ which sends a word $w=q_{t_1}\ldots q_{t_l}\in \widetilde{L}(\Sigma\cup\{0\},\vec{k})$ to the rational number \begin{center} $\;g(w)=\sum_{t\in dom^-(w)}q_t\frac{(-1)^{-t}}{(-t+1)!}\;+\;\sum_{t\in dom^+(w)}q_t(-1)^{t+1}t!.$  \end{center}
\noindent It is easy to see that the restriction of $g$ to the set of the constant words $\widetilde{L}(\Sigma,\vec{k})$ is one-to-one and onto $\mathbb{Q}\setminus\{0\}$ and that $g(w_1\star w_2)=g(w_1)+g(w_2)$ for every $w_1<_{\textsl{R}_1}w_2\in \widetilde{L}(\Sigma\cup\{0\},\vec{k}).$ Also, observe that, via the function $g,$ each variable word $w=q_{t_1}\ldots q_{t_l}\in \widetilde{L}(\Sigma,\vec{k};0)$ corresponds to a function $q$ which sends every $(i,j)\in \nat\times\nat$ with $1\leq i\leq -\max dom^-(w),$ $1\leq j\leq \min dom^+(w),$ to $$q(i,j)=g(T_{(j,i)}(w))= \sum_{t\in C^-}q_t\frac{(-1)^{-t}}{(-t+1)!} + i\sum_{t\in V^-}\frac{(-1)^{-t}}{(-t+1)!} +  \sum_{t\in C^+}q_t(-1)^{t+1}t!+j\sum_{t\in V^+}(-1)^{t+1}t!,$$ where $C^-=\{t\in dom^-(w):\;q_t\in \Sigma\},$ $V^-=\{t\in dom^-(w):\;q_t=0\}$ and $C^+=\{t\in dom^+(w):\;q_t\in \Sigma\},$ $V^+=\{t\in dom^+(w):\;q_t=0\}.$

 For two non-zero rational numbers $q_1,q_2\in g(\widetilde{L}_0(\Sigma,\vec{k}))$ we set \begin{center} $q_1\prec q_2 \;\Longleftrightarrow\; g^{-1}(q_1)<_{\textsl{R}_1}g^{-1}(q_2).$ \end{center}

\begin{note}Let $(X,+)$ arbitrary semigroup. For $(x_n)_{n\in \nat},\;(z_n)_{n\in \nat}\subseteq X$ we set
\begin{center} $FS[(x_n)_{n\in \nat}]=\{x_{n_1}\;+\;\ldots\;+\;x_{n_l}:\;n_1<\ldots<n_l\in \nat\}.$\end{center}
\end{note}

\begin{thm}\label{thm:block-Ramsey05}
Let $\mathbb{Q}=Q_1\cup\ldots \cup Q_r$ for $r\in \nat$. Then, there exist $1\leq i_{0}\leq r$ and for every $n\in \nat$ a function $q_n:\{1,\ldots,n\}\times\{1,\ldots,n\}\cup\{(0,0)\}\rightarrow \mathbb{Q}$ with
$$q_n(i,j)=\sum_{t\in C_n^-}q^n_t\frac{(-1)^{-t}}{(-t+1)!} + i\sum_{t\in V_n^-}\frac{(-1)^{-t}}{(-t+1)!} +  \sum_{t\in C_n^+}q^n_t(-1)^{t+1}t!+j\sum_{t\in V_n^+}(-1)^{t+1}t!,$$
\noindent where
    $C_n^-,V_n^-\in [\zat^-]^{<\omega}_{>0},\;C_n^+,V_n^+\in [\nat]^{<\omega}_{>0}$ with $C_n^-\cap V_n^-=\emptyset=C_n^+\cap V_n^+,$ $q^n_t\in \nat$ with $1\leq q_t^n\leq -t$ for $t\in C_n^-,$ $1\leq q_t^n\leq t$ for $t\in C_n^+,$ which satisfy
 $q_n(i_n,j_n)\prec q_{n+1}(i_{n+1},j_{n+1})$ for every $n\in \nat,$ and
$$FS\big[\big(q_n(i_{n},j_{n})\big)_{n\in \nat}\big]\subseteq Q_{i_0}$$ for all $((i_n,j_n))_{n\in \nat}\subseteq \nat\times\nat\cup\{(0,0)\}$ with $0\leq i_n,j_n\leq n$ for every $n\in \nat.$
\end{thm}

\begin{proof}[Proof] Let the function $g:\widetilde{L}(\Sigma,\vec{k};0)\rightarrow\mathbb{Q}$ defined above.
 According to Theorem~\ref{thm:block-Ramsey02} there exists a sequence $(w_{n})_{n\in \nat}\in \widetilde{L}^\infty(\Sigma,\vec{k};0)$
and $1\leq i_{0} \leq r,$ such that
$T_{(i_1,j_1)}(w_{n_1})\star \ldots \star T_{(i_\lambda,j_\lambda)}(w_{n_\lambda})\in g^{-1}(Q_{i_{0}})$,
for every $\lambda\in\nat$, $n_1<\ldots<n_\lambda\in\nat,$ $(i_l,j_l)\in \nat\times\nat\cup\{(0,0)\}$ such that $0\leq i_l,j_l \leq n_l,$ for every $1\leq l \leq \lambda$ and $(0,0)\in \{(i_1,j_1),\ldots,(i_\lambda,j_\lambda)\}$.
Let $w_n=w_{m_1^n}^n\ldots w_{m_{l_n}^n}^n$ for every $n\in \nat.$ Set $q_n(i,j)=g(w_{3n-2}\star T_{(j,i)}(w_{3n-1})\star T_{(1,1)}(w_{3n}))$ for every $n\in \nat$ and $(i,j)\in \nat\times\nat\cup\{(0,0)\}$ with $0\leq i,j\leq n.$ The functions $q_n$ satisfy the required properties.
\end{proof}

Theorem~\ref{thm:block-Ramsey05} has the following finitistic form.

\begin{note} For every $n\in \nat$ we set $\mathbb{Q}_n=\{q\in \mathbb{Q}:\;|dom(q)|=n\}.$\end{note}

\begin{cor}\label{thm:RHJ}
Let $r,m\in \nat$ and $n_1,\ldots,n_m\in \nat.$ There exists $n_0\equiv n_0(r,m,n_1,\ldots,n_m)\in \nat$ such that if $\mathbb{Q}_{n_0}=Q_1\cup\ldots\cup Q_r,$ then there exist $1\leq i_0\leq r$ and, for each $1\leq s\leq m,$ a function
 $q_s:\{1,\ldots,n_s\}\times\{1,\ldots,n_s\}\rightarrow \mathbb{Q}$ with
$$q_s(i,j)=\sum_{t\in C_s^-}q^s_t\frac{(-1)^{-t}}{(-t+1)!} + i\sum_{t\in V_s^-}\frac{(-1)^{-t}}{(-t+1)!} +  \sum_{t\in C_s^+}q^s_t(-1)^{t+1}t!+j\sum_{t\in V_s^+}(-1)^{t+1}t!,$$
\noindent where
    $C_s^-\in [\zat^-]^{<\omega},V_s^-\in [\zat^-]^{<\omega}_{>0},\;C_s^+\in [\nat]^{<\omega},V_s^+\in [\nat]^{<\omega}_{>0}$ with $C_s^-\cap V_s^-=\emptyset=C_s^+\cap V_s^+,$ $q^s_t\in \nat$ with $1\leq q_t^s\leq -t$ for $t\in C_s^-,$ $1\leq q_t^s\leq t$ for $t\in C_s^+,$ which satisfy
 $q_1(i_1,j_1)\prec\ldots\prec q_{m}(i_{m},j_{m}),$ and
$$ q_1(i_1,j_1)+\ldots+ q_{m}(i_{m},j_{m})\in Q_{i_0}$$ for all $(i_s,j_s)\in \nat\times\nat$ with $1\leq i_s,j_s\leq n_s,\;1\leq s\leq m.$
\end{cor}

Combining Theorem~\ref{thm:block-Ramsey} with the representation of rational numbers via the function $g,$  analogously to Theorem~\ref{thm:block-Ramsey05} we have the following Ramsey-type partition theorem for every countable order $\xi$ for the set of rational numbers. The case $\xi=1$ corresponds to Theorem~\ref{thm:block-Ramsey05}.

\begin{note} For an arbitrary semigroup $(X,+)$ and a sequence $(x_n)_{n\in \nat}\subseteq X,$ for \\  $y_1=x_{n_1}+\ldots+x_{n_l},$ $y_2=x_{m_1}+\ldots+x_{m_\nu}\in FS[(x_n)_{n\in \nat}]$ we write $y_1<y_2$ if $n_l<m_1,$\\ and

$\big[FS[(x_n)_{n\in \nat}]\big]^{<\infty}_{>0}=\{(y_1,\ldots,y_m):\;m\in \nat, \;y_1<\ldots<y_m \in FS[(x_n)_{n\in \nat}]\}.$

\noindent For every countable ordinal $\xi\geq 1$ and every $n\in \nat$ we set

  $\mathbb{Q}^{<\infty}=\{(q_1,\ldots,q_l):\;l\in \nat,\;q_1\prec\ldots \prec q_l\in \mathbb{Q}\setminus\{0\}\}\cup\{\emptyset\},$ and

  $\mathbb{Q}^{\xi}=\{(q_1,\ldots,q_l)\in \mathbb{Q}^{<\infty}:\;\{\min dom^+(q_1),\ldots,\min dom^+(q_l)\}\in \mathcal{A}_{\xi}\}.$ \end{note}

\begin{thm}\label{thm:block-Ramsey005} Let $\xi\geq 1$ be a countable ordinal and a family $G\subseteq \mathbb{Q}^{<\infty}.$ Then, for each $n\in \nat$ there exists a function $q_n:\{1,\ldots,n\}\times\{1,\ldots,n\}\cup\{(0,0)\}\rightarrow \mathbb{Q}$ with
$$q_n(i,j)=\sum_{t\in C_n^-}q^n_t\frac{(-1)^{-t}}{(-t+1)!} + i\sum_{t\in V_n^-}\frac{(-1)^{-t}}{(-t+1)!} +  \sum_{t\in C_n^+}q^n_t(-1)^{t+1}t!+j\sum_{t\in V_n^+}(-1)^{t+1}t!,$$
\noindent where
    $C_n^-,V_n^-\in [\zat^-]^{<\omega}_{>0},\;C_n^+,V_n^+\in [\nat]^{<\omega}_{>0}$ with $C_n^-\cap V_n^-=\emptyset=C_n^+\cap V_n^+,$ $q^n_t\in \nat$ with $1\leq q_t^n\leq -t$ for $t\in C_n^-,$ $1\leq q_t^n\leq t$ for $t\in C_n^+,$ which satisfy
 $q_n(i_n,j_n)\prec q_{n+1}(i_{n+1},j_{n+1})$ for every $n\in \nat,$ and
\begin{center}
either $\mathbb{Q}^{\xi}\cap\big[FS\big[\big(q_n(i_{n},j_{n})\big)_{n\in \nat}\big]\big]^{<\infty}_{>0}\subseteq G,\;\;\;\;$\\ or $\mathbb{Q}^{\xi}\cap\big[FS\big[\big(q_n(i_{n},j_{n})\big)_{n\in \nat}\big]\big]^{<\infty}_{>0}\subseteq \mathbb{Q}^{<\infty}\setminus G$
\end{center}
for all $((i_n,j_n))_{n\in \nat}\subseteq \nat\times\nat\cup\{(0,0)\}$ with $0\leq i_n,j_n\leq n$ for every $n\in \nat.$
\end{thm}

Theorem~\ref{thm:block-Ramsey005} has the following finitistic form.

\begin{note} For every countable ordinal $\xi\geq 1$ and every $n\in \nat$ we set

  $\mathbb{Q}_n^{\xi}=\{(q_1,\ldots,q_l)\in \mathbb{Q}^{\xi}:\; |dom(q_1)|+\ldots+|dom(q_l)|=n\}.$\end{note}

\noindent If $\Sigma=\{\alpha_n:\;n\in \mathbb{Z}^\ast\}$ and $\vec{k}=(k_n)_{n\in\mathbb{Z}^\ast},$ where $\alpha_{-n}=\alpha_n=n$ and $k_{-n}=k_n=n$ for $n\in \nat,$ then we define the function $\widetilde{g}:\widetilde{L}^{<\infty}(\Sigma,\vec{k})\rightarrow\mathbb{Q}^{<\infty},$ with \begin{center} $\;\widetilde{g}((w_1,\ldots,w_l))=(g(w_1),\ldots,g(w_l))\in \mathbb{Q}^{<\infty}.$\end{center}

\begin{cor}\label{thm:xiRHJ}Let $\xi \geq 1$ be a countable ordinal, $\Sigma=\{\alpha_n:\;n\in \zat^\ast\}$ with $\alpha_{-n}=\alpha_n=n$ for $n\in \nat,$ $\upsilon\notin \Sigma$ a variable, $r,l\in \nat$ and $\vec{k}=(k_n)_{n\in \zat^\ast}\subseteq \nat$ with $k_{-n}=k_n=n$ for $n\in \nat.$ Then, there exists $n_0\equiv n_0(\xi,r,l)\in \nat$ such that if $\mathbb{Q}^{\xi}_{n_0}=Q_1\cup\ldots\cup Q_r,$ there exists ${\bf{t}}=(t_1,\ldots,t_l)\in \widetilde{L}^{<\infty}(\Sigma,\vec{k};\upsilon)$ such that for some $1\leq i_0\leq r$ to satisfy \begin{center} $\mathbb{Q}^{\xi}_{n_0}\cap \widetilde{g}(\widetilde{E}^{<\infty}({\bf{t}}))\subseteq C_{i_0}.$\end{center}
\end{cor}

\noindent As a corollary of Theorem~\ref{cor:blockNW} we have the following Nash-Williams-type theorem for the set $\mathbb{Q}$ of rational numbers.

\begin{note} For a semigroup $(X,+)$ and $(x_n)_{n\in \nat},\;(z_n)_{n\in \nat}\subseteq X,$ we set

$\big[FS[(x_n)_{n\in \nat}]\big]^{\nat}=\{(y_n)_{n\in \nat} : \;y_n \in FS[(x_n)_{n\in \nat}]$ and $y_n<y_{n+1}$ for every $n \in \nat\}.$\end{note}

\begin{thm}\label{thm:block-Ramsey0005} Let $\U \subseteq \mathbb{Q}^{\nat}=\{(q_n)_{n\in\nat} :\; q_n\in \mathbb{Q}\}$ be closed in the product topology considering $\mathbb{Q}$ with the discrete topology. Then, for each $n\in \nat$ there exists a function $q_n:\{1,\ldots,n\}\times\{1,\ldots,n\}\cup\{(0,0)\}\rightarrow \mathbb{Q}$  with
$$q_n(i,j)=\sum_{t\in C_n^-}q^n_t\frac{(-1)^{-t}}{(-t+1)!} + i\sum_{t\in V_n^-}\frac{(-1)^{-t}}{(-t+1)!} +  \sum_{t\in C_n^+}q^n_t(-1)^{t+1}t!+j\sum_{t\in V_n^+}(-1)^{t+1}t!,$$
\noindent where
    $C_n^-,V_n^-\in [\zat^-]^{<\omega}_{>0},\;C_n^+,V_n^+\in [\nat]^{<\omega}_{>0}$ with $C_n^-\cap V_n^-=\emptyset=C_n^+\cap V_n^+,$ $q^n_t\in \nat,$ with $1\leq q_t^n\leq -t$ for $t\in C_n^-,$ $1\leq q_t^n\leq t$ for $t\in C_n^+,$ which satisfy $q_n(i_n,j_n)\prec q_{n+1}(i_{n+1},j_{n+1})$ for every $n\in \nat,$ and
$$ \text{either}\;\;\big[FS[(q_n(i_{n},j_{n}))_{n\in \nat}]\big]^\nat\subseteq \U,\;\;\text{or}\;\;\big[FS[(q_n(i_{n},j_{n}))_{n\in \nat}]\big]^\nat\subseteq \mathbb{Q}^{\nat}\setminus \U$$ for all $(i_n,j_n)_{n\in \nat}\subseteq \nat\times\nat\cup\{(0,0)\}$ with $0\leq i_n,j_n\leq n$ for every $n\in \nat.$
\end{thm}

\begin{proof}[Proof]
Let the alphabet $\Sigma=\{\alpha_n:n\in\mathbb{Z}^\ast\},$ $\vec{k}=(k_n)_{n\in \zat^\ast},$ where $\alpha_{-n}=\alpha_n=n$ and $k_{-n}=k_n=n$ for every $n\in \nat,$ and $\upsilon=0$.  We set $\hat{g} : \widetilde{L}^\infty (\Sigma, \vec{k} ; 0)\rightarrow \mathbb{Q}^{\nat}$ with
$\hat{g}( (w_n)_{n\in\nat} )=(g(w_n))_{n\in\nat}$.
The family $\hat{g}^{-1}(\U)\subseteq \widetilde{L}^\infty(\Sigma, \vec{k} ; 0)$ is  pointwise closed, since the function $\hat{g}$ is continuous. So, according to Theorem~\ref{cor:blockNW}, there exists $\vec{w}=(w_n)_{n\in\nat}\in \widetilde{L}^\infty (\Sigma, \vec{k} ; 0)$ such that \begin{center} either $\widetilde{EV}^{\infty}(\vec{w})\subseteq \hat{g}^{-1}(\U)$, or $\widetilde{EV}^{\infty}(\vec{w})\subseteq \widetilde{L}^{\infty}(\Sigma, \vec{k} ; 0) \setminus \hat{g}^{-1}(\U)$.
\end{center}
 From this point on, the proof is analogous to the one of Theorem~\ref{thm:block-Ramsey05}.
\end{proof}

\section{Partition theorems for sequences in an arbitrary semigroup}

In this section we apply the results of sections 1,2 and 3 to an arbitrary semigroup (commutative or non-commutative). So, we get a strong simultaneous extension  of van der Waerden's theorem and Hindman's partition theorem for an arbitrary semigroup (in Theorem~\ref{thm:block-Ramsey06} for a non-commutative and in Theorem~\ref{thm: 011} for a commutative semigroup) extending Theorems 14.14 and 14.15 in [HS], an extended to every countable order Ramsey-type partition theorem for an arbitrary semigroup (Theorem~\ref{thm: xilocated}) and a partition theorem for the infinite sequences in an arbitrary semigroup (in Theorem~\ref{cor:ncentralNW} for a non-commutative and in Theorem~\ref{thm: 11011} for a commutative semigroup), applying the partition Theorems~\ref{thm:block-Ramsey02}, ~\ref{thm:block-Ramsey} and ~\ref{cor:blockNW} for $\omega$-$\zat^\ast$-located words respectively.

\begin{note}
 For $F,G \in [\mathbb{Z}]_{>0}^{<\omega}$, we write $F<G$ if and only if one of the following three conditions is satisfied:

$(1)\;\;F\cap(\mathbb{Z}^{-}\cup\{0\})=\emptyset$ and $\max F< \min G,$

$(2)\;\;F\cap(\nat\cup\{0\})=\emptyset$ and $\min F> \max G,$

$(3)\;\;G= A_1\cup A_2$ where $A_1,A_2\neq \emptyset$ and
$\max A_1< \min F\leq \max F<\min A_2.$
\end{note}

Let $(X,+)$ be a semigroup, $(y_{l,n})_{n \in \mathbb{Z}^\ast}\subseteq X$ for every $l\in \zat$ and $\vec{k}=(k_n)_{n \in \mathbb{Z}^\ast}\subseteq\nat$ where $(k_n)_{n \in \nat}$ and $(k_{-n})_{n \in \nat}$ are increasing sequences. Setting $\Sigma=\{\alpha_n:\;n\in \mathbb{Z}^\ast\}$, where $\alpha_n=n$ for $n\in \zat^\ast$ and $\upsilon=0,$  we define the function \begin{center}$\psi:\widetilde{L}(\Sigma\cup\{0\},\vec{k})\rightarrow X \;\;\text{with}\;\; \psi(w_{n_1}\ldots w_{n_l})=\sum^{l}_{i=1}y_{w_{n_i},n_i}.$
\end{center}
For $(w_n)_{n\in \nat}\in \widetilde{L}^\infty(\Sigma,\vec{k};0)$ we set for every $n\in \nat$ \begin{center} $u_n(i,j)=\psi(T_{(1,1)}(w_{4n-3})\star T_{(i,j)}(w_{4n-2})\star w_{4n-1}\star T_{(1,1)}(w_{4n}))$ \end{center} for every $(i,j)\in \nat\times\nat\cup\{(0,0)\}$ with $0\leq i\leq k_n,\;0\leq j\leq k_{-n}.$

 \noindent In case $X$ is a commutative semigroup,  \begin{center} $u_n(i,j)=\sum_{t\in E_n}y_{w_t,t}+\sum_{t\in H_n}y_{-j,t}+\sum_{t\in L_n}y_{i,t}$ \end{center}
 where $E_n\in [\zat]^{<\omega}_{>0},\;H_n\in [\zat^-]^{<\omega}_{>0},\;L_n\in [\nat]^{<\omega}_{>0}$ with $E_n\cap H_n= \emptyset,\;E_n\cap L_n= \emptyset$ and $E_n<E_{n+1},\;H_n<H_{n+1},\;L_n<L_{n+1},$ and $w_t \in \{-k_{t},\ldots,-1,0\}$ if $t<0,\;w_t \in \{0,1,\ldots,k_t\}$ if $t>0.$

\noindent In case $X$ is a non-commutative semigroup,
 \begin{center} $u_n(i,j)=\alpha_n(j)+\beta_n(i)\;\;$ with $\;\;\;\;\;\;\;\;\;\;\;\;\;\;\;\;\;\;\;\;\;\;\;\; \;\;\;\;\;\;\;\;\;\;\;\;\;\;\;\;\;\;\;\;\;\;\;\;\;\;\;\;\;$\end{center} \begin{center}$\alpha_n(j)=\sum^{m^{1}_n}_{s=1}(\sum_{t \in E_s^n}y_{w_t,t}\;+\;\sum_{t\in H_s^n} y_{-j,t})\;+
\sum_{t\in E_{m^1_n+1}^n}y_{w_t,t}\;\;$  and \\ $\beta_n(i)=\sum^{m^{2}_n}_{s=1}(\sum_{t\in D_s^n}y_{w_t,t}\;+\sum_{t \in L_s^n}y_{i,t})\;+ \sum_{t\in D_{m^2_n+1}^n}y_{w_t,t},\;\;\;\;\;\;\;\;\;\;\;\;$ \end{center} where
$m_n^1,m_n^2\in \nat,$ $E_n=\cup^{m_n^1+1}_{i=1}E^n_i,$ $H_n=\cup^{m_n^1}_{i=1}H^n_i\in [\mathbb{Z}^-]^{<\omega}_{>0},\;D_n=\cup^{m_n^2+1}_{i=1}D^n_i,\;  L_n=\cup^{m_n^2}_{i=1}L^n_i \in [\nat]^{<\omega}_{>0},$ with $E_{n}<E_{n+1},$ $ H_{n}<H_{n+1}, D_{n}<D_{n+1}, L_n<L_{n+1},$ and $w_t \in \{-k_{t},\ldots,-1,0\}$ if $t<0,\;w_t \in \{0,1,\ldots,k_t\}$ if $t>0.$

As a consequence of Theorem~\ref{thm:block-Ramsey02}, via the function $\psi,$ we get the following partition theorems for semigroups.

\begin{note}Let $(X,+)$ arbitrary semigroup. For $(x_n)_{n\in \nat},\;(z_n)_{n\in \nat}\subseteq X$ we set
\begin{center} $FS[(x_n,z_n)]=\{x_{n_l}+\ldots+x_{n_1}+z_{n_1}+\ldots+z_{n_l}:\; n_1<\ldots<n_l\in \nat\}.$\end{center}
\end{note}

\begin{thm}\label{thm:block-Ramsey06}
Let $(X,+)$ be a non-commutative semigroup, $(y_{l,n})_{n \in \mathbb{Z}^\ast}\subseteq X$ for every $l\in \zat$ and $\vec{k}=(k_n)_{n \in \mathbb{Z}^\ast}\subseteq\nat,$ where $(k_n)_{n \in \nat},(k_{-n})_{n \in \nat}$ are increasing sequences. If $X=A_{1}\cup\ldots\cup A_r$ for $r \in \mathbb{N},$ then there exist $1\leq i_{0}\leq r$ and for each $n\in \nat$ a function

\noindent $u_n:\{1,\ldots,k_{n}\}\times\{1,\ldots,k_{-n}\}\cup\{(0,0)\}\rightarrow X\times X$ with $u_n(i,j)=\alpha_n(j)+\beta_n(i)$ and

 \begin{center} $\alpha_n(j)=\sum^{m^{1}_n}_{s=1}(\sum_{t \in E_s^n}y_{w_t,t}\;+\;\sum_{t\in H_s^n} y_{-j,t})\;+
\sum_{t\in E_{m^1_n+1}^n}y_{w_t,t},\;\;\;\;\;\;$\\

$\beta_n(i)=\sum^{m^{2}_n}_{s=1}(\sum_{t\in D_s^n}y_{w_t,t}\;+\sum_{t \in L_s^n}y_{i,t})\;+ \sum_{t\in D_{m^2_n+1}^n}y_{w_t,t}\;\;\;\;\;\;\;\;\;\;\;\;$\end{center}

 \noindent where $(m^{1}_{n})_{n\in \mathbb{N}},(m^{2}_{n})_{n\in \mathbb{N}}\subseteq \nat,$ $(E_{n})_{n\in \mathbb{N}},$ $(H_{n})_{n\in \mathbb{N}}\subseteq [\mathbb{Z}^-]^{<\omega}_{>0},$ $(D_{n})_{n\in \mathbb{N}},$ $(L_{n})_{n\in \mathbb{N}}\subseteq [\nat]^{<\omega}_{>0},$
$E_n=\cup^{m_n^1+1}_{i=1}E^n_i,$ $H_n=\cup^{m_n^1}_{i=1}H^n_i,$ $D_n=\cup^{m_n^2+1}_{i=1}D^n_i,$ $L_n=\cup^{m_n^2}_{i=1}L^n_i$
 with $E_{n}<E_{n+1},$ $H_{n}<H_{n+1},$ $D_{n}<D_{n+1},$ $L_n<L_{n+1}$ for every $n\in\nat,$ and $w_t \in \{-k_{t},\ldots,-1,0\}$ if $t\in \cup_{n\in \nat}E_n,$ $w_t \in \{0,1,\ldots,k_t\}$ if $t\in \cup_{n\in \nat}D_n,$ which satisfy
$$ FS\big[ \big(\alpha_n(j_n),\beta_n(i_n) \big)_{n\in \nat}\big]\subseteq A_{i_{0}}$$ for all $((i_n,j_n))_{n\in \nat}\subseteq \nat\times\nat\cup\{(0,0)\}$ with $0\leq i_n\leq k_n,$ $0\leq j_n\leq k_{-n}$ for every $n\in \nat.$
\end{thm}

\begin{proof}[Proof]
Let $\Sigma=\{\alpha_n:\;n\in \mathbb{Z}^\ast\}$, where $\alpha_n=n$ for $n\in \zat^\ast,$ $\upsilon=0$ and $\psi:\widetilde{L}(\Sigma,\vec{k};0)\rightarrow X$ the function defined above. Then $\widetilde{L}(\Sigma,\vec{k};0)=\psi^{-1}(A_{1})\cup\ldots \cup \psi^{-1}(A_{r})$. According to Theorem~\ref{thm:block-Ramsey02}, there exists $1\leq i_{0} \leq r$ and a sequence $(w_{n})_{n\in \nat}\subseteq \widetilde{L}^\infty(\Sigma,\vec{k};\upsilon)$ such that
$T_{(p_1,q_1)}(w_{n_1})\star \ldots \star T_{(p_\lambda,q_\lambda)}(w_{n_\lambda})\in \psi^{-1}(A_{i_{0}})$,
for every $\lambda\in\nat$, $n_1<\ldots<n_\lambda\in\nat$, $(p_i,q_i)\in \nat\times\nat\cup\{(0,0)\}$ with $0\leq p_i \leq k_{n_i},\; 0\leq q_i\leq k_{-n_i}$ for every $1\leq i \leq \lambda$ and $(0,0)\in \{(p_1,q_1),\ldots,(p_\lambda,q_\lambda)\}.$

We set $u_n(i,j)=\psi(T_{(1,1)}(w_{4n-3})\star T_{(i,j)}(w_{4n-2})\star w_{4n-1}\star T_{(1,1)}(w_{4n}))$ for every $n\in \nat$  and $(i,j)\in \nat\times\nat\cup\{(0,0)\}$ with $0\leq i\leq k_n,\;0\leq j\leq k_{-n}.$
\end{proof}

In case of a commutative semigroup $(X,+)$, Theorem~\ref{thm:block-Ramsey06} has the following simplified statement.
\begin{thm}
\label{thm: 011}
Let $(X,+)$ be a commutative semigroup, $(y_{l,n})_{n \in \mathbb{Z}^\ast}\subseteq X$ for every $l\in \zat$ and $\vec{k}=(k_n)_{n \in \mathbb{Z}^\ast}\subseteq\nat,$ where $(k_n)_{n \in \nat}$ and $(k_{-n})_{n \in \nat}$ are increasing sequences. If $X=A_{1}\cup\ldots\cup A_r$ for $r \in \mathbb{N}$, then there exist $1\leq i_{0}\leq r$ and for each $n\in \nat$ a function $u_n:\{1,\ldots,k_{n}\}\times\{1,\ldots,k_{-n}\}\cup\{(0,0)\}\rightarrow X$ with
 $$u_n(i,j)=\sum_{t\in E_n}y_{w_t,t}+\sum_{t\in H_n}y_{-j,t}+\sum_{t\in L_n}y_{i,t}$$
\noindent where $(E_{n})_{n\in \mathbb{N}}\subseteq [\mathbb{Z}]^{<\omega}_{>0},\;(H_{n})_{n\in \mathbb{N}}\subseteq [\mathbb{Z}^-]^{<\omega}_{>0},\;(L_{n})_{n\in \mathbb{N}}\subseteq [\nat]^{<\omega}_{>0},$ $E_n\cap H_n=\emptyset,$ $\;E_n\cap L_n=\emptyset,$ with $E_n<E_{n+1},\;H_n<H_{n+1},\;L_n<L_{n+1}$ for every $n\in\nat,$
and $w_t \in \{-k_{t},\ldots,-1,0\}$ if $t<0,\;t\in\cup_{n\in \nat}E_n,$ $w_t \in \{0,1,\ldots,k_t\}$ if $t>0,\;t\in\cup_{n\in \nat}E_n,$ such that
$$ FS\big[ \big(u_n(i_n,j_n)\big)_{n\in \nat}\big]\subseteq A_{i_{0}}$$ for all $((i_n,j_n))_{n\in \nat}\subseteq \nat\times\nat\cup\{(0,0)\}$ with $0\leq i_n\leq k_n,$ $0\leq j_n\leq k_{-n}$ for every $n\in \nat.$
\end{thm}

A particular case of Theorem~\ref{thm: 011} gives the following corollary.

\begin{note} For a semigroup $(X,+)$ and $(x_n)_{n\in \mathbb{Z}}\subseteq X$ we set
\begin{center}$FS[(x_{-n})_{n\in \nat}]=\{x_{-n_l}\;+\;\ldots\;+\;x_{-n_1}:\;n_1<\ldots<n_l\in \nat\}$ and\\
$FS[(x_n)_{n\in \mathbb{Z}}]=\{x_{n_1}\;+\;\ldots\;+\;x_{n_l}:\;n_1<\ldots<n_l\in \mathbb{Z}\}.\;\;\;\;\;\;\;\;\;\;\;\;$\end{center}
\end{note}

\begin{cor}\label{cor: 05} Let $(X,+)$ be a commutative semigroup and $(x_n)_{n\in \mathbb{Z}}\subseteq X.$ If $X=A_1\cup\ldots \cup A_r$ for $r \in \nat$, then there exists $1\leq i_{0}\leq r$,  $(a_n)_{n\in \nat}\subseteq FS[(x_n)_{n\in \mathbb{Z}}],$ $(b_n)_{n\in \nat}\subseteq FS[(x_n)_{n\in \nat}]$ and $(c_n)_{n\in \nat}\subseteq FS[(x_{-n})_{n\in \nat}]$ such that
\begin{center}
$ FS\big[ \big(a_n+i_n b_n+j_n c_n)_{n\in \nat}\big]\subseteq A_{i_{0}}$ \\ for all $((i_n,j_n))_{n\in \nat}\subseteq \nat\times\nat\cup\{(0,0)\}$ with $0\leq i_n,j_n\leq n$ for every $n\in \nat.$
\end{center}
\end{cor}

\begin{proof}[Proof] Set $y_{s,n}=|s|\cdot x_n,$ for every $s\in  \mathbb{Z},\;n\in \mathbb{Z}^\ast$ and $k_n=k_{-n}=n$ for every $n \in \nat$ and apply Theorem~\ref{thm: 011}.
\end{proof}

 Applying Theorem~\ref{thm:block-Ramsey}, via the function $\psi,$ we can derive the following Ramsey type partition theorems for each countable order $\xi$ for an arbitrary semigroup $(X,+),$ extending Theorems~\ref{thm:block-Ramsey06} and ~\ref{thm: 011}, which corresponds to the case $\xi=1.$

\begin{thm}\label{thm: xilocated} Let $(X,+)$ be an arbitrary semigroup, $\xi\geq 1$ a countable ordinal, $\Sigma=\{\alpha_n:\;n\in \zat^\ast\}$ with $\alpha_n=n$ for every $n\in \zat^\ast,$ $\vec{k}=(k_n)_{n\in\mathbb{Z}^\ast}\subseteq \nat$ such that
$(k_n)_{n\in\nat}$, $(k_{-n})_{n\in\nat}$ are increasing sequences and $(y_{l,n})_{n \in \zat^\ast}\subseteq X$ for every $l \in \mathbb{Z}$. For every family $G\subseteq [X]^{<\omega}_{>0}$ of finite subsets of $X$, there exists $\vec{w}=(w_n)_{n\in\nat} \in \widetilde{L}^{\infty}(\Sigma, \vec{k} ; 0)$  such that
\begin{center} either
$\{(\psi(z_1),\ldots,\psi(z_n))\in [X]^{<\omega}_{>0} :\; (z_1,\ldots,z_n)\in \widetilde{L}^{\xi}(\Sigma, \vec{k};0)\cap \widetilde{EV}^{<\infty}(\vec{w})\}\subseteq G\;\;\;\;\;\;\;\;$\\ or $\{(\psi(z_1),\ldots,\psi(z_n))\in [X]^{<\omega}_{>0} :\; (z_1,\ldots,z_n)\in \widetilde{L}^{\xi}(\Sigma, \vec{k};0)\cap \widetilde{EV}^{<\infty}(\vec{w})\}\subseteq [X]^{<\omega}_{>0}\setminus G.$
\end{center}
\end{thm}

  The particular case of Theorem~\ref{thm: xilocated} for $\xi=m$ a finite ordinal has the following form in case of a non-commutative semigroup.

 \begin{note} Let $(X,+)$ be an arbitrary semigroup, $(x_n)_{n\in \nat},(z_n)_{n\in \nat}\subseteq X$ and $m\in \nat.$\\  For $y_1=x_{n_l}+\ldots+x_{n_1}+z_{n_1}+\ldots+z_{n_l},\;y_2=x_{m_\nu}+\ldots+x_{m_1}+z_{m_1}+\ldots+z_{m_\nu}\in FS[(x_n,z_n)_{n\in \nat}]$ we write $y_1<y_2$ if $n_l<m_1,$\\
 $\big[FS[(x_n)_{n\in \nat}]\big]^m=\{(y_1,\ldots,y_m): \;y_1<\ldots<y_m \in FS[(x_n)_{n\in \nat}]\},$\\
 $\big[FS[(x_n,z_n)_{n\in \nat}]\big]^m=\{(y_1,\ldots,y_m): \;y_1<\ldots<y_m \in FS[(x_n,z_n)_{n\in \nat}]\},$ and\\ $X^{m}$ is the set of all the subsets of $X$ with exactly $m$ elements. \end{note}

\begin{cor}\label{thm:block-Ramsey006}
  Let $(X,+)$ be a non-commutative semigroup, $m\in \nat,$ $(y_{l,n})_{n \in \mathbb{Z}^\ast}\subseteq X$ for every $l\in \zat$ and $\vec{k}=(k_n)_{n \in \mathbb{Z}^\ast}\subseteq\nat$ where $(k_n)_{n \in \nat},\;(k_{-n})_{n \in \nat}$ are increasing sequences. If $X^m=A_{1}\cup\ldots\cup A_r$ for $r \in \mathbb{N},$ then  there exist $1\leq i_{0}\leq r$ and for each $n\in \nat$ a function

\noindent $u_n:\{1,\ldots,k_{n}\}\times\{1,\ldots,k_{-n}\}\cup\{(0,0)\}\rightarrow X\times X$ with $u_n(i,j)=\alpha_n(j)+\beta_n(i)$ and

 \begin{center} $\alpha_n(j)=\sum^{m^{1}_n}_{s=1}(\sum_{t \in E_s^n}y_{w_t,t}\;+\;\sum_{t\in H_s^n} y_{-j,t})\;+
\sum_{t\in E_{m^1_n+1}^n}y_{w_t,t},\;\;\;\;\;\;$\\

$\beta_n(i)=\sum^{m^{2}_n}_{s=1}(\sum_{t\in D_s^n}y_{w_t,t}\;+\sum_{t \in L_s^n}y_{i,t})\;+ \sum_{t\in D_{m^2_n+1}^n}y_{w_t,t}\;\;\;\;\;\;\;\;\;\;\;\;$\end{center}

 \noindent where $(m^{1}_{n})_{n\in \mathbb{N}},(m^{2}_{n})_{n\in \mathbb{N}}\subseteq \nat,$ $(E_{n})_{n\in \mathbb{N}},$ $(H_{n})_{n\in \mathbb{N}}\subseteq [\mathbb{Z}^-]^{<\omega}_{>0},$ $(D_{n})_{n\in \mathbb{N}},$ $(L_{n})_{n\in \mathbb{N}}\subseteq [\nat]^{<\omega}_{>0},$
$E_n=\cup^{m_n^1+1}_{i=1}E^n_i,$ $H_n=\cup^{m_n^1}_{i=1}H^n_i,$ $D_n=\cup^{m_n^2+1}_{i=1}D^n_i,\;  L_n=\cup^{m_n^2}_{i=1}L^n_i$
with $E_{n}<E_{n+1},$ $H_{n}<H_{n+1},$ $D_{n}<D_{n+1},$ $L_n<L_{n+1}$ for every $n\in \nat,$ and $w_t \in \{-k_{t},\ldots,-1,0\}$ if $t\in \cup_{n\in \nat}E_n,$ $w_t \in \{0,1,\ldots,k_t\}$ if $t\in \cup_{n\in \nat}D_n,$ which satisfy
$$ \big[FS\big[ \big(\alpha_n(j_n),\beta_n(i_n)\big)_{n\in \nat}\big]\big]^m\subseteq A_{i_{0}}$$ for all $((i_n,j_n))_{n\in \nat}\subseteq \nat\times\nat\cup\{(0,0)\}$ with $0\leq i_n\leq k_n,\;0\leq j_n\leq k_{-n}$ for every $n\in \nat.$
\end{cor}

Corollary~\ref{thm:block-Ramsey006} has a simplified statement in case $(X,+)$ is commutative.

\begin{cor}
\label{thm: 1011}
Let $(X,+)$ be a commutative semigroup, $m\in \nat,$ $(y_{l,n})_{n \in \mathbb{Z}^\ast}\subseteq X$ for every $l\in \zat$ and $\vec{k}=(k_n)_{n \in \mathbb{Z}^\ast}\subseteq\nat$ where $(k_n)_{n \in \nat}$ and $(k_{-n})_{n \in \nat}$ are increasing sequences. If $X^m=A_{1}\cup\ldots\cup A_r$ for $r \in \mathbb{N},$ then there exist $1\leq i_{0}\leq r$ and for each $n\in \nat$ a function $u_n:\{1,\ldots,k_{n}\}\times\{1,\ldots,k_{-n}\}\cup\{(0,0)\}\rightarrow X$ with
 $$u_n(i,j)=\sum_{t\in E_n}y_{w_t,t}+\sum_{t\in H_n}y_{-j,t}+\sum_{t\in L_n}y_{i,t}$$
\noindent where $(E_{n})_{n\in \mathbb{N}}\subseteq [\mathbb{Z}]^{<\omega}_{>0},\;(H_{n})_{n\in \mathbb{N}}\subseteq [\mathbb{Z}^-]^{<\omega}_{>0},\;(L_{n})_{n\in \mathbb{N}}\subseteq [\nat]^{<\omega}_{>0},$ $E_n\cap H_n=\emptyset,$ $\;E_n\cap L_n=\emptyset,$ with $E_n<E_{n+1},\;H_n<H_{n+1},\;L_n<L_{n+1}$ for every $n\in\nat,$
and $w_t \in \{-k_{t},\ldots,-1,0\}$ if $t<0,\;t\in\cup_{n\in \nat}E_n,$ $w_t \in \{0,1,\ldots,k_t\}$ if $t>0,\;t\in\cup_{n\in \nat}E_n,$ such that
$$ \big[FS\big[ \big(u_n(i_n,j_n)\big)_{n\in \nat}\big]\big]^m\subseteq A_{i_{0}}$$ for all $((i_n,j_n))_{n\in \nat}\subseteq \nat\times\nat\cup\{(0,0)\}$ with $0\leq i_n\leq k_n,$ $0\leq j_n\leq k_{-n}$ for every $n\in \nat.$
\end{cor}

Finally, we state a partition theorem for the infinite sequences of a semigroup.

\begin{note} For a semigroup $(X,+)$ and $(x_n)_{n\in \nat},\;(z_n)_{n\in \nat}\subseteq X,$ we set

\noindent $\big[FS[(x_n,z_n)_{n\in \nat}]\big]^{\nat}=\{(y_n)_{n\in \nat} : \;y_n \in FS[(x_n,z_n)_{n\in \nat}]$ and $y_n<y_{n+1}$ for every $n \in \nat\},$\\
  $X^{\nat}=\{(x_n)_{n\in\nat} :\; x_n\in X\}.$
 \end{note}

 \noindent We endow the set $X$ with the discrete topology and $X^{\nat}$ with the product topology (equivalently by the pointwise convergence topology).

\begin{thm}
\label{cor:ncentralNW}
Let $(X,+)$ be a non-commutative semigroup, $(y_{l,n})_{n \in \mathbb{Z}^\ast}\subseteq X$ for every $l \in \mathbb{Z}$ and $\vec{k}=(k_n)_{n \in \mathbb{Z}^\ast}\subseteq \nat$ where $(k_n)_{n \in \nat}$ and $(k_{-n})_{n \in \nat}$ are increasing sequences. If $\U \subseteq X^{\nat}$ is closed in the product topology, then for each $n\in \nat$ there exists a function

\noindent $u_n:\{1,\ldots,k_{n}\}\times\{1,\ldots,k_{-n}\}\cup\{(0,0)\}\rightarrow X\times X$ with $u_n(i,j)=\alpha_n(j)+\beta_n(i)$ and

 \begin{center} $\alpha_n(j)=\sum^{m^{1}_n}_{s=1}(\sum_{t \in E_s^n}y_{w_t,t}\;+\;\sum_{t\in H_s^n} y_{-j,t})\;+
\sum_{t\in E_{m^1_n+1}^n}y_{w_t,t},\;\;\;\;\;\;$\\

$\beta_n(i)=\sum^{m^{2}_n}_{s=1}(\sum_{t\in D_s^n}y_{w_t,t}\;+\sum_{t \in L_s^n}y_{i,t})\;+ \sum_{t\in D_{m^2_n+1}^n}y_{w_t,t}\;\;\;\;\;\;\;\;\;\;\;\;$\end{center}

 \noindent where $(m^{1}_{n})_{n\in \mathbb{N}},(m^{2}_{n})_{n\in \mathbb{N}}\subseteq \nat,$ $(E_{n})_{n\in \mathbb{N}},$ $(H_{n})_{n\in \mathbb{N}}\subseteq [\mathbb{Z}^-]^{<\omega}_{>0},$ $(D_{n})_{n\in \mathbb{N}},$ $(L_{n})_{n\in \mathbb{N}}\subseteq [\nat]^{<\omega}_{>0},$
$E_n=\cup^{m_n^1+1}_{i=1}E^n_i,$ $H_n=\cup^{m_n^1}_{i=1}H^n_i,$ $D_n=\cup^{m_n^2+1}_{i=1}D^n_i,$ $L_n=\cup^{m_n^2}_{i=1}L^n_i$
with $E_{n}<E_{n+1},$ $H_{n}<H_{n+1},$ $D_{n}<D_{n+1},$ $L_n<L_{n+1}$ for every $n\in \nat,$ and $w_t \in \{-k_{t},\ldots,-1,0\}$ if $t\in \cup_{n\in \nat}E_n,$ $w_t \in \{0,1,\ldots,k_t\}$ if $t\in \cup_{n\in \nat}D_n,$ such that
    $$\text{either}\;\big[FS\big[\big(\alpha_n(j_n),\beta_n(i_n)\big)_{n\in \nat}\big]\big]^\nat\subseteq \U,\;\;  \text{or}\;\;\big[FS\big[\big(\alpha_n(j_n),\beta_n(i_n)\big)_{n\in \nat}\big]\big]^\nat\subseteq X^\nat\setminus \U$$ for all $((i_n,j_n))_{n\in \nat}\subseteq \nat\times\nat\cup\{(0,0)\}$ with $0\leq i_n\leq k_n,$ $0\leq j_n\leq k_{-n}$ for every $n\in \nat.$
\end{thm}

\begin{proof}[Proof]
Let the alphabet $\Sigma=\{\alpha_n:\;n\in \zat^\ast\}$ with $\alpha_n=n$ for every $n\in\zat^\ast,$ $\upsilon=0$ and
the function $\hat{\psi} : \widetilde{L}^\infty (\Sigma, \vec{k} ; 0)\rightarrow X^{\nat}$ with
$\hat{\psi}( (w_n)_{n\in\nat} )=(\psi(w_n))_{n\in\nat}$.
The family $\hat{\psi}^{-1}(\U)\subseteq \widetilde{L}^\infty(\Sigma, \vec{k} ; 0)$ is  pointwise closed, since the function $\hat{\varphi}$ is continuous. So, according to Theorem~\ref{cor:blockNW}, there exists $\vec{w}=(w_n)_{n\in\nat}\in \widetilde{L}^\infty (\Sigma, \vec{k} ; 0)$ such that
\begin{center}
either  $\widetilde{EV}^{\infty}(\vec{w})\subseteq \hat{\psi}^{-1}(\U)$, or $\widetilde{EV}^{\infty}(\vec{w})\subseteq \widetilde{L}^{\infty}(\Sigma, \vec{k} ; 0) \setminus \hat{\psi}^{-1}(\U)$.
\end{center}
 From this point on, the proof is analogous to the one of Theorem~\ref{thm:block-Ramsey06}.
\end{proof}

Theorem~\ref{cor:ncentralNW} has a simplified statement in case $(X,+)$ is commutative.

\begin{thm}
\label{thm: 11011}
Let $(X,+)$ be a commutative semigroup, $(y_{l,n})_{n \in \mathbb{Z}^\ast}\subseteq X$ for every $l\in \zat$ and $\vec{k}=(k_n)_{n \in \mathbb{Z}^\ast}\subseteq\nat$ where $(k_n)_{n \in \nat}$ and $(k_{-n})_{n \in \nat}$ are increasing sequences. If $\U \subseteq X^{\nat}$ is closed in the product topology, then for each $n\in \nat$ there exist a function\\ $u_n:\{1,\ldots,k_{n}\}\times\{1,\ldots,k_{-n}\}\cup\{(0,0)\}\rightarrow X$ with
 $$u_n(i,j)=\sum_{t\in E_n}y_{w_t,t}+\sum_{t\in H_n}y_{-j,t}+\sum_{t\in L_n}y_{i,t}$$
\noindent where $(E_{n})_{n\in \mathbb{N}}\subseteq [\mathbb{Z}]^{<\omega}_{>0},\;(H_{n})_{n\in \mathbb{N}}\subseteq [\mathbb{Z}^-]^{<\omega}_{>0},\;(L_{n})_{n\in \mathbb{N}}\subseteq [\nat]^{<\omega}_{>0},$ $E_n\cap H_n=\emptyset,$ $\;E_n\cap L_n=\emptyset,$ with $E_n<E_{n+1},\;H_n<H_{n+1},\;L_n<L_{n+1}$ for every $n\in\nat,$
and $w_t \in \{-k_{t},\ldots,-1,0\}$ if $t<0,\;t\in\cup_{n\in \nat}E_n,$ $w_t \in \{0,1,\ldots,k_t\}$ if $t>0,\;t\in\cup_{n\in \nat}E_n,$ such that
$$\text{either}\;\big[FS\big[\big(u_n(i_n,j_n)\big)_{n\in \nat}\big]\big]^\nat\subseteq \U,\;\;  \text{or}\;\;\big[FS\big[\big(u_n(i_n,j_n)\big)_{n\in \nat}\big]\big]^\nat\subseteq X^\nat\setminus \U$$ for all $((i_n,j_n))_{n\in \nat}\subseteq \nat\times\nat\cup\{(0,0)\}$ with $0\leq i_n\leq k_n,$ $0\leq j_n\leq k_{-n}$ for every $n\in \nat.$
\end{thm}

{\bf{Acknowledgments.}} We wish to thank Professor S. Negrepontis for helpful discussions and support during the preparation of this paper. The first author acknowledge partial support from the Kapodistrias research grant of Athens University. The second author acknowledge partial support from the National Scholarship Foundation of Greece.

\bigskip
{\footnotesize
\noindent
\newline
Vassiliki Farmaki:
\newline
{\sc Department of Mathematics, Athens University, Panepistemiopolis, 15784 Athens, Greece}
\newline
E-mail address: vfarmaki@math.uoa.gr

\medskip
\noindent
Andreas Koutsogiannis:
\newline
{\sc Department of Mathematics, Athens University, Panepistemiopolis, 15784 Athens, Greece}
\newline
E-mail address: akoutsos@math.uoa.gr
\end{document}